\documentclass[sn-mathphys-num]{sn-jnl}
\usepackage{eurosym}
\usepackage{amssymb}
\usepackage{mathrsfs}
\usepackage{dsfont}
\usepackage{amsfonts}
\usepackage{amstext}
\usepackage{amssymb}
\usepackage{amsmath,mathtools}
\usepackage{graphicx}
\usepackage{hyperref}
\usepackage{url}
\usepackage{subfig}
\usepackage{caption}
\usepackage{bm}
\usepackage{enumerate}
\usepackage{tikz}
\usepackage{enumitem}

\usepackage{amsthm}
\usepackage{multirow}
\usepackage{xcolor}
\usepackage{textcomp}%
\usepackage{manyfoot}%
\usepackage{booktabs}
\usepackage{algorithm}%
\usepackage{algorithmicx}%
\usepackage{algpseudocode}%
\usepackage{listings}
\usepackage[normalem]{ulem}

\usepackage[pagewise]{lineno}

\mathtoolsset{showonlyrefs}
\hypersetup{
	colorlinks   = true,
	citecolor    = blue,
	linkcolor    = blue,
	urlcolor     = blue
}
\allowdisplaybreaks
\newtheorem{theorem}{Theorem}[section]
\newtheorem{proposition}[theorem]{Proposition}
\newtheorem{lemma}[theorem]{Lemma}
\newtheorem{corollary}[theorem]{Corollary}
\newtheorem{definition}[theorem]{Definition}
\newtheorem{example}[theorem]{Example}

\newtheorem{remark}[theorem]{Remark}

\renewcommand{\theequation}{\thesection.\arabic{equation}}
\newenvironment{notation}{\smallskip{\sc Notation.}\rm}{\smallskip}

\DeclareMathOperator*{\osc}{Osc}
\DeclareMathOperator*{\diam}{diam}

\DeclareMathOperator*{\supp}{supp}

\numberwithin{equation}{section}
\newcounter{counterConstant}

\newcommand*{\dif}{\mathop{}\!\mathrm{d}}
\input{tcilatex}

\def\qed{\hfill$\square$\par}

\def\func#1{\mathop{\mathrm{#1}}\nolimits}

\def\Xint#1{\mathchoice
{\XXint\displaystyle\textstyle{#1}}%
{\XXint\textstyle\scriptstyle{#1}}%
{\XXint\scriptstyle\scriptscriptstyle{#1}}%
{\XXint\scriptscriptstyle\scriptscriptstyle{#1}}%
\!\int}
\def\XXint#1#2#3{{\setbox0=\hbox{$#1{#2#3}{\int}$ }
\vcenter{\hbox{$#2#3$ }}\kern-.6\wd0}}

\def\oint{\Xint-}

\def\fint{\Xint-}

\def\FRAME#1#2#3#4#5#6#7#8
{
 \begin{figure}[ptbh]
 \begin{center}
 \includegraphics[height=#3]{#7}
 \caption{#5}
 \label{#6}
 \end{center}
 \end{figure}
}

\def\enddoc{\end{document}}

\begin{document}
\title[Article Title]{Besov-Lipschitz norm and $p$-energy measure on scale-irregular Vicsek sets}
\date{\today}

\author[1]{\fnm{Aobo} \sur{Chen}}\email{cab21@mails.tsinghua.edu.cn}

\author[2]{\fnm{Jin} \sur{Gao}}\email{gaojin@hznu.edu.cn}

\author*[3]{\fnm{Zhenyu} \sur{Yu}}\email{yuzy23@nudt.edu.cn}

\author[4]{\fnm{Junda} \sur{Zhang}}\email{summerfish@scut.edu.cn}

\affil[1]{\orgdiv{Department of Mathematical Sciences}, \orgname{Tsinghua University}, \orgaddress{\city{Beijing}, \postcode{100084}, \country{China}}}

\affil[2]{\orgdiv{Department of Mathematics}, \orgname{Hangzhou Normal University}, \orgaddress{\city{Hangzhou}, \postcode{310036},  \country{China}}}

\affil*[3]{\orgdiv{Department of Mathematics, College of Science}, \orgname{National University
of Defense Technology}, \orgaddress{\city{Changsha}, \postcode{410073}, \country{China}}}

\affil[4]{\orgdiv{School of Mathematics}, \orgname{South China University of Technology}, \orgaddress{\city{Guangzhou}, \postcode{510641}, \country{China}}}

\abstract{In this paper, we establish the existence of \textit{p}-energy norms and the corresponding \textit{p}-energy measures for scale-irregular Vicsek sets, which may lack self-similarity. We also investigate the characterization of \textit{p}-energy norms in terms of Besov-Lipschitz norms, with their weak monotonicity and the corresponding Bourgain-Brezis-Mironescu convergence.}

\date{\today}

\maketitle
\tableofcontents

\section{Introduction}\label{s.intro}

The study of \textit{non-linear potential theory} on metric measure spaces
has attracted significant attention in recent decades due to its important
role in classical analysis and differential equations. Many previous studies
have concentrated on the $p$-energy within specific classes of fractals. For
example, $p$-energy is constructed on self-similar p.c.f. sets by Herman, Peirone, and Strichartz \cite{HPS04}, Cao, Gu, and Qiu \cite{CGQ22}, Gao, Yu, and Zhang \cite{GYZ23}; on the
standard Sierpi\'nski carpet by Shimizu \cite{Shi24}, Murugan and
Shimizu \cite{MS25}; and on more general fractal spaces by Kigami \cite{Kig23}.
In these previous works, the self-similarity significantly influences the
construction of $p$-energy.

One aim of this work is trying to construct $p$-energy norm and $p$-energy measure without using the self-similar structure of the underlying fractal. The issue was first highlighted by Murugan and Shimizu in \cite[Problem 12.5]{MS25}, posing the challenge of constructing $p$-energy measure on the Sierpi\'nski carpet without using the self-similarity, and also establishing its basic properties. However, as far as the authors are concerned, many vital tools such as \textit{Fekete's lemma} used in \cite[Lemma 4.4]{CGQ22}, \textit{%
combinatorial ball Loewner condition} in \cite[Definition 3.1]{MS25} and
\textit{Knight Move argument} in \cite{Kig23,Shi24} are no longer
applicable without self-similarity. {Another aim is generalising the celebrated Bourgain-Brezis-Mironescu (BBM) convergence of $p$-energy semi-norms to non-self-similar fractals, where these semi-norms are supposed to have a different form from the self-similar case.}

For this reason, we concern ourselves here with the \textit{scale-irregular Vicsek sets}, which is a class of {homogeneous \textit{Moran sets} (see for example \cite{FWW97}) and $V$-variable fractals (with $V=1$, see for example \cite{BHS08})}. The advantage of Vicsek sets lies in their distinctive ``gradient structure", which allows us to construct $p$-energy norm even in the absence of self-similarity. It is worth mentioning that, when $p=2$, the $2$-energy (or Dirichlet form) may be constructed {and studied} by a probabilistic approach {as Hambly \cite{Ham92}}, Barlow and Hambly \cite{BH97} did on scale-irregular Sierpi\'nski gaskets.

Our approach is mainly motivated by the works of Baudoin and Chen \cite{BC23,BC24}. We will define $p$-energy norm as the limit of $p$-energy norms on discrete approximating graphs, since the underlying geometry structure ensures the monotonicity of discrete energy norms. To this end, for each scale-irregular Vicsek set $K^{\bm{l}}$ determined by contraction ratio sequence $\bm{l}$ (see Definition \ref{d.vicsek}), we always equip $K^{\bm{l}}$ with the Euclidean metric $d$ and the canonical Borel probability measure $\mu$ given by \eqref{e.defmu}. For a Borel set $B$, write $\fint_{B}f\dif\mu:=\mu(B)^{-1}\int_{B}f\dif\mu$. For a set $A\subset\mathbb{C}$ (the complex plane), denote its diameter by $\diam(A):=\sup_{x,y\in A}d(x,y)$.

\begin{theorem}
\label{t.energy} Let $(K^{\bm{l}}, d, \mu)$ be a scale-irregular Vicsek set.
For each $1<p<\infty$, there exists a normed vector space $(\mathcal{F}%
_{p},\lVert \cdot\rVert_{\mathcal{F}_{p}})$ and a semi-norm $\mathcal{E}_{p}$
on $\mathcal{F}_{p}$ with the following properties.

\begin{enumerate}[label=\textup{(\arabic*)}]

\item $(\mathcal{F}_{p},\lVert \cdot\rVert_{\mathcal{F}_{p}})$ is a
uniformly convex separable reflexive Banach space.

\item $\mathcal{F}_{p}$ forms an algebra under the pointwise product, that is, $%
uv\in\mathcal{F}_{p}$ whenever $u,v\in\mathcal{F}_{p}$. Moreover,
\begin{equation*}
\mathcal{E}_{p}(uv)\leq2^{p-1}\left(\lVert u\rVert_{C(K^{\bm{l}})}^{p}%
\mathcal{E}_{p}(v)+\lVert v\rVert_{C(K^{\bm{l}})}^{p}\mathcal{E}%
_{p}(u)\right)\ \text{for all }u,v\in\mathcal{F}_{p}.
\end{equation*}

\item {\normalfont{(Regularity)}} $\mathcal{F}_{p}\subset C(K^{\bm{l}})$ is a
dense subspace of $(C(K^{\bm{l}}),\lVert \cdot\rVert_{\infty})$.

\item {\normalfont{(Lipschitz contractivity)}} For every $u\in\mathcal{F}%
_{p} $ and $1$-Lipschitz function $\varphi:\mathbb{R}\rightarrow \mathbb{R}$%
, we have $\varphi\circ u\in\mathcal{F}_{p}$ and $\mathcal{E}%
_{p}(\varphi\circ u)\leq\mathcal{E}_{p}(u)$.

\item {\normalfont{(Spectral gap)}}{There is a constant $C\geq1$ such that,} for every $u\in\mathcal{F}_{p}$,%
\begin{equation*}
\int_{K^{\bm{l}}}\left|u(x)-\fint_{K^{\bm{l}}} u \dif \mu \right|^{p} \dif %
\mu(x)\leq {C}\func{diam}(K^{\bm{l}})^{p-1} \mathcal{E}_{p}(u).
\end{equation*}

\item {\normalfont{(Strong locality)}} If $u,v\in\mathcal{F}_{p}$ satisfy $%
\supp(u)\cap\supp(v-a\mathds{1}_{K^{\bm{l}}})=\emptyset$ for
some $a\in\mathbb{R}$, then
\begin{equation*}
\mathcal{E}(u+v)= \mathcal{E}(u)+\mathcal{E}(v).
\end{equation*}
\end{enumerate}
\end{theorem}
The explicit expressions of $\mathcal{E}_{p}$ and $\mathcal{F}_{p}$ are given in Definition \ref{penergy_1}. This helps to understand the dependence of the Sobolev spaces $\mathcal{F}_{p} $ on the exponent $p$. As we will see in Remark \ref{r.inter} that, for $1<p\neq q<\infty$, the intersection $\mathcal{F}_{p}\cap\mathcal{F}_{q}$ contains (many) non-constant functions.

Our next result shows the existence and some properties of the \textit{$p$%
-energy measure} corresponding to the $p$-energy norm in Theorem \ref{t.energy}.

\begin{theorem}
\label{t.emc} Let $(K^{\bm{l}}, d, \mu)$ be a scale-irregular Vicsek set and
$(\mathcal{E}_{p}, \mathcal{F}_{p})$ be the $p$-energy in Theorem \ref%
{t.energy}. Then there exists a family of Borel finite measures $%
\{\Gamma_{p}\langle u \rangle\}_{u\in\mathcal{F}_{p}}$ on $K^{\bm{l}}$
satisfying the following:

\begin{enumerate}[label=\textup{(\arabic*)}]
\item For $u\in\mathcal{F}_{p}$, $\Gamma_{p}\langle u \rangle(K^{\bm{l}%
})=\mathcal{E}_{p}(u)$. Moreover, $\Gamma_{p}\langle u \rangle =0$ if
and only if $u$ is constant.

\item For any two $u_{1},u_{2}\in\mathcal{F}_{p}$, and any non-negative
Borel measurable function $g$ on $K^{\bm{l}}$,
\begin{equation}  \label{e.em.tri0}
\left(\int_{K^{\bm{l}}}g \dif\Gamma_{p}\langle
u_{1}+u_{2}\rangle\right)^{1/p}\leq\left(\int_{K^{\bm{l}}}g\dif%
\Gamma_{p}\langle u_{1}\rangle\right)^{1/p}+\left(\int_{K^{\bm{l}}}g\dif%
\Gamma_{p}\langle u_{2}\rangle\right)^{1/p}.
\end{equation}

\item For any $f\in C^{1}(\mathbb{R})$ and any $u\in\mathcal{F}_{p}$,
\begin{equation}  \label{e.lip1}
\dif\Gamma_{p}\langle f\circ u
\rangle(x)=\left|f^{\prime}(u(x))\right|^{p}\dif\Gamma_{p}\langle u
\rangle( x)\ \text{for $\Gamma_{p}\langle u\rangle$-a.e. $x\in K^{\bm{l}}$}.
\end{equation}

\item {\normalfont{(Energy image density property)}} For any $u\in\mathcal{F}%
_{p}$, $u_{*}\left(\Gamma_{p}\langle u \rangle\right)\ll \mathscr{L}^{1}$.
Here $u_{*}\left(\Gamma_{p}\langle u \rangle\right)$ is the push-forward
measure defined by $u_{*}\left(\Gamma_{p}\langle u
\rangle\right)(A):=\Gamma_{p}\langle u \rangle(u^{-1}(A))$ for all Borel
subsets $A\subset\mathbb{R}$, and $\mathscr{L}^{1}$ is the Lebesgue measure
on $\mathbb{R}$.
\end{enumerate}
\end{theorem}

Some previous works have also constructed $p$-energy measures and discussed their properties by utilizing the self-similarity, as Hino \cite[Lemma 4.1]{Hin05}
did for $p=2$, or as Murugan and Shimizu \cite[Section 9]{MS25} did
for the standard Sierpi\'nski carpet. Motivated by \cite{BC23}, we use the gradient
structure as an alternative approach for both construction and direct
validation on the properties of the $p$-energy measure.

Within our framework, word spaces are applicable to construct the $p$-energy measure, see Proposition \ref{p.em}; however, the validation of related properties is impeded by the absence of self-similarity. This gap is bridged in Proposition \ref{p.coinc} by showing the coincidence of energy measures constructed by these two different approaches, and this suggests a compatibility between the gradient structure and the fractal structure.

{We define more discrete $p$-energy norms, denoted by $\mathcal{E}_{p,\infty}^{\beta}$ and $\mathcal{E}_{p,p}^{\beta}$ for $1<p<\infty$ and $0<\beta<\infty$ in Definition \ref{defE}. Then the $p$-energy $\mathcal{E}_{p}$ constructed in Theorem \ref{t.energy} is equivalent to {$\mathcal{E}_{p,\infty}^{\beta^{\ast}}$ with fixed $\beta^{\ast}>0$}. Intuitively, when $p=2$, the norm $\mathcal{E}_{2}={\mathcal{E}_{2,\infty}^{\beta^{\ast}}}$ gives a strongly local Dirichlet form, while $\mathcal{E}_{2,2}^{\beta}$ gives a non-local Dirichlet form (see \eqref{219}). Moreover, under an additional condition that \emph{$\bm{l}$ consists of finitely many contraction ratios}, we can describe the above norms in terms of Besov-Lipschitz norms. Let $\phi$ be the increasing scale function given in \eqref{phi}.}

\begin{definition}\label{def[}
{Fix $\beta^{\ast}\in(0,\infty)$.} For every $1<p<\infty $, $0\leq\beta <\infty $ and $0<r<\diam(K^{\bm{l}})$, define\footnote{{Our definition of Besov-Lipschitz spaces in Definition \ref{def[} may initially seem confusing, as the spaces $B_{p,q}^{\beta^{\ast}}$ does not actually depend on the choice of $\beta^{\ast}$. We use this definition to ensure the consistency with existing literature (e.g. \cite{Bau24}) regarding the \emph{critical Besov exponent} in \eqref{e.critical} in the self-similar case. See Remark \ref{r.coinBau}.}
}
\begin{equation}\label{e.belip}
\Phi _{u}^{\beta}(r):=\int_{K^{\bm{l}}}\fint_{B(x,r)}\frac{|u(x)-u(y)|^{p}}{\phi (r)^{\beta/\beta^{\ast}}}\dif\mu (y)\dif\mu (x),\ 0<r\leq\diam(K^{\bm{l}}),\ u\in L^{p}(K^{\bm{l}},\mu).
\end{equation}
For $1<q<\infty$, let
\[\lbrack u]_{B_{p,q}^{\beta }}:=\left(\int_{0}^{\diam (K^{\bm{l}})}\left(\Phi _{u}^{\beta}(r)\right)^{q/p}\frac{\dif r}{r}\right)^{1/q}\ \text{and}\ B_{p,q}^{\beta }:=\left\{ u\in L^{p}(K^{\bm{l}},\mu):[u]_{B_{p,q}^{\beta }}<\infty \right\};\]
for $q=\infty$, let
\[\lbrack u]_{B_{p,\infty}^{\beta}}:=\sup_{r\in(0,\diam (K^{\bm{l}})]}\Phi _{u}^{\beta}(r)^{1/p}\ \text{and}\ B_{p,\infty}^{\beta}:=\left\{ u\in L^{p}(K^{\bm{l}},\mu):[u]_{B_{p,\infty}^{\beta }}<\infty \right\}.
\]
\end{definition}

The above Besov-Lipschitz spaces were introduced by Jonsson and Wallin \cite{JW84} for Euclidean spaces, by Korevaar and Schoen \cite{KS93} for Riemann domains and by Jonsson \cite{Jon96} for the Sierpi\'nski gasket. The main difference between Besov-Lipschitz spaces on scale-irregular fractals and those on manifolds or self-similar fractals is that, $\phi$ is not a power function, which makes the analysis more complicated.

Our last main result is that, for $\beta$ close enough to $\beta^{\ast}$, the discrete semi-norms $\mathcal{E}_{p}$, $\mathcal{E}_{p,\infty}^{\beta}$ and $\mathcal{E}_{p,p}^{\beta}$ in Definition \ref{defE} are comparable with the Besov-Lipschitz norms $[\cdot]_{B_{p,\infty}^{\beta^{\ast}}}$, $[\cdot]_{B_{p,\infty}^{\beta }}$ and $[\cdot]_{B_{p,p}^{\beta }}$ respectively. Furthermore, the weak monotonicity and BBM convergence of $p$-energy semi-norms are established on scale-irregular Vicsek sets.
\begin{theorem}\label{thm3}
{If $\bm{l}=(l_{n})_{n=1}^{\infty}$ satisfies $\sup\limits_{n \geq 1}l_{n}<\infty$,} then for each $1<p<\infty$:
\begin{enumerate}[label=\textup{(\arabic*)}]
\item There exists a constant $\epsilon_{p}\in (0,1)$ {depending only on $\bm{l}$ and $p$}, such that for every $\beta\in(\epsilon_{p}\beta^{\ast},\infty)$, we have $B_{p,p}^{\beta}\subset B_{p,\infty}^{\beta}\subset  C(K^{\bm{l}})$, and there exists $C\geq1$ such that for all $u\in C(K^{\bm{l}})$,  \[C^{-1}\lbrack u]_{B_{p,\infty}^{\beta}}^{p}\leq \mathcal{E}_{p,\infty }^{\beta}(u)\leq C \lbrack u]_{B_{p,\infty}^{\beta}}^{p},\]
\[C^{-1}\lbrack u]_{B_{p,p}^{\beta}}^{p}\leq\mathcal{E}_{p,p}^{\beta}(u)\leq C \lbrack u]_{B_{p,p}^{\beta}}^{p}.\]
In particular, when $\beta=\beta^{\ast}$, {we have $\mathcal{E}_{p}(u)\asymp [u]_{B_{p,\infty}^{\beta^{\ast}}}$ for $u\in \mathcal{F}_{p}=B_{p,\infty}^{\beta^{\ast}}$ (given in Theorem \ref{t.energy})}.
\item $\beta^{\ast}$ is the {\em critical Besov exponent}, that is
\begin{equation}\label{e.critical}
\beta^{\ast }=\max \left\{ \beta \in[0,\infty): B_{p,\infty
}^{\beta }(K^{\bm{l}})\ \text{contains non-constant functions}\right\} .
\end{equation}
\item {\normalfont{(Weak-monotonicity property)}} For all $u\in \mathcal{F}_{p}=B_{p,\infty }^{\beta ^{\ast }}$,
\begin{equation*}
\sup_{r\in (0,\diam (K^{\bm{l}})]}\Phi _{u}^{\beta^{\ast}}(r)\leq
C\liminf_{r\rightarrow 0}\Phi _{u}^{\beta^{\ast}}(r).
\end{equation*}
\item {\normalfont{(Bourgain-Brezis-Mironescu (BBM) convergence)}} There exists $C>1$ such that for all $u\in \mathcal{F}_{p}=B_{p,\infty }^{\beta^{\ast}}$,
\begin{equation}\label{e.bbmdis}
C^{-1}\mathcal{E}_{p}(u)\leq \liminf_{\beta \uparrow
\beta^*}(\beta^{\ast}-\beta)\mathcal{E}_{p,p}^\beta(u)\leq \limsup_{\beta
\uparrow \beta^{\ast}}(\beta^{\ast}-\beta)\mathcal{E}_{p,p}^\beta(u)\leq C\mathcal{%
E} _{p}(u),
\end{equation}
and
\begin{equation}\label{e.bbmbes}
C^{-1}[u]_{{B}_{p,\infty }^{\beta^{\ast}}}^{p}\leq \liminf_{\beta
\uparrow \beta^{\ast}}(\beta^{\ast}-\beta )[u]_{{B}_{p,p}^{\beta
}}^{p}\leq \limsup_{\beta \uparrow \beta^{\ast}}(\beta^{\ast
}-\beta )[u]_{{B}_{p,p}^{\beta }}^{p}\leq C[u]_{{B}_{p,\infty }^{\beta^{\ast }}}^{p}.
\end{equation}
\end{enumerate}
\end{theorem}

The structure of the paper is as follows. In Section \ref{s.geo}, we
introduce scale-irregular Vicsek set and its measure, and discuss their properties, {including the volume doubling property (Proposition \ref{l.meas}), the existence of the Hausdorff measure $\mathcal{H}^{\alpha}$ (Proposition \ref{HHH}), the Ahlfors regularity (Proposition \ref{af}) and the non-self-similarity (Theorem \ref{t.non-sml})}. In Section \ref{s.energy}, we construct $p$-energy norm and $p$-energy measure on scale-irregular Vicsek sets and prove Theorems \ref{t.energy} and \ref{t.emc}. In Section \ref{s.norm}, we study Besov-Lipschitz norms related to the $p$-energy, including the weak monotonicity and BBM convergence, and prove Theorem \ref{thm3}. {Some possible extensions of our results are discussed in Section \ref{s.Discuss}.}

\begin{notation} The letters $C$,$C^{\prime }$, $C_{i}$,$C_{i}^{\prime }$,
$C_{i}^{\prime \prime }$ and $c$ are universal positive constants which may vary at each occurrence. The sign $\asymp $ means that
both $\leq $ and $\geq $ are true with uniform values of $C$ depending only
on $K^{\bm{l}}$. For $a,b\in \mathbb{R},\ a\wedge b:=\min \{a,b\}$, $a\vee b:=\max \{a,b\}$. We use $%
\#A$ for the cardinality of a set $A$.
\end{notation}

\section{Geometry and measure of a scale-irregular Vicsek set}\label{s.geo}
The arrangement of this section is as follows. We state the definition and related notions of scale-irregular Vicsek sets in Section \ref{s.pre}. The volume doubling property of the measures on scale-irregular Vicsek sets are analyzed in Section \ref{s.meas}. {The criterion for the existence of the Hausdorff measure $\mathcal{H}^{\alpha}$ and Ahlfors regularity, and sufficient conditions for scale-irregular Vicsek sets to be non-self-similar are given in Section \ref{s.nss}.}
\subsection{Preliminaries}\label{s.pre}

Define five points in the complex plane $\mathbb{C}$ by
\begin{equation*}
q_0:=0,\ q_j:=\exp((2j-1)\pi i/4),\quad1\leq j\leq4.
\end{equation*}
Let $K_{0}$ be the closed unit square in $\mathbb{C}$ with vertices $%
\{q_{j}\}_{j=1}^{4}$. Given an odd number $l\geq3$, define
\begin{equation*}
S_{l}:=\left\{2nl^{-1}q_{j}:0\leq n\leq\frac{1}{2}(l-1),\ 1\leq
j\leq4\right\}
\end{equation*}
so that $\# S_{l}=2l-1$. We assign $S_{l}$ the discrete topology for each $l$%
. For convenience, let $l_{0}:=1$. For any infinite sequence $\bm{l}=(l_{k})_{k=1}^{\infty }$, where each $%
l_{k}\geq 3$ is an odd integer for $k\geq1$, define
\begin{equation*}
\rho_{n}:=2\prod_{k=0}^{n}l_{k}^{-1}\ \text{for $0\leq n<\infty$,}\ W_{n}^{\bm{l}}:=%
\prod_{k=1}^{n}S_{l_{k}}\ \text{for $1\leq n\leq\infty$ and}\ W_{*}^{\bm{l}%
}:=\bigcup_{1\leq n<\infty}W_{n}^{\bm{l}}.
\end{equation*}
We assign $W_{n}^{\bm{l}}$ and $W_{*}^{\bm{l}}$ the product topology. For
each $w=w_{1}w_{2}\cdots\in W_{\infty}^{\bm{l}}$, we define $%
[w]_{n}:=w_{1}w_{2}\cdots w_{n}\in W_{n}^{\bm{l}}$ and $[w]_{n}$
for $w\in W_{k}^{\bm{l}}$ when $k\geq n\geq1$ similarly. For $%
w=w_{1}\cdots w_{n}\in W_{*}^{\bm{l}}$, we write
\begin{equation*}
S(w):= \{v\in W_{n+1}^{\bm{l}}: [v]_{n}=w\}= \{w_{1}w_{2}\cdots
w_{n}w_{n+1}:w_{n+1}\in S_{l_{n+1}}\}.
\end{equation*}

For $\alpha\in(0,1)$, we define a function $\delta$ on $W_{\infty}^{\bm{l}%
}\times W_{\infty}^{\bm{l}}$ by
\begin{equation*}
\delta(w,\tau):=%
\begin{cases}
\alpha^{\min\{n:[w]_{n}\neq[\tau]_{n}\}-1}\quad & \text{if }w\neq\tau, \\
0 \quad & \text{if }w=\tau.%
\end{cases}%
\end{equation*}
Then $\delta$ is a metric on $W_{\infty}^{\bm{l}}$ and generates the same
topology on $W_{\infty}^{\bm{l}}$.

For each $w\in S_{l}$, we define a map
\begin{equation*}
F_{w}^{l}(z):=w+l^{-1}z,\ z\in \mathbb{C},
\end{equation*}%
and for each $w=w_{1}\ldots w_{n}\in W_{\ast }^{\bm{l}}$, define
\begin{equation*}
F_{w}^{\bm{l}}:=F_{w_{1}}^{l_{1}}\circ \cdots \circ F_{w_{n}}^{l_{n}}.
\end{equation*}%
Let $d$ be the Eucildean metric on $\mathbb{C}$. Note that for each $w=w_{1}\ldots w_{n}\in W_{\ast }^{\bm{l}}$, the set $%
F_{w}^{\bm{l}}(K_{0})$ is an isometric copy of $%
[0,2^{1/2}l_{1}^{-1}l_{2}^{-1}\cdots l_{n}^{-1}]^{2}$, i.e., $F_{w}^{\bm{l}%
}(K_{0})$ is a square with side length $2^{1/2}l_{1}^{-1}l_{2}^{-1}\cdots
l_{n}^{-1}$ and
\begin{equation*}
\diam(F_{w}^{\bm{l}}(K_{0}))=2l_{1}^{-1}l_{2}^{-1}\cdots l_{n}^{-1}\leq
2\cdot 3^{-|w|}.
\end{equation*}

\begin{definition}[The scale-irregular Vicsek set]
\label{d.vicsek} For any infinite sequence $\bm{l}=(l_{k})_{k=1}^{\infty }$,
where each $l_{k}\geq 3$ is an odd integer, define
$K^{\bm{l}}$ to be the non-empty compact subset of $K_{0}$ by
\begin{equation*}
K^{\bm{l}}:=\bigcap_{n=1}^{\infty }\bigcup_{w\in W_{n}^{\bm{l}}}F_{w}^{\bm{l}%
}(K_{0}).
\end{equation*}
The metric on $K^{\bm{l}}$ is given by the restriction of the Euclidean
metric $d$ on $\mathbb{C}$ to $K^{\bm{l}}$. We call the metric space $K^{\bm{l}}$ a \emph{%
scale-irregular Vicsek set} (see Fig. \ref{fig1} for illustration).
\end{definition}

\begin{figure}[tbph]
\centering
\subfloat[$K^{\bm{l}}$ at level $3$, with $\bm{l}=353\cdots$.]{
\begin{tikzpicture}[scale=0.13]
\draw[fill=black] (0,0)--(1,0)--(1,1)--(0,1)--cycle;
\draw[fill=black] (2,0)--(3,0)--(3,1)--(2,1)--cycle;
\draw[fill=black] (1,1)--(2,1)--(2,2)--(1,2)--cycle;
\draw[fill=black] (0,2)--(1,2)--(1,3)--(0,3)--cycle;
\draw[fill=black] (2,2)--(3,2)--(3,3)--(2,3)--cycle;

\draw[fill=black] (0+6,0+6)--(1+6,0+6)--(1+6,1+6)--(0+6,1+6)--cycle;
\draw[fill=black] (2+6,0+6)--(3+6,0+6)--(3+6,1+6)--(2+6,1+6)--cycle;
\draw[fill=black] (1+6,1+6)--(2+6,1+6)--(2+6,2+6)--(1+6,2+6)--cycle;
\draw[fill=black] (0+6,2+6)--(1+6,2+6)--(1+6,3+6)--(0+6,3+6)--cycle;
\draw[fill=black] (2+6,2+6)--(3+6,2+6)--(3+6,3+6)--(2+6,3+6)--cycle;

\draw[fill=black] (0+3,0+3)--(1+3,0+3)--(1+3,1+3)--(0+3,1+3)--cycle;
\draw[fill=black] (2+3,0+3)--(3+3,0+3)--(3+3,1+3)--(2+3,1+3)--cycle;
\draw[fill=black] (1+3,1+3)--(2+3,1+3)--(2+3,2+3)--(1+3,2+3)--cycle;
\draw[fill=black] (0+3,2+3)--(1+3,2+3)--(1+3,3+3)--(0+3,3+3)--cycle;
\draw[fill=black] (2+3,2+3)--(3+3,2+3)--(3+3,3+3)--(2+3,3+3)--cycle;

\draw[fill=black] (0+9,0+3)--(1+9,0+3)--(1+9,1+3)--(0+9,1+3)--cycle;
\draw[fill=black] (2+9,0+3)--(3+9,0+3)--(3+9,1+3)--(2+9,1+3)--cycle;
\draw[fill=black] (1+9,1+3)--(2+9,1+3)--(2+9,2+3)--(1+9,2+3)--cycle;
\draw[fill=black] (0+9,2+3)--(1+9,2+3)--(1+9,3+3)--(0+9,3+3)--cycle;
\draw[fill=black] (2+9,2+3)--(3+9,2+3)--(3+9,3+3)--(2+9,3+3)--cycle;

\draw[fill=black] (0+3,0+9)--(1+3,0+9)--(1+3,1+9)--(0+3,1+9)--cycle;
\draw[fill=black] (2+3,0+9)--(3+3,0+9)--(3+3,1+9)--(2+3,1+9)--cycle;
\draw[fill=black] (1+3,1+9)--(2+3,1+9)--(2+3,2+9)--(1+3,2+9)--cycle;
\draw[fill=black] (0+3,2+9)--(1+3,2+9)--(1+3,3+9)--(0+3,3+9)--cycle;
\draw[fill=black] (2+3,2+9)--(3+3,2+9)--(3+3,3+9)--(2+3,3+9)--cycle;

\draw[fill=black] (0+9,0+9)--(1+9,0+9)--(1+9,1+9)--(0+9,1+9)--cycle;
\draw[fill=black] (2+9,0+9)--(3+9,0+9)--(3+9,1+9)--(2+9,1+9)--cycle;
\draw[fill=black] (1+9,1+9)--(2+9,1+9)--(2+9,2+9)--(1+9,2+9)--cycle;
\draw[fill=black] (0+9,2+9)--(1+9,2+9)--(1+9,3+9)--(0+9,3+9)--cycle;
\draw[fill=black] (2+9,2+9)--(3+9,2+9)--(3+9,3+9)--(2+9,3+9)--cycle;

\draw[fill=black] (0,0+12)--(1,0+12)--(1,1+12)--(0,1+12)--cycle;
\draw[fill=black] (2,0+12)--(3,0+12)--(3,1+12)--(2,1+12)--cycle;
\draw[fill=black] (1,1+12)--(2,1+12)--(2,2+12)--(1,2+12)--cycle;
\draw[fill=black] (0,2+12)--(1,2+12)--(1,3+12)--(0,3+12)--cycle;
\draw[fill=black] (2,2+12)--(3,2+12)--(3,3+12)--(2,3+12)--cycle;

\draw[fill=black] (0+12,0)--(1+12,0)--(1+12,1)--(0+12,1)--cycle;
\draw[fill=black] (2+12,0)--(3+12,0)--(3+12,1)--(2+12,1)--cycle;
\draw[fill=black] (1+12,1)--(2+12,1)--(2+12,2)--(1+12,2)--cycle;
\draw[fill=black] (0+12,2)--(1+12,2)--(1+12,3)--(0+12,3)--cycle;
\draw[fill=black] (2+12,2)--(3+12,2)--(3+12,3)--(2+12,3)--cycle;

\draw[fill=black] (0+12,0+12)--(1+12,0+12)--(1+12,1+12)--(0+12,1+12)--cycle;
\draw[fill=black] (2+12,0+12)--(3+12,0+12)--(3+12,1+12)--(2+12,1+12)--cycle;
\draw[fill=black] (1+12,1+12)--(2+12,1+12)--(2+12,2+12)--(1+12,2+12)--cycle;
\draw[fill=black] (0+12,2+12)--(1+12,2+12)--(1+12,3+12)--(0+12,3+12)--cycle;
\draw[fill=black] (2+12,2+12)--(3+12,2+12)--(3+12,3+12)--(2+12,3+12)--cycle;

\draw[fill=black] (0+30,0)--(1+30,0)--(1+30,1)--(0+30,1)--cycle;
\draw[fill=black] (2+30,0)--(3+30,0)--(3+30,1)--(2+30,1)--cycle;
\draw[fill=black] (1+30,1)--(2+30,1)--(2+30,2)--(1+30,2)--cycle;
\draw[fill=black] (0+30,2)--(1+30,2)--(1+30,3)--(0+30,3)--cycle;
\draw[fill=black] (2+30,2)--(3+30,2)--(3+30,3)--(2+30,3)--cycle;

\draw[fill=black] (0+6+30,0+6)--(1+6+30,0+6)--(1+6+30,1+6)--(0+6+30,1+6)--cycle;
\draw[fill=black] (2+6+30,0+6)--(3+6+30,0+6)--(3+6+30,1+6)--(2+6+30,1+6)--cycle;
\draw[fill=black] (1+6+30,1+6)--(2+6+30,1+6)--(2+6+30,2+6)--(1+6+30,2+6)--cycle;
\draw[fill=black] (0+6+30,2+6)--(1+6+30,2+6)--(1+6+30,3+6)--(0+6+30,3+6)--cycle;
\draw[fill=black] (2+6+30,2+6)--(3+6+30,2+6)--(3+6+30,3+6)--(2+6+30,3+6)--cycle;

\draw[fill=black] (0+3+30,0+3)--(1+3+30,0+3)--(1+3+30,1+3)--(0+3+30,1+3)--cycle;
\draw[fill=black] (2+3+30,0+3)--(3+3+30,0+3)--(3+3+30,1+3)--(2+3+30,1+3)--cycle;
\draw[fill=black] (1+3+30,1+3)--(2+3+30,1+3)--(2+3+30,2+3)--(1+3+30,2+3)--cycle;
\draw[fill=black] (0+3+30,2+3)--(1+3+30,2+3)--(1+3+30,3+3)--(0+3+30,3+3)--cycle;
\draw[fill=black] (2+3+30,2+3)--(3+3+30,2+3)--(3+3+30,3+3)--(2+3+30,3+3)--cycle;

\draw[fill=black] (0+9+30,0+3)--(1+9+30,0+3)--(1+9+30,1+3)--(0+9+30,1+3)--cycle;
\draw[fill=black] (2+9+30,0+3)--(3+9+30,0+3)--(3+9+30,1+3)--(2+9+30,1+3)--cycle;
\draw[fill=black] (1+9+30,1+3)--(2+9+30,1+3)--(2+9+30,2+3)--(1+9+30,2+3)--cycle;
\draw[fill=black] (0+9+30,2+3)--(1+9+30,2+3)--(1+9+30,3+3)--(0+9+30,3+3)--cycle;
\draw[fill=black] (2+9+30,2+3)--(3+9+30,2+3)--(3+9+30,3+3)--(2+9+30,3+3)--cycle;

\draw[fill=black] (0+3+30,0+9)--(1+3+30,0+9)--(1+3+30,1+9)--(0+3+30,1+9)--cycle;
\draw[fill=black] (2+3+30,0+9)--(3+3+30,0+9)--(3+3+30,1+9)--(2+3+30,1+9)--cycle;
\draw[fill=black] (1+3+30,1+9)--(2+3+30,1+9)--(2+3+30,2+9)--(1+3+30,2+9)--cycle;
\draw[fill=black] (0+3+30,2+9)--(1+3+30,2+9)--(1+3+30,3+9)--(0+3+30,3+9)--cycle;
\draw[fill=black] (2+3+30,2+9)--(3+3+30,2+9)--(3+3+30,3+9)--(2+3+30,3+9)--cycle;

\draw[fill=black] (0+9+30,0+9)--(1+9+30,0+9)--(1+9+30,1+9)--(0+9+30,1+9)--cycle;
\draw[fill=black] (2+9+30,0+9)--(3+9+30,0+9)--(3+9+30,1+9)--(2+9+30,1+9)--cycle;
\draw[fill=black] (1+9+30,1+9)--(2+9+30,1+9)--(2+9+30,2+9)--(1+9+30,2+9)--cycle;
\draw[fill=black] (0+9+30,2+9)--(1+9+30,2+9)--(1+9+30,3+9)--(0+9+30,3+9)--cycle;
\draw[fill=black] (2+9+30,2+9)--(3+9+30,2+9)--(3+9+30,3+9)--(2+9+30,3+9)--cycle;

\draw[fill=black] (0+30,0+12)--(1+30,0+12)--(1+30,1+12)--(0+30,1+12)--cycle;
\draw[fill=black] (2+30,0+12)--(3+30,0+12)--(3+30,1+12)--(2+30,1+12)--cycle;
\draw[fill=black] (1+30,1+12)--(2+30,1+12)--(2+30,2+12)--(1+30,2+12)--cycle;
\draw[fill=black] (0+30,2+12)--(1+30,2+12)--(1+30,3+12)--(0+30,3+12)--cycle;
\draw[fill=black] (2+30,2+12)--(3+30,2+12)--(3+30,3+12)--(2+30,3+12)--cycle;

\draw[fill=black] (0+12+30,0)--(1+12+30,0)--(1+12+30,1)--(0+12+30,1)--cycle;
\draw[fill=black] (2+12+30,0)--(3+12+30,0)--(3+12+30,1)--(2+12+30,1)--cycle;
\draw[fill=black] (1+12+30,1)--(2+12+30,1)--(2+12+30,2)--(1+12+30,2)--cycle;
\draw[fill=black] (0+12+30,2)--(1+12+30,2)--(1+12+30,3)--(0+12+30,3)--cycle;
\draw[fill=black] (2+12+30,2)--(3+12+30,2)--(3+12+30,3)--(2+12+30,3)--cycle;

\draw[fill=black] (0+12+30,0+12)--(1+12+30,0+12)--(1+12+30,1+12)--(0+12+30,1+12)--cycle;
\draw[fill=black] (2+12+30,0+12)--(3+12+30,0+12)--(3+12+30,1+12)--(2+12+30,1+12)--cycle;
\draw[fill=black] (1+12+30,1+12)--(2+12+30,1+12)--(2+12+30,2+12)--(1+12+30,2+12)--cycle;
\draw[fill=black] (0+12+30,2+12)--(1+12+30,2+12)--(1+12+30,3+12)--(0+12+30,3+12)--cycle;
\draw[fill=black] (2+12+30,2+12)--(3+12+30,2+12)--(3+12+30,3+12)--(2+12+30,3+12)--cycle;

\draw[fill=black] (0,0+30)--(1,0+30)--(1,1+30)--(0,1+30)--cycle;
\draw[fill=black] (2,0+30)--(3,0+30)--(3,1+30)--(2,1+30)--cycle;
\draw[fill=black] (1,1+30)--(2,1+30)--(2,2+30)--(1,2+30)--cycle;
\draw[fill=black] (0,2+30)--(1,2+30)--(1,3+30)--(0,3+30)--cycle;
\draw[fill=black] (2,2+30)--(3,2+30)--(3,3+30)--(2,3+30)--cycle;

\draw[fill=black] (0+6,0+6+30)--(1+6,0+6+30)--(1+6,1+6+30)--(0+6,1+6+30)--cycle;
\draw[fill=black] (2+6,0+6+30)--(3+6,0+6+30)--(3+6,1+6+30)--(2+6,1+6+30)--cycle;
\draw[fill=black] (1+6,1+6+30)--(2+6,1+6+30)--(2+6,2+6+30)--(1+6,2+6+30)--cycle;
\draw[fill=black] (0+6,2+6+30)--(1+6,2+6+30)--(1+6,3+6+30)--(0+6,3+6+30)--cycle;
\draw[fill=black] (2+6,2+6+30)--(3+6,2+6+30)--(3+6,3+6+30)--(2+6,3+6+30)--cycle;

\draw[fill=black] (0+3,0+3+30)--(1+3,0+3+30)--(1+3,1+3+30)--(0+3,1+3+30)--cycle;
\draw[fill=black] (2+3,0+3+30)--(3+3,0+3+30)--(3+3,1+3+30)--(2+3,1+3+30)--cycle;
\draw[fill=black] (1+3,1+3+30)--(2+3,1+3+30)--(2+3,2+3+30)--(1+3,2+3+30)--cycle;
\draw[fill=black] (0+3,2+3+30)--(1+3,2+3+30)--(1+3,3+3+30)--(0+3,3+3+30)--cycle;
\draw[fill=black] (2+3,2+3+30)--(3+3,2+3+30)--(3+3,3+3+30)--(2+3,3+3+30)--cycle;

\draw[fill=black] (0+9,0+3+30)--(1+9,0+3+30)--(1+9,1+3+30)--(0+9,1+3+30)--cycle;
\draw[fill=black] (2+9,0+3+30)--(3+9,0+3+30)--(3+9,1+3+30)--(2+9,1+3+30)--cycle;
\draw[fill=black] (1+9,1+3+30)--(2+9,1+3+30)--(2+9,2+3+30)--(1+9,2+3+30)--cycle;
\draw[fill=black] (0+9,2+3+30)--(1+9,2+3+30)--(1+9,3+3+30)--(0+9,3+3+30)--cycle;
\draw[fill=black] (2+9,2+3+30)--(3+9,2+3+30)--(3+9,3+3+30)--(2+9,3+3+30)--cycle;

\draw[fill=black] (0+3,0+9+30)--(1+3,0+9+30)--(1+3,1+9+30)--(0+3,1+9+30)--cycle;
\draw[fill=black] (2+3,0+9+30)--(3+3,0+9+30)--(3+3,1+9+30)--(2+3,1+9+30)--cycle;
\draw[fill=black] (1+3,1+9+30)--(2+3,1+9+30)--(2+3,2+9+30)--(1+3,2+9+30)--cycle;
\draw[fill=black] (0+3,2+9+30)--(1+3,2+9+30)--(1+3,3+9+30)--(0+3,3+9+30)--cycle;
\draw[fill=black] (2+3,2+9+30)--(3+3,2+9+30)--(3+3,3+9+30)--(2+3,3+9+30)--cycle;

\draw[fill=black] (0+9,0+9+30)--(1+9,0+9+30)--(1+9,1+9+30)--(0+9,1+9+30)--cycle;
\draw[fill=black] (2+9,0+9+30)--(3+9,0+9+30)--(3+9,1+9+30)--(2+9,1+9+30)--cycle;
\draw[fill=black] (1+9,1+9+30)--(2+9,1+9+30)--(2+9,2+9+30)--(1+9,2+9+30)--cycle;
\draw[fill=black] (0+9,2+9+30)--(1+9,2+9+30)--(1+9,3+9+30)--(0+9,3+9+30)--cycle;
\draw[fill=black] (2+9,2+9+30)--(3+9,2+9+30)--(3+9,3+9+30)--(2+9,3+9+30)--cycle;

\draw[fill=black] (0,0+12+30)--(1,0+12+30)--(1,1+12+30)--(0,1+12+30)--cycle;
\draw[fill=black] (2,0+12+30)--(3,0+12+30)--(3,1+12+30)--(2,1+12+30)--cycle;
\draw[fill=black] (1,1+12+30)--(2,1+12+30)--(2,2+12+30)--(1,2+12+30)--cycle;
\draw[fill=black] (0,2+12+30)--(1,2+12+30)--(1,3+12+30)--(0,3+12+30)--cycle;
\draw[fill=black] (2,2+12+30)--(3,2+12+30)--(3,3+12+30)--(2,3+12+30)--cycle;

\draw[fill=black] (0+12,0+30)--(1+12,0+30)--(1+12,1+30)--(0+12,1+30)--cycle;
\draw[fill=black] (2+12,0+30)--(3+12,0+30)--(3+12,1+30)--(2+12,1+30)--cycle;
\draw[fill=black] (1+12,1+30)--(2+12,1+30)--(2+12,2+30)--(1+12,2+30)--cycle;
\draw[fill=black] (0+12,2+30)--(1+12,2+30)--(1+12,3+30)--(0+12,3+30)--cycle;
\draw[fill=black] (2+12,2+30)--(3+12,2+30)--(3+12,3+30)--(2+12,3+30)--cycle;

\draw[fill=black] (0+12,0+12+30)--(1+12,0+12+30)--(1+12,1+12+30)--(0+12,1+12+30)--cycle;
\draw[fill=black] (2+12,0+12+30)--(3+12,0+12+30)--(3+12,1+12+30)--(2+12,1+12+30)--cycle;
\draw[fill=black] (1+12,1+12+30)--(2+12,1+12+30)--(2+12,2+12+30)--(1+12,2+12+30)--cycle;
\draw[fill=black] (0+12,2+12+30)--(1+12,2+12+30)--(1+12,3+12+30)--(0+12,3+12+30)--cycle;
\draw[fill=black] (2+12,2+12+30)--(3+12,2+12+30)--(3+12,3+12+30)--(2+12,3+12+30)--cycle;

\draw[fill=black] (0+30,0+30)--(1+30,0+30)--(1+30,1+30)--(0+30,1+30)--cycle;
\draw[fill=black] (2+30,0+30)--(3+30,0+30)--(3+30,1+30)--(2+30,1+30)--cycle;
\draw[fill=black] (1+30,1+30)--(2+30,1+30)--(2+30,2+30)--(1+30,2+30)--cycle;
\draw[fill=black] (0+30,2+30)--(1+30,2+30)--(1+30,3+30)--(0+30,3+30)--cycle;
\draw[fill=black] (2+30,2+30)--(3+30,2+30)--(3+30,3+30)--(2+30,3+30)--cycle;

\draw[fill=black] (0+6+30,0+6+30)--(1+6+30,0+6+30)--(1+6+30,1+6+30)--(0+6+30,1+6+30)--cycle;
\draw[fill=black] (2+6+30,0+6+30)--(3+6+30,0+6+30)--(3+6+30,1+6+30)--(2+6+30,1+6+30)--cycle;
\draw[fill=black] (1+6+30,1+6+30)--(2+6+30,1+6+30)--(2+6+30,2+6+30)--(1+6+30,2+6+30)--cycle;
\draw[fill=black] (0+6+30,2+6+30)--(1+6+30,2+6+30)--(1+6+30,3+6+30)--(0+6+30,3+6+30)--cycle;
\draw[fill=black] (2+6+30,2+6+30)--(3+6+30,2+6+30)--(3+6+30,3+6+30)--(2+6+30,3+6+30)--cycle;

\draw[fill=black] (0+3+30,0+3+30)--(1+3+30,0+3+30)--(1+3+30,1+3+30)--(0+3+30,1+3+30)--cycle;
\draw[fill=black] (2+3+30,0+3+30)--(3+3+30,0+3+30)--(3+3+30,1+3+30)--(2+3+30,1+3+30)--cycle;
\draw[fill=black] (1+3+30,1+3+30)--(2+3+30,1+3+30)--(2+3+30,2+3+30)--(1+3+30,2+3+30)--cycle;
\draw[fill=black] (0+3+30,2+3+30)--(1+3+30,2+3+30)--(1+3+30,3+3+30)--(0+3+30,3+3+30)--cycle;
\draw[fill=black] (2+3+30,2+3+30)--(3+3+30,2+3+30)--(3+3+30,3+3+30)--(2+3+30,3+3+30)--cycle;

\draw[fill=black] (0+9+30,0+3+30)--(1+9+30,0+3+30)--(1+9+30,1+3+30)--(0+9+30,1+3+30)--cycle;
\draw[fill=black] (2+9+30,0+3+30)--(3+9+30,0+3+30)--(3+9+30,1+3+30)--(2+9+30,1+3+30)--cycle;
\draw[fill=black] (1+9+30,1+3+30)--(2+9+30,1+3+30)--(2+9+30,2+3+30)--(1+9+30,2+3+30)--cycle;
\draw[fill=black] (0+9+30,2+3+30)--(1+9+30,2+3+30)--(1+9+30,3+3+30)--(0+9+30,3+3+30)--cycle;
\draw[fill=black] (2+9+30,2+3+30)--(3+9+30,2+3+30)--(3+9+30,3+3+30)--(2+9+30,3+3+30)--cycle;

\draw[fill=black] (0+3+30,0+9+30)--(1+3+30,0+9+30)--(1+3+30,1+9+30)--(0+3+30,1+9+30)--cycle;
\draw[fill=black] (2+3+30,0+9+30)--(3+3+30,0+9+30)--(3+3+30,1+9+30)--(2+3+30,1+9+30)--cycle;
\draw[fill=black] (1+3+30,1+9+30)--(2+3+30,1+9+30)--(2+3+30,2+9+30)--(1+3+30,2+9+30)--cycle;
\draw[fill=black] (0+3+30,2+9+30)--(1+3+30,2+9+30)--(1+3+30,3+9+30)--(0+3+30,3+9+30)--cycle;
\draw[fill=black] (2+3+30,2+9+30)--(3+3+30,2+9+30)--(3+3+30,3+9+30)--(2+3+30,3+9+30)--cycle;

\draw[fill=black] (0+9+30,0+9+30)--(1+9+30,0+9+30)--(1+9+30,1+9+30)--(0+9+30,1+9+30)--cycle;
\draw[fill=black] (2+9+30,0+9+30)--(3+9+30,0+9+30)--(3+9+30,1+9+30)--(2+9+30,1+9+30)--cycle;
\draw[fill=black] (1+9+30,1+9+30)--(2+9+30,1+9+30)--(2+9+30,2+9+30)--(1+9+30,2+9+30)--cycle;
\draw[fill=black] (0+9+30,2+9+30)--(1+9+30,2+9+30)--(1+9+30,3+9+30)--(0+9+30,3+9+30)--cycle;
\draw[fill=black] (2+9+30,2+9+30)--(3+9+30,2+9+30)--(3+9+30,3+9+30)--(2+9+30,3+9+30)--cycle;

\draw[fill=black] (0+30,0+12+30)--(1+30,0+12+30)--(1+30,1+12+30)--(0+30,1+12+30)--cycle;
\draw[fill=black] (2+30,0+12+30)--(3+30,0+12+30)--(3+30,1+12+30)--(2+30,1+12+30)--cycle;
\draw[fill=black] (1+30,1+12+30)--(2+30,1+12+30)--(2+30,2+12+30)--(1+30,2+12+30)--cycle;
\draw[fill=black] (0+30,2+12+30)--(1+30,2+12+30)--(1+30,3+12+30)--(0+30,3+12+30)--cycle;
\draw[fill=black] (2+30,2+12+30)--(3+30,2+12+30)--(3+30,3+12+30)--(2+30,3+12+30)--cycle;

\draw[fill=black] (0+12+30,0+30)--(1+12+30,0+30)--(1+12+30,1+30)--(0+12+30,1+30)--cycle;
\draw[fill=black] (2+12+30,0+30)--(3+12+30,0+30)--(3+12+30,1+30)--(2+12+30,1+30)--cycle;
\draw[fill=black] (1+12+30,1+30)--(2+12+30,1+30)--(2+12+30,2+30)--(1+12+30,2+30)--cycle;
\draw[fill=black] (0+12+30,2+30)--(1+12+30,2+30)--(1+12+30,3+30)--(0+12+30,3+30)--cycle;
\draw[fill=black] (2+12+30,2+30)--(3+12+30,2+30)--(3+12+30,3+30)--(2+12+30,3+30)--cycle;

\draw[fill=black] (0+12+30,0+12+30)--(1+12+30,0+12+30)--(1+12+30,1+12+30)--(0+12+30,1+12+30)--cycle;
\draw[fill=black] (2+12+30,0+12+30)--(3+12+30,0+12+30)--(3+12+30,1+12+30)--(2+12+30,1+12+30)--cycle;
\draw[fill=black] (1+12+30,1+12+30)--(2+12+30,1+12+30)--(2+12+30,2+12+30)--(1+12+30,2+12+30)--cycle;
\draw[fill=black] (0+12+30,2+12+30)--(1+12+30,2+12+30)--(1+12+30,3+12+30)--(0+12+30,3+12+30)--cycle;
\draw[fill=black] (2+12+30,2+12+30)--(3+12+30,2+12+30)--(3+12+30,3+12+30)--(2+12+30,3+12+30)--cycle;

\draw[fill=black] (0+15,0+15)--(1+15,0+15)--(1+15,1+15)--(0+15,1+15)--cycle;
\draw[fill=black] (2+15,0+15)--(3+15,0+15)--(3+15,1+15)--(2+15,1+15)--cycle;
\draw[fill=black] (1+15,1+15)--(2+15,1+15)--(2+15,2+15)--(1+15,2+15)--cycle;
\draw[fill=black] (0+15,2+15)--(1+15,2+15)--(1+15,3+15)--(0+15,3+15)--cycle;
\draw[fill=black] (2+15,2+15)--(3+15,2+15)--(3+15,3+15)--(2+15,3+15)--cycle;

\draw[fill=black] (0+6+15,0+6+15)--(1+6+15,0+6+15)--(1+6+15,1+6+15)--(0+6+15,1+6+15)--cycle;
\draw[fill=black] (2+6+15,0+6+15)--(3+6+15,0+6+15)--(3+6+15,1+6+15)--(2+6+15,1+6+15)--cycle;
\draw[fill=black] (1+6+15,1+6+15)--(2+6+15,1+6+15)--(2+6+15,2+6+15)--(1+6+15,2+6+15)--cycle;
\draw[fill=black] (0+6+15,2+6+15)--(1+6+15,2+6+15)--(1+6+15,3+6+15)--(0+6+15,3+6+15)--cycle;
\draw[fill=black] (2+6+15,2+6+15)--(3+6+15,2+6+15)--(3+6+15,3+6+15)--(2+6+15,3+6+15)--cycle;

\draw[fill=black] (0+3+15,0+3+15)--(1+3+15,0+3+15)--(1+3+15,1+3+15)--(0+3+15,1+3+15)--cycle;
\draw[fill=black] (2+3+15,0+3+15)--(3+3+15,0+3+15)--(3+3+15,1+3+15)--(2+3+15,1+3+15)--cycle;
\draw[fill=black] (1+3+15,1+3+15)--(2+3+15,1+3+15)--(2+3+15,2+3+15)--(1+3+15,2+3+15)--cycle;
\draw[fill=black] (0+3+15,2+3+15)--(1+3+15,2+3+15)--(1+3+15,3+3+15)--(0+3+15,3+3+15)--cycle;
\draw[fill=black] (2+3+15,2+3+15)--(3+3+15,2+3+15)--(3+3+15,3+3+15)--(2+3+15,3+3+15)--cycle;

\draw[fill=black] (0+9+15,0+3+15)--(1+9+15,0+3+15)--(1+9+15,1+3+15)--(0+9+15,1+3+15)--cycle;
\draw[fill=black] (2+9+15,0+3+15)--(3+9+15,0+3+15)--(3+9+15,1+3+15)--(2+9+15,1+3+15)--cycle;
\draw[fill=black] (1+9+15,1+3+15)--(2+9+15,1+3+15)--(2+9+15,2+3+15)--(1+9+15,2+3+15)--cycle;
\draw[fill=black] (0+9+15,2+3+15)--(1+9+15,2+3+15)--(1+9+15,3+3+15)--(0+9+15,3+3+15)--cycle;
\draw[fill=black] (2+9+15,2+3+15)--(3+9+15,2+3+15)--(3+9+15,3+3+15)--(2+9+15,3+3+15)--cycle;

\draw[fill=black] (0+3+15,0+9+15)--(1+3+15,0+9+15)--(1+3+15,1+9+15)--(0+3+15,1+9+15)--cycle;
\draw[fill=black] (2+3+15,0+9+15)--(3+3+15,0+9+15)--(3+3+15,1+9+15)--(2+3+15,1+9+15)--cycle;
\draw[fill=black] (1+3+15,1+9+15)--(2+3+15,1+9+15)--(2+3+15,2+9+15)--(1+3+15,2+9+15)--cycle;
\draw[fill=black] (0+3+15,2+9+15)--(1+3+15,2+9+15)--(1+3+15,3+9+15)--(0+3+15,3+9+15)--cycle;
\draw[fill=black] (2+3+15,2+9+15)--(3+3+15,2+9+15)--(3+3+15,3+9+15)--(2+3+15,3+9+15)--cycle;

\draw[fill=black] (0+9+15,0+9+15)--(1+9+15,0+9+15)--(1+9+15,1+9+15)--(0+9+15,1+9+15)--cycle;
\draw[fill=black] (2+9+15,0+9+15)--(3+9+15,0+9+15)--(3+9+15,1+9+15)--(2+9+15,1+9+15)--cycle;
\draw[fill=black] (1+9+15,1+9+15)--(2+9+15,1+9+15)--(2+9+15,2+9+15)--(1+9+15,2+9+15)--cycle;
\draw[fill=black] (0+9+15,2+9+15)--(1+9+15,2+9+15)--(1+9+15,3+9+15)--(0+9+15,3+9+15)--cycle;
\draw[fill=black] (2+9+15,2+9+15)--(3+9+15,2+9+15)--(3+9+15,3+9+15)--(2+9+15,3+9+15)--cycle;

\draw[fill=black] (0+15,0+12+15)--(1+15,0+12+15)--(1+15,1+12+15)--(0+15,1+12+15)--cycle;
\draw[fill=black] (2+15,0+12+15)--(3+15,0+12+15)--(3+15,1+12+15)--(2+15,1+12+15)--cycle;
\draw[fill=black] (1+15,1+12+15)--(2+15,1+12+15)--(2+15,2+12+15)--(1+15,2+12+15)--cycle;
\draw[fill=black] (0+15,2+12+15)--(1+15,2+12+15)--(1+15,3+12+15)--(0+15,3+12+15)--cycle;
\draw[fill=black] (2+15,2+12+15)--(3+15,2+12+15)--(3+15,3+12+15)--(2+15,3+12+15)--cycle;

\draw[fill=black] (0+12+15,0+15)--(1+12+15,0+15)--(1+12+15,1+15)--(0+12+15,1+15)--cycle;
\draw[fill=black] (2+12+15,0+15)--(3+12+15,0+15)--(3+12+15,1+15)--(2+12+15,1+15)--cycle;
\draw[fill=black] (1+12+15,1+15)--(2+12+15,1+15)--(2+12+15,2+15)--(1+12+15,2+15)--cycle;
\draw[fill=black] (0+12+15,2+15)--(1+12+15,2+15)--(1+12+15,3+15)--(0+12+15,3+15)--cycle;
\draw[fill=black] (2+12+15,2+15)--(3+12+15,2+15)--(3+12+15,3+15)--(2+12+15,3+15)--cycle;

\draw[fill=black] (0+12+15,0+12+15)--(1+12+15,0+12+15)--(1+12+15,1+12+15)--(0+12+15,1+12+15)--cycle;
\draw[fill=black] (2+12+15,0+12+15)--(3+12+15,0+12+15)--(3+12+15,1+12+15)--(2+12+15,1+12+15)--cycle;
\draw[fill=black] (1+12+15,1+12+15)--(2+12+15,1+12+15)--(2+12+15,2+12+15)--(1+12+15,2+12+15)--cycle;
\draw[fill=black] (0+12+15,2+12+15)--(1+12+15,2+12+15)--(1+12+15,3+12+15)--(0+12+15,3+12+15)--cycle;
\draw[fill=black] (2+12+15,2+12+15)--(3+12+15,2+12+15)--(3+12+15,3+12+15)--(2+12+15,3+12+15)--cycle;
\end{tikzpicture}	

}\hspace{40pt}
\subfloat[$K^{\bm{l}}$ at level $3$, with $\bm{l}=533\cdots$.]{
\begin{tikzpicture}[scale=0.13]
\draw[fill=black] (0,0)--(1,0)--(1,1)--(0,1)--cycle;
\draw[fill=black] (2,0)--(3,0)--(3,1)--(2,1)--cycle;
\draw[fill=black] (1,1)--(2,1)--(2,2)--(1,2)--cycle;
\draw[fill=black] (0,2)--(1,2)--(1,3)--(0,3)--cycle;
\draw[fill=black] (2,2)--(3,2)--(3,3)--(2,3)--cycle;

\draw[fill=black] (0,0+6)--(1,0+6)--(1,1+6)--(0,1+6)--cycle;
\draw[fill=black] (2,0+6)--(3,0+6)--(3,1+6)--(2,1+6)--cycle;
\draw[fill=black] (1,1+6)--(2,1+6)--(2,2+6)--(1,2+6)--cycle;
\draw[fill=black] (0,2+6)--(1,2+6)--(1,3+6)--(0,3+6)--cycle;
\draw[fill=black] (2,2+6)--(3,2+6)--(3,3+6)--(2,3+6)--cycle;

\draw[fill=black] (0+6,0+6)--(1+6,0+6)--(1+6,1+6)--(0+6,1+6)--cycle;
\draw[fill=black] (2+6,0+6)--(3+6,0+6)--(3+6,1+6)--(2+6,1+6)--cycle;
\draw[fill=black] (1+6,1+6)--(2+6,1+6)--(2+6,2+6)--(1+6,2+6)--cycle;
\draw[fill=black] (0+6,2+6)--(1+6,2+6)--(1+6,3+6)--(0+6,3+6)--cycle;
\draw[fill=black] (2+6,2+6)--(3+6,2+6)--(3+6,3+6)--(2+6,3+6)--cycle;

\draw[fill=black] (0+3,0+3)--(1+3,0+3)--(1+3,1+3)--(0+3,1+3)--cycle;
\draw[fill=black] (2+3,0+3)--(3+3,0+3)--(3+3,1+3)--(2+3,1+3)--cycle;
\draw[fill=black] (1+3,1+3)--(2+3,1+3)--(2+3,2+3)--(1+3,2+3)--cycle;
\draw[fill=black] (0+3,2+3)--(1+3,2+3)--(1+3,3+3)--(0+3,3+3)--cycle;
\draw[fill=black] (2+3,2+3)--(3+3,2+3)--(3+3,3+3)--(2+3,3+3)--cycle;

\draw[fill=black] (6,0)--(7,0)--(7,1)--(6,1)--cycle;
\draw[fill=black] (8,0)--(9,0)--(9,1)--(8,1)--cycle;
\draw[fill=black] (7,1)--(8,1)--(8,2)--(7,2)--cycle;
\draw[fill=black] (6,2)--(7,2)--(7,3)--(6,3)--cycle;
\draw[fill=black] (8,2)--(9,2)--(9,3)--(8,3)--cycle;

\draw[fill=black] (0+36,0)--(1+36,0)--(1+36,1)--(0+36,1)--cycle;
\draw[fill=black] (2+36,0)--(3+36,0)--(3+36,1)--(2+36,1)--cycle;
\draw[fill=black] (1+36,1)--(2+36,1)--(2+36,2)--(1+36,2)--cycle;
\draw[fill=black] (0+36,2)--(1+36,2)--(1+36,3)--(0+36,3)--cycle;
\draw[fill=black] (2+36,2)--(3+36,2)--(3+36,3)--(2+36,3)--cycle;

\draw[fill=black] (0+36,0+6)--(1+36,0+6)--(1+36,1+6)--(0+36,1+6)--cycle;
\draw[fill=black] (2+36,0+6)--(3+36,0+6)--(3+36,1+6)--(2+36,1+6)--cycle;
\draw[fill=black] (1+36,1+6)--(2+36,1+6)--(2+36,2+6)--(1+36,2+6)--cycle;
\draw[fill=black] (0+36,2+6)--(1+36,2+6)--(1+36,3+6)--(0+36,3+6)--cycle;
\draw[fill=black] (2+36,2+6)--(3+36,2+6)--(3+36,3+6)--(2+36,3+6)--cycle;

\draw[fill=black] (0+6+36,0+6)--(1+6+36,0+6)--(1+6+36,1+6)--(0+6+36,1+6)--cycle;
\draw[fill=black] (2+6+36,0+6)--(3+6+36,0+6)--(3+6+36,1+6)--(2+6+36,1+6)--cycle;
\draw[fill=black] (1+6+36,1+6)--(2+6+36,1+6)--(2+6+36,2+6)--(1+6+36,2+6)--cycle;
\draw[fill=black] (0+6+36,2+6)--(1+6+36,2+6)--(1+6+36,3+6)--(0+6+36,3+6)--cycle;
\draw[fill=black] (2+6+36,2+6)--(3+6+36,2+6)--(3+6+36,3+6)--(2+6+36,3+6)--cycle;

\draw[fill=black] (0+3+36,0+3)--(1+3+36,0+3)--(1+3+36,1+3)--(0+3+36,1+3)--cycle;
\draw[fill=black] (2+3+36,0+3)--(3+3+36,0+3)--(3+3+36,1+3)--(2+3+36,1+3)--cycle;
\draw[fill=black] (1+3+36,1+3)--(2+3+36,1+3)--(2+3+36,2+3)--(1+3+36,2+3)--cycle;
\draw[fill=black] (0+3+36,2+3)--(1+3+36,2+3)--(1+3+36,3+3)--(0+3+36,3+3)--cycle;
\draw[fill=black] (2+3+36,2+3)--(3+3+36,2+3)--(3+3+36,3+3)--(2+3+36,3+3)--cycle;

\draw[fill=black] (6+36,0)--(7+36,0)--(7+36,1)--(6+36,1)--cycle;
\draw[fill=black] (8+36,0)--(9+36,0)--(9+36,1)--(8+36,1)--cycle;
\draw[fill=black] (7+36,1)--(8+36,1)--(8+36,2)--(7+36,2)--cycle;
\draw[fill=black] (6+36,2)--(7+36,2)--(7+36,3)--(6+36,3)--cycle;
\draw[fill=black] (8+36,2)--(9+36,2)--(9+36,3)--(8+36,3)--cycle;

\draw[fill=black] (0+36,0+36)--(1+36,0+36)--(1+36,1+36)--(0+36,1+36)--cycle;
\draw[fill=black] (2+36,0+36)--(3+36,0+36)--(3+36,1+36)--(2+36,1+36)--cycle;
\draw[fill=black] (1+36,1+36)--(2+36,1+36)--(2+36,2+36)--(1+36,2+36)--cycle;
\draw[fill=black] (0+36,2+36)--(1+36,2+36)--(1+36,3+36)--(0+36,3+36)--cycle;
\draw[fill=black] (2+36,2+36)--(3+36,2+36)--(3+36,3+36)--(2+36,3+36)--cycle;

\draw[fill=black] (0+36,0+6+36)--(1+36,0+6+36)--(1+36,1+6+36)--(0+36,1+6+36)--cycle;
\draw[fill=black] (2+36,0+6+36)--(3+36,0+6+36)--(3+36,1+6+36)--(2+36,1+6+36)--cycle;
\draw[fill=black] (1+36,1+6+36)--(2+36,1+6+36)--(2+36,2+6+36)--(1+36,2+6+36)--cycle;
\draw[fill=black] (0+36,2+6+36)--(1+36,2+6+36)--(1+36,3+6+36)--(0+36,3+6+36)--cycle;
\draw[fill=black] (2+36,2+6+36)--(3+36,2+6+36)--(3+36,3+6+36)--(2+36,3+6+36)--cycle;

\draw[fill=black] (0+6+36,0+6+36)--(1+6+36,0+6+36)--(1+6+36,1+6+36)--(0+6+36,1+6+36)--cycle;
\draw[fill=black] (2+6+36,0+6+36)--(3+6+36,0+6+36)--(3+6+36,1+6+36)--(2+6+36,1+6+36)--cycle;
\draw[fill=black] (1+6+36,1+6+36)--(2+6+36,1+6+36)--(2+6+36,2+6+36)--(1+6+36,2+6+36)--cycle;
\draw[fill=black] (0+6+36,2+6+36)--(1+6+36,2+6+36)--(1+6+36,3+6+36)--(0+6+36,3+6+36)--cycle;
\draw[fill=black] (2+6+36,2+6+36)--(3+6+36,2+6+36)--(3+6+36,3+6+36)--(2+6+36,3+6+36)--cycle;

\draw[fill=black] (0+3+36,0+3+36)--(1+3+36,0+3+36)--(1+3+36,1+3+36)--(0+3+36,1+3+36)--cycle;
\draw[fill=black] (2+3+36,0+3+36)--(3+3+36,0+3+36)--(3+3+36,1+3+36)--(2+3+36,1+3+36)--cycle;
\draw[fill=black] (1+3+36,1+3+36)--(2+3+36,1+3+36)--(2+3+36,2+3+36)--(1+3+36,2+3+36)--cycle;
\draw[fill=black] (0+3+36,2+3+36)--(1+3+36,2+3+36)--(1+3+36,3+3+36)--(0+3+36,3+3+36)--cycle;
\draw[fill=black] (2+3+36,2+3+36)--(3+3+36,2+3+36)--(3+3+36,3+3+36)--(2+3+36,3+3+36)--cycle;

\draw[fill=black] (6+36,0+36)--(7+36,0+36)--(7+36,1+36)--(6+36,1+36)--cycle;
\draw[fill=black] (8+36,0+36)--(9+36,0+36)--(9+36,1+36)--(8+36,1+36)--cycle;
\draw[fill=black] (7+36,1+36)--(8+36,1+36)--(8+36,2+36)--(7+36,2+36)--cycle;
\draw[fill=black] (6+36,2+36)--(7+36,2+36)--(7+36,3+36)--(6+36,3+36)--cycle;
\draw[fill=black] (8+36,2+36)--(9+36,2+36)--(9+36,3+36)--(8+36,3+36)--cycle;

\draw[fill=black] (0,0+36)--(1,0+36)--(1,1+36)--(0,1+36)--cycle;
\draw[fill=black] (2,0+36)--(3,0+36)--(3,1+36)--(2,1+36)--cycle;
\draw[fill=black] (1,1+36)--(2,1+36)--(2,2+36)--(1,2+36)--cycle;
\draw[fill=black] (0,2+36)--(1,2+36)--(1,3+36)--(0,3+36)--cycle;
\draw[fill=black] (2,2+36)--(3,2+36)--(3,3+36)--(2,3+36)--cycle;

\draw[fill=black] (0,0+6+36)--(1,0+6+36)--(1,1+6+36)--(0,1+6+36)--cycle;
\draw[fill=black] (2,0+6+36)--(3,0+6+36)--(3,1+6+36)--(2,1+6+36)--cycle;
\draw[fill=black] (1,1+6+36)--(2,1+6+36)--(2,2+6+36)--(1,2+6+36)--cycle;
\draw[fill=black] (0,2+6+36)--(1,2+6+36)--(1,3+6+36)--(0,3+6+36)--cycle;
\draw[fill=black] (2,2+6+36)--(3,2+6+36)--(3,3+6+36)--(2,3+6+36)--cycle;

\draw[fill=black] (0+6,0+6+36)--(1+6,0+6+36)--(1+6,1+6+36)--(0+6,1+6+36)--cycle;
\draw[fill=black] (2+6,0+6+36)--(3+6,0+6+36)--(3+6,1+6+36)--(2+6,1+6+36)--cycle;
\draw[fill=black] (1+6,1+6+36)--(2+6,1+6+36)--(2+6,2+6+36)--(1+6,2+6+36)--cycle;
\draw[fill=black] (0+6,2+6+36)--(1+6,2+6+36)--(1+6,3+6+36)--(0+6,3+6+36)--cycle;
\draw[fill=black] (2+6,2+6+36)--(3+6,2+6+36)--(3+6,3+6+36)--(2+6,3+6+36)--cycle;

\draw[fill=black] (0+3,0+3+36)--(1+3,0+3+36)--(1+3,1+3+36)--(0+3,1+3+36)--cycle;
\draw[fill=black] (2+3,0+3+36)--(3+3,0+3+36)--(3+3,1+3+36)--(2+3,1+3+36)--cycle;
\draw[fill=black] (1+3,1+3+36)--(2+3,1+3+36)--(2+3,2+3+36)--(1+3,2+3+36)--cycle;
\draw[fill=black] (0+3,2+3+36)--(1+3,2+3+36)--(1+3,3+3+36)--(0+3,3+3+36)--cycle;
\draw[fill=black] (2+3,2+3+36)--(3+3,2+3+36)--(3+3,3+3+36)--(2+3,3+3+36)--cycle;

\draw[fill=black] (6,0+36)--(7,0+36)--(7,1+36)--(6,1+36)--cycle;
\draw[fill=black] (8,0+36)--(9,0+36)--(9,1+36)--(8,1+36)--cycle;
\draw[fill=black] (7,1+36)--(8,1+36)--(8,2+36)--(7,2+36)--cycle;
\draw[fill=black] (6,2+36)--(7,2+36)--(7,3+36)--(6,3+36)--cycle;
\draw[fill=black] (8,2+36)--(9,2+36)--(9,3+36)--(8,3+36)--cycle;

\draw[fill=black] (0+9,0+9)--(1+9,0+9)--(1+9,1+9)--(0+9,1+9)--cycle;
\draw[fill=black] (2+9,0+9)--(3+9,0+9)--(3+9,1+9)--(2+9,1+9)--cycle;
\draw[fill=black] (1+9,1+9)--(2+9,1+9)--(2+9,2+9)--(1+9,2+9)--cycle;
\draw[fill=black] (0+9,2+9)--(1+9,2+9)--(1+9,3+9)--(0+9,3+9)--cycle;
\draw[fill=black] (2+9,2+9)--(3+9,2+9)--(3+9,3+9)--(2+9,3+9)--cycle;

\draw[fill=black] (0+9,0+6+9)--(1+9,0+6+9)--(1+9,1+6+9)--(0+9,1+6+9)--cycle;
\draw[fill=black] (2+9,0+6+9)--(3+9,0+6+9)--(3+9,1+6+9)--(2+9,1+6+9)--cycle;
\draw[fill=black] (1+9,1+6+9)--(2+9,1+6+9)--(2+9,2+6+9)--(1+9,2+6+9)--cycle;
\draw[fill=black] (0+9,2+6+9)--(1+9,2+6+9)--(1+9,3+6+9)--(0+9,3+6+9)--cycle;
\draw[fill=black] (2+9,2+6+9)--(3+9,2+6+9)--(3+9,3+6+9)--(2+9,3+6+9)--cycle;

\draw[fill=black] (0+6+9,0+6+9)--(1+6+9,0+6+9)--(1+6+9,1+6+9)--(0+6+9,1+6+9)--cycle;
\draw[fill=black] (2+6+9,0+6+9)--(3+6+9,0+6+9)--(3+6+9,1+6+9)--(2+6+9,1+6+9)--cycle;
\draw[fill=black] (1+6+9,1+6+9)--(2+6+9,1+6+9)--(2+6+9,2+6+9)--(1+6+9,2+6+9)--cycle;
\draw[fill=black] (0+6+9,2+6+9)--(1+6+9,2+6+9)--(1+6+9,3+6+9)--(0+6+9,3+6+9)--cycle;
\draw[fill=black] (2+6+9,2+6+9)--(3+6+9,2+6+9)--(3+6+9,3+6+9)--(2+6+9,3+6+9)--cycle;

\draw[fill=black] (0+3+9,0+3+9)--(1+3+9,0+3+9)--(1+3+9,1+3+9)--(0+3+9,1+3+9)--cycle;
\draw[fill=black] (2+3+9,0+3+9)--(3+3+9,0+3+9)--(3+3+9,1+3+9)--(2+3+9,1+3+9)--cycle;
\draw[fill=black] (1+3+9,1+3+9)--(2+3+9,1+3+9)--(2+3+9,2+3+9)--(1+3+9,2+3+9)--cycle;
\draw[fill=black] (0+3+9,2+3+9)--(1+3+9,2+3+9)--(1+3+9,3+3+9)--(0+3+9,3+3+9)--cycle;
\draw[fill=black] (2+3+9,2+3+9)--(3+3+9,2+3+9)--(3+3+9,3+3+9)--(2+3+9,3+3+9)--cycle;

\draw[fill=black] (6+9,0+9)--(7+9,0+9)--(7+9,1+9)--(6+9,1+9)--cycle;
\draw[fill=black] (8+9,0+9)--(9+9,0+9)--(9+9,1+9)--(8+9,1+9)--cycle;
\draw[fill=black] (7+9,1+9)--(8+9,1+9)--(8+9,2+9)--(7+9,2+9)--cycle;
\draw[fill=black] (6+9,2+9)--(7+9,2+9)--(7+9,3+9)--(6+9,3+9)--cycle;
\draw[fill=black] (8+9,2+9)--(9+9,2+9)--(9+9,3+9)--(8+9,3+9)--cycle;

\draw[fill=black] (0+9,0+27)--(1+9,0+27)--(1+9,1+27)--(0+9,1+27)--cycle;
\draw[fill=black] (2+9,0+27)--(3+9,0+27)--(3+9,1+27)--(2+9,1+27)--cycle;
\draw[fill=black] (1+9,1+27)--(2+9,1+27)--(2+9,2+27)--(1+9,2+27)--cycle;
\draw[fill=black] (0+9,2+27)--(1+9,2+27)--(1+9,3+27)--(0+9,3+27)--cycle;
\draw[fill=black] (2+9,2+27)--(3+9,2+27)--(3+9,3+27)--(2+9,3+27)--cycle;

\draw[fill=black] (0+9,0+6+27)--(1+9,0+6+27)--(1+9,1+6+27)--(0+9,1+6+27)--cycle;
\draw[fill=black] (2+9,0+6+27)--(3+9,0+6+27)--(3+9,1+6+27)--(2+9,1+6+27)--cycle;
\draw[fill=black] (1+9,1+6+27)--(2+9,1+6+27)--(2+9,2+6+27)--(1+9,2+6+27)--cycle;
\draw[fill=black] (0+9,2+6+27)--(1+9,2+6+27)--(1+9,3+6+27)--(0+9,3+6+27)--cycle;
\draw[fill=black] (2+9,2+6+27)--(3+9,2+6+27)--(3+9,3+6+27)--(2+9,3+6+27)--cycle;

\draw[fill=black] (0+6+9,0+6+27)--(1+6+9,0+6+27)--(1+6+9,1+6+27)--(0+6+9,1+6+27)--cycle;
\draw[fill=black] (2+6+9,0+6+27)--(3+6+9,0+6+27)--(3+6+9,1+6+27)--(2+6+9,1+6+27)--cycle;
\draw[fill=black] (1+6+9,1+6+27)--(2+6+9,1+6+27)--(2+6+9,2+6+27)--(1+6+9,2+6+27)--cycle;
\draw[fill=black] (0+6+9,2+6+27)--(1+6+9,2+6+27)--(1+6+9,3+6+27)--(0+6+9,3+6+27)--cycle;
\draw[fill=black] (2+6+9,2+6+27)--(3+6+9,2+6+27)--(3+6+9,3+6+27)--(2+6+9,3+6+27)--cycle;

\draw[fill=black] (0+3+9,0+3+27)--(1+3+9,0+3+27)--(1+3+9,1+3+27)--(0+3+9,1+3+27)--cycle;
\draw[fill=black] (2+3+9,0+3+27)--(3+3+9,0+3+27)--(3+3+9,1+3+27)--(2+3+9,1+3+27)--cycle;
\draw[fill=black] (1+3+9,1+3+27)--(2+3+9,1+3+27)--(2+3+9,2+3+27)--(1+3+9,2+3+27)--cycle;
\draw[fill=black] (0+3+9,2+3+27)--(1+3+9,2+3+27)--(1+3+9,3+3+27)--(0+3+9,3+3+27)--cycle;
\draw[fill=black] (2+3+9,2+3+27)--(3+3+9,2+3+27)--(3+3+9,3+3+27)--(2+3+9,3+3+27)--cycle;

\draw[fill=black] (6+9,0+27)--(7+9,0+27)--(7+9,1+27)--(6+9,1+27)--cycle;
\draw[fill=black] (8+9,0+27)--(9+9,0+27)--(9+9,1+27)--(8+9,1+27)--cycle;
\draw[fill=black] (7+9,1+27)--(8+9,1+27)--(8+9,2+27)--(7+9,2+27)--cycle;
\draw[fill=black] (6+9,2+27)--(7+9,2+27)--(7+9,3+27)--(6+9,3+27)--cycle;
\draw[fill=black] (8+9,2+27)--(9+9,2+27)--(9+9,3+27)--(8+9,3+27)--cycle;

\draw[fill=black] (0+27,0+9)--(1+27,0+9)--(1+27,1+9)--(0+27,1+9)--cycle;
\draw[fill=black] (2+27,0+9)--(3+27,0+9)--(3+27,1+9)--(2+27,1+9)--cycle;
\draw[fill=black] (1+27,1+9)--(2+27,1+9)--(2+27,2+9)--(1+27,2+9)--cycle;
\draw[fill=black] (0+27,2+9)--(1+27,2+9)--(1+27,3+9)--(0+27,3+9)--cycle;
\draw[fill=black] (2+27,2+9)--(3+27,2+9)--(3+27,3+9)--(2+27,3+9)--cycle;

\draw[fill=black] (0+27,0+6+9)--(1+27,0+6+9)--(1+27,1+6+9)--(0+27,1+6+9)--cycle;
\draw[fill=black] (2+27,0+6+9)--(3+27,0+6+9)--(3+27,1+6+9)--(2+27,1+6+9)--cycle;
\draw[fill=black] (1+27,1+6+9)--(2+27,1+6+9)--(2+27,2+6+9)--(1+27,2+6+9)--cycle;
\draw[fill=black] (0+27,2+6+9)--(1+27,2+6+9)--(1+27,3+6+9)--(0+27,3+6+9)--cycle;
\draw[fill=black] (2+27,2+6+9)--(3+27,2+6+9)--(3+27,3+6+9)--(2+27,3+6+9)--cycle;

\draw[fill=black] (0+6+27,0+6+9)--(1+6+27,0+6+9)--(1+6+27,1+6+9)--(0+6+27,1+6+9)--cycle;
\draw[fill=black] (2+6+27,0+6+9)--(3+6+27,0+6+9)--(3+6+27,1+6+9)--(2+6+27,1+6+9)--cycle;
\draw[fill=black] (1+6+27,1+6+9)--(2+6+27,1+6+9)--(2+6+27,2+6+9)--(1+6+27,2+6+9)--cycle;
\draw[fill=black] (0+6+27,2+6+9)--(1+6+27,2+6+9)--(1+6+27,3+6+9)--(0+6+27,3+6+9)--cycle;
\draw[fill=black] (2+6+27,2+6+9)--(3+6+27,2+6+9)--(3+6+27,3+6+9)--(2+6+27,3+6+9)--cycle;

\draw[fill=black] (0+3+27,0+3+9)--(1+3+27,0+3+9)--(1+3+27,1+3+9)--(0+3+27,1+3+9)--cycle;
\draw[fill=black] (2+3+27,0+3+9)--(3+3+27,0+3+9)--(3+3+27,1+3+9)--(2+3+27,1+3+9)--cycle;
\draw[fill=black] (1+3+27,1+3+9)--(2+3+27,1+3+9)--(2+3+27,2+3+9)--(1+3+27,2+3+9)--cycle;
\draw[fill=black] (0+3+27,2+3+9)--(1+3+27,2+3+9)--(1+3+27,3+3+9)--(0+3+27,3+3+9)--cycle;
\draw[fill=black] (2+3+27,2+3+9)--(3+3+27,2+3+9)--(3+3+27,3+3+9)--(2+3+27,3+3+9)--cycle;

\draw[fill=black] (6+27,0+9)--(7+27,0+9)--(7+27,1+9)--(6+27,1+9)--cycle;
\draw[fill=black] (8+27,0+9)--(9+27,0+9)--(9+27,1+9)--(8+27,1+9)--cycle;
\draw[fill=black] (7+27,1+9)--(8+27,1+9)--(8+27,2+9)--(7+27,2+9)--cycle;
\draw[fill=black] (6+27,2+9)--(7+27,2+9)--(7+27,3+9)--(6+27,3+9)--cycle;
\draw[fill=black] (8+27,2+9)--(9+27,2+9)--(9+27,3+9)--(8+27,3+9)--cycle;

\draw[fill=black] (0+18,0+18)--(1+18,0+18)--(1+18,1+18)--(0+18,1+18)--cycle;
\draw[fill=black] (2+18,0+18)--(3+18,0+18)--(3+18,1+18)--(2+18,1+18)--cycle;
\draw[fill=black] (1+18,1+18)--(2+18,1+18)--(2+18,2+18)--(1+18,2+18)--cycle;
\draw[fill=black] (0+18,2+18)--(1+18,2+18)--(1+18,3+18)--(0+18,3+18)--cycle;
\draw[fill=black] (2+18,2+18)--(3+18,2+18)--(3+18,3+18)--(2+18,3+18)--cycle;

\draw[fill=black] (0+18,0+6+18)--(1+18,0+6+18)--(1+18,1+6+18)--(0+18,1+6+18)--cycle;
\draw[fill=black] (2+18,0+6+18)--(3+18,0+6+18)--(3+18,1+6+18)--(2+18,1+6+18)--cycle;
\draw[fill=black] (1+18,1+6+18)--(2+18,1+6+18)--(2+18,2+6+18)--(1+18,2+6+18)--cycle;
\draw[fill=black] (0+18,2+6+18)--(1+18,2+6+18)--(1+18,3+6+18)--(0+18,3+6+18)--cycle;
\draw[fill=black] (2+18,2+6+18)--(3+18,2+6+18)--(3+18,3+6+18)--(2+18,3+6+18)--cycle;

\draw[fill=black] (0+6+18,0+6+18)--(1+6+18,0+6+18)--(1+6+18,1+6+18)--(0+6+18,1+6+18)--cycle;
\draw[fill=black] (2+6+18,0+6+18)--(3+6+18,0+6+18)--(3+6+18,1+6+18)--(2+6+18,1+6+18)--cycle;
\draw[fill=black] (1+6+18,1+6+18)--(2+6+18,1+6+18)--(2+6+18,2+6+18)--(1+6+18,2+6+18)--cycle;
\draw[fill=black] (0+6+18,2+6+18)--(1+6+18,2+6+18)--(1+6+18,3+6+18)--(0+6+18,3+6+18)--cycle;
\draw[fill=black] (2+6+18,2+6+18)--(3+6+18,2+6+18)--(3+6+18,3+6+18)--(2+6+18,3+6+18)--cycle;

\draw[fill=black] (0+3+18,0+3+18)--(1+3+18,0+3+18)--(1+3+18,1+3+18)--(0+3+18,1+3+18)--cycle;
\draw[fill=black] (2+3+18,0+3+18)--(3+3+18,0+3+18)--(3+3+18,1+3+18)--(2+3+18,1+3+18)--cycle;
\draw[fill=black] (1+3+18,1+3+18)--(2+3+18,1+3+18)--(2+3+18,2+3+18)--(1+3+18,2+3+18)--cycle;
\draw[fill=black] (0+3+18,2+3+18)--(1+3+18,2+3+18)--(1+3+18,3+3+18)--(0+3+18,3+3+18)--cycle;
\draw[fill=black] (2+3+18,2+3+18)--(3+3+18,2+3+18)--(3+3+18,3+3+18)--(2+3+18,3+3+18)--cycle;

\draw[fill=black] (6+18,0+18)--(7+18,0+18)--(7+18,1+18)--(6+18,1+18)--cycle;
\draw[fill=black] (8+18,0+18)--(9+18,0+18)--(9+18,1+18)--(8+18,1+18)--cycle;
\draw[fill=black] (7+18,1+18)--(8+18,1+18)--(8+18,2+18)--(7+18,2+18)--cycle;
\draw[fill=black] (6+18,2+18)--(7+18,2+18)--(7+18,3+18)--(6+18,3+18)--cycle;
\draw[fill=black] (8+18,2+18)--(9+18,2+18)--(9+18,3+18)--(8+18,3+18)--cycle;

\draw[fill=black] (0+27,0+27)--(1+27,0+27)--(1+27,1+27)--(0+27,1+27)--cycle;
\draw[fill=black] (2+27,0+27)--(3+27,0+27)--(3+27,1+27)--(2+27,1+27)--cycle;
\draw[fill=black] (1+27,1+27)--(2+27,1+27)--(2+27,2+27)--(1+27,2+27)--cycle;
\draw[fill=black] (0+27,2+27)--(1+27,2+27)--(1+27,3+27)--(0+27,3+27)--cycle;
\draw[fill=black] (2+27,2+27)--(3+27,2+27)--(3+27,3+27)--(2+27,3+27)--cycle;

\draw[fill=black] (0+27,0+6+27)--(1+27,0+6+27)--(1+27,1+6+27)--(0+27,1+6+27)--cycle;
\draw[fill=black] (2+27,0+6+27)--(3+27,0+6+27)--(3+27,1+6+27)--(2+27,1+6+27)--cycle;
\draw[fill=black] (1+27,1+6+27)--(2+27,1+6+27)--(2+27,2+6+27)--(1+27,2+6+27)--cycle;
\draw[fill=black] (0+27,2+6+27)--(1+27,2+6+27)--(1+27,3+6+27)--(0+27,3+6+27)--cycle;
\draw[fill=black] (2+27,2+6+27)--(3+27,2+6+27)--(3+27,3+6+27)--(2+27,3+6+27)--cycle;

\draw[fill=black] (0+6+27,0+6+27)--(1+6+27,0+6+27)--(1+6+27,1+6+27)--(0+6+27,1+6+27)--cycle;
\draw[fill=black] (2+6+27,0+6+27)--(3+6+27,0+6+27)--(3+6+27,1+6+27)--(2+6+27,1+6+27)--cycle;
\draw[fill=black] (1+6+27,1+6+27)--(2+6+27,1+6+27)--(2+6+27,2+6+27)--(1+6+27,2+6+27)--cycle;
\draw[fill=black] (0+6+27,2+6+27)--(1+6+27,2+6+27)--(1+6+27,3+6+27)--(0+6+27,3+6+27)--cycle;
\draw[fill=black] (2+6+27,2+6+27)--(3+6+27,2+6+27)--(3+6+27,3+6+27)--(2+6+27,3+6+27)--cycle;

\draw[fill=black] (0+3+27,0+3+27)--(1+3+27,0+3+27)--(1+3+27,1+3+27)--(0+3+27,1+3+27)--cycle;
\draw[fill=black] (2+3+27,0+3+27)--(3+3+27,0+3+27)--(3+3+27,1+3+27)--(2+3+27,1+3+27)--cycle;
\draw[fill=black] (1+3+27,1+3+27)--(2+3+27,1+3+27)--(2+3+27,2+3+27)--(1+3+27,2+3+27)--cycle;
\draw[fill=black] (0+3+27,2+3+27)--(1+3+27,2+3+27)--(1+3+27,3+3+27)--(0+3+27,3+3+27)--cycle;
\draw[fill=black] (2+3+27,2+3+27)--(3+3+27,2+3+27)--(3+3+27,3+3+27)--(2+3+27,3+3+27)--cycle;

\draw[fill=black] (6+27,0+27)--(7+27,0+27)--(7+27,1+27)--(6+27,1+27)--cycle;
\draw[fill=black] (8+27,0+27)--(9+27,0+27)--(9+27,1+27)--(8+27,1+27)--cycle;
\draw[fill=black] (7+27,1+27)--(8+27,1+27)--(8+27,2+27)--(7+27,2+27)--cycle;
\draw[fill=black] (6+27,2+27)--(7+27,2+27)--(7+27,3+27)--(6+27,3+27)--cycle;
\draw[fill=black] (8+27,2+27)--(9+27,2+27)--(9+27,3+27)--(8+27,3+27)--cycle;
\end{tikzpicture}	
}\newline
\caption{Two scale-irregular Vicsek sets $K^{\bm{l}}$.}
\label{fig1}
\end{figure}
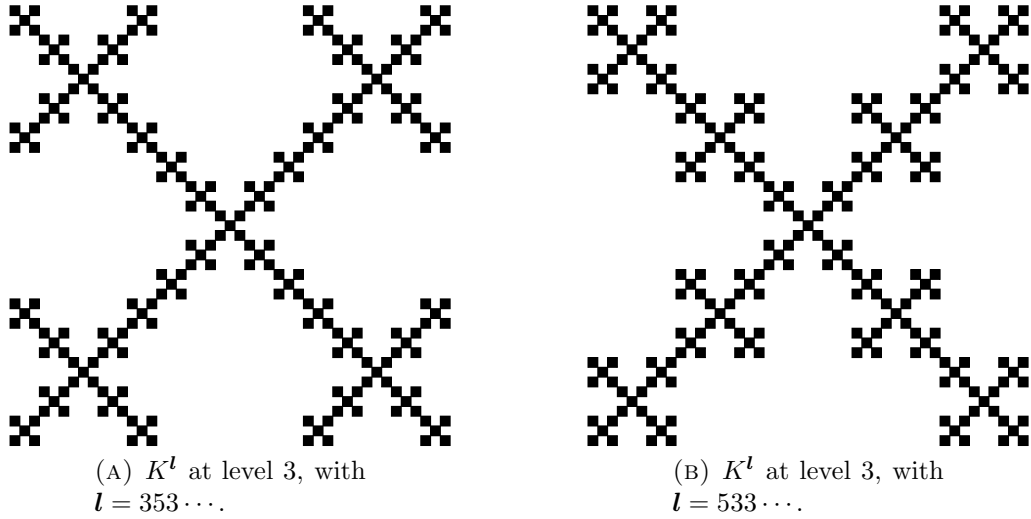

For $w\in W_{*}^{\bm{l}}$, we write $K_{w}^{\bm{l}}:=F_{w}^{\bm{l}}(K_{0})\cap K^{\bm{l}}$. We call $K_{w}^{\bm{l}}$ a {\em level-$n$ cell} if $w\in W_{n}^{\bm{l}}$.

\begin{proposition}
\label{l.inters} For $w\in W_{*}^{\bm{l}}$, we have $K_{w}^{\bm{l}%
}=\bigcup_{v\in S(w)} K_{v}^{\bm{l}}$, namely,
\begin{equation*}
F_{w}^{\bm{l}}(K_{0})\cap K^{\bm{l}}=\bigcup_{v\in S(w)}F_{v}^{\bm{l}%
}(K_{0})\cap K^{\bm{l}}.
\end{equation*}
\end{proposition}

\begin{proof}
Since%
\begin{equation*}
F_{w}^{\bm{l}}(K_{0})\cap\left(\bigcup_{v\in W_{|w|+1}^{\bm{l}}}F_{v}^{\bm{l}%
}(K_{0})\right)=\bigcup_{v\in S(w)}F_{v}^{\bm{l}}(K_{0})\text{ for any $w\in
W_{*}^{\bm{l}}$},
\end{equation*}
the conclusion follows by taking the intersection with $K^{\bm{l}}$ on both sides.
\end{proof}

\begin{proposition}
\label{p.coding} The map $\chi:W_{\infty}^{\bm{l}}\rightarrow K^{\bm{l}}$
defined by
\begin{equation*}
\{\chi(w)\}=\bigcap_{n\geq1}F_{[w]_{n}}^{\bm{l}}(K_{0})\text{ for all }w\in
W_{\infty}^{\bm{l}}
\end{equation*}%
is a continuous surjective map from $\left(W_{\infty}^{\bm{l}},\delta\right)$ to $\left(K^{\bm{l}},d\right)$. {%
Moreover, $\#\chi^{-1}(x)\leq 2$ for all $x\in K^{\bm{l}}$ and $%
\#\chi^{-1}(x)=2$ only if $x\in \bigcup_{n=1}^{\infty}\bigcup_{w\in W_{n}^{%
\bm{l}}}F_{w}^{\bm{l}}(\{q_{j}\}_{j=1}^{4})$.}
\end{proposition}

We will call $\chi$ the \textit{coding map}.

\begin{proof}
We first show that $\bigcap_{n\geq1}F_{[w]_{n}}^{\bm{l}}(K_{0})$ is a singleton
in $K^{\bm{l}}$ so that $\chi$ is well-defined. Indeed, since $%
F_{[w]_{n}}^{\bm{l}}(K_{0})$ are compact and form a decreasing sequence (with respect to inclusion) as $n$ increases, we
conclude by Cantor's Intersection Theorem that $\bigcap_{n\geq1}F_{[w]_{n}}^{%
\bm{l}}(K_{0})$ is non-empty. Since $\diam(F_{[w]_{n}}^{\bm{l}%
}(K_{0}))\rightarrow0$ as $n\rightarrow \infty$, the set $%
\bigcap_{n\geq1}F_{[w]_{n}}^{\bm{l}}(K_{0})$ cannot contain more than two
points, showing $\#\bigcap_{n\geq1}F_{[w]_{n}}^{\bm{l}}(K_{0})=1$.

To see that $\chi$ is surjective, note that for any $x\in K^{\bm{l}}=\bigcap_{n=1}^{\infty }\bigcup_{w\in W_{n}^{\bm{l}%
}}F_{w}^{\bm{l}}(K_{0})$, there exists $w_{1}\in W_{1}^{\bm{l}}$ for each
such that $x\in F_{w_{1}}^{\bm{l}}(K_{0})$. Whenever we find $w_{n}$, we can
apply proposition \ref{l.inters} and find $w_{n+1}\in S_{l_{n+1}}$ such that $x\in
F_{w_{1}w_{2}\cdots w_{n}w_{n+1}}^{\bm{l}}(K_{0})$. Let $w=w_{1}w_{2}\cdots$%
, then $x\in \bigcap_{n\geq1}F_{[w]_{n}}^{\bm{l}}(K_{0})=\{\chi(w)\}$ by
definition, showing the desired.

To see that $\chi$ is continuous, note that when $\min\{n:[w]_{n}\neq[\tau]%
_{n}\}-1=k$, i.e. $\delta(w,\tau)=\alpha^{k}$, both $\chi(w)$ and $%
\chi(\tau) $ belong to $F_{[w]_{k}}^{\bm{l}}(K_{0})$. Thus $d(\chi(w),\chi(%
\tau))\leq\diam F_{[w]_{k}}^{\bm{l}}(K_{0})$. As $\delta(w,\tau)\rightarrow0$%
, we have $k\rightarrow \infty$ and thus $d(\chi(w),\chi(\tau))\leq\diam %
F_{[w]_{k}}^{\bm{l}}(K_{0})\rightarrow0$, showing the desired.

To show the final assertion, suppose that $x=\chi(w)=\chi(w^{\prime })$ for
two distinct infinite words $w,w^{\prime }\in W_{\infty}^{\bm{l}}$. We can
write $w=[w]_{k}\tau,w^{\prime }=[w]_{k}\tau^{\prime}$ where $%
[w]_{k}=w_{1}w_{2}\cdots w_{k}$ denotes the longest common initial word of $%
w,w^{\prime}$ with length $k$ (when $k=0$, $[w]_{k}$ is the empty word), so
that $[\tau]_1\neq[\tau^{\prime}]_1$. Clearly, we have $x\in
F_{[w]_{k}[\tau]_1}^{\bm{l}}(K_{0}) \cap F_{[w]_{k}[\tau^{\prime}]_{1}}^{%
\bm{l}}(K_{0})$, which means that there exist two points $z_{1},z_{2}\in
K_{0}$ such that
\begin{equation}  \label{e.2.41}
x=F_{w_{1}}^{l_{1}}\circ \cdots \circ F_{w_{k}}^{l_{k}}\circ
F_{[\tau]_{1}}^{l_{k+1}}(z_{1})=F_{w_{1}}^{l_{1}}\circ \cdots \circ
F_{w_{k}}^{l_{k}}\circ F_{[\tau^{\prime}]_{1}}^{l_{k+1}}(z_{2}).
\end{equation}
Since each $F_{w_{j}}^{l_{j}}$, $1\leq j\leq k$ is invertible, we can apply
$\left(F_{w_{j}}^{l_{j}}\right)^{-1}\ (1\leq j\leq k)$ successively on both
sides of \eqref{e.2.41} and obtain
\begin{equation}  \label{e.2.42}
F_{[\tau]_{1}}^{l_{k+1}}(z_{1})=F_{[\tau^{\prime}]_{1}}^{l_{k+1}}(z_{2}).
\end{equation}
Let $[\tau]_{1}=2nl_{k+1}^{-1}q$ and $[\tau^{\prime}]_{1}=2n^{\prime}
l_{k+1}^{-1}q^{\prime}$ where $0\leq n\neq n^{\prime}\leq(l_{k+1}-1)/2 $ and
$q,q^{\prime}\in\{q_{j}\}_{j=1}^{4}$. Then \eqref{e.2.42} can be interpreted
as
\begin{equation*}
2nl_{k+1}^{-1}q+l_{k+1}^{-1}z_{1}=2n^{\prime}l_{k+1}^{-1}q^{%
\prime}+l_{k+1}^{-1}z_{2},
\end{equation*}
i.e., $z_{1}-z_{2}=2(n^{\prime}q^{\prime}-nq)$ and thus $d(z_{1},z_{2})=2d(n^{\prime}q^{\prime},nq)$. We show $d(z_{1},z_{2}) =2$.
\begin{enumerate}[label=\textup{(\arabic*)}]
\item If $q\neq q^{\prime}$, then
\begin{align}
4\geq d(z_{1},z_{2})^{2}&=4d(n^{\prime}q^{\prime},nq)^{2}=4\left({n^{2}+{n^{\prime}}^{2}-2n^{\prime}n\func{Re}(q\cdot \overline{%
q^{\prime}})} \right)\geq 4\left({n^{2}+{n^{\prime}}^{2}}\right)\geq 4
\end{align}
since $\func{Re}(q\cdot
\overline{q^{\prime}})\in\{0,-1\}$ and $n\neq n^{\prime}$, thus $d(z_{1},z_{2})=2$.

\item If $q=q^{\prime }$, then $2\geq \d(z_{1},z_{2}) =2\lvert
n^{\prime }-n\rvert \geq 2$, so $d(z_{1},z_{2}) =2$.
\end{enumerate}

Therefore, $z_{1}$ and $z_{2}$ must be the endpoints of the diagonal of $K_{0}$,
i.e., either $\{z_{1},z_{2}\}=\{q_{1},q_{3}\}$ or $\{z_{1},z_{2}\}=%
\{q_{2},q_{4}\}$ and
\begin{align}
\text{either }&x=F_{[w]_{k}[\tau]_1}^{\bm{l}}(q_{1}) \cap
F_{[w]_{k}[\tau^{\prime}]_{1}}^{\bm{l}}(q_{3}), \\
&x=F_{[w]_{k}[\tau]_1}^{\bm{l}}(q_{3}) \cap F_{[w]_{k}[\tau^{\prime}]_{1}}^{%
\bm{l}}(q_{1}), \\
&x=F_{[w]_{k}[\tau]_1}^{\bm{l}}(q_{2}) \cap F_{[w]_{k}[\tau^{\prime}]_{1}}^{%
\bm{l}}(q_{4}), \\
\text{or }&x=F_{[w]_{k}[\tau]_1}^{\bm{l}}(q_{4}) \cap
F_{[w]_{k}[\tau^{\prime}]_{1}}^{\bm{l}}(q_{2}).
\end{align}
The proof is complete.
\end{proof}

We state some terminologies in graph theory for scale-irregular Vicsek sets collecting from \cite{BC23, Mur19, DRY23}.

\begin{definition}[Graph and Cable system]
\label{d.graph} Fix $\bm{l}=(l_{k})_{k=1}^{\infty}$.

\begin{enumerate}[label=\textup{(\arabic*)}]
\item For each $n\geq0$, define $V_{0}:=\{q_{j}\}_{j=0}^{4}$ for $n=0$, and $%
V_{n}:=\bigcup_{w\in W_{n}^{\bm{l}}}F_{w}^{\bm{l}}(V_{0})$ for $n\geq1$.
Define
\begin{equation*}
E_{n}:=\left\{(x,y)\in V_{n}\times V_{n}: d(x,y)=l_{0}^{-1}l_{1}^{-1}l_{2}^{-1}\cdots
l_{n}^{-1}\right\} \text{ for }n\geq 0,
\end{equation*}
so that $(V_{n},E_{n})$ is a finite connected planer graph. We write $x\sim y $ whenever $(x,y)\in E_{n}$ and say that $x$ and $y$ are \textit{adjacent}.

\item For each $n\geq0$, by replacing each edge in $E_{n}$ by
an isometric copy of the line segment $[0,l_{0}^{-1}l_{1}^{-1}l_{2}^{-1}\cdots l_{n}^{-1}]$ and
gluing them in an obvious way at the vertices, we obtain a set $\overline{V}_{n}
$, called the corresponding \textit{cable system of $(V_{n},E_{n})$}. With an
abuse of notation, we regard $\overline{V}_{n}$ as a subset of $K^{\bm{l}}$
(see Fig. \ref{fig2}).

\item Define the \textit{skeleton} $\mathcal{S}:=\bigcup_{n=0}^{\infty}%
\overline{V}_{n}$. Let $\nu$ be the Lebesgue measure on $\mathcal{S}$, that is, $%
\nu$ assign $l_{0}^{-1}l_{1}^{-1}l_{2}^{-1}\cdots l_{n}^{-1}$ for each isometric copy
of $[0,l_{0}^{-1}l_{1}^{-1}l_{2}^{-1}\cdots l_{n}^{-1}]$. We extend $\nu$ to $K^{\bm{l}%
}$ by letting $\nu(K^{\bm{l}}\setminus \mathcal{S})=0$.

\item For every $n \geq 0$ and every adjacent $x, y \in V_{n}$, write $%
e(x,y)$ the geodesic in $\mathcal{S}$ connecting $x$ and $y$, namely, the
linear map from $[0,1]$ to the isometric copy of $[0,l_{0}^{-1}l_{1}^{-1}l_{2}^{-1}%
\cdots l_{n}^{-1}]$ connecting $u$ and $v$ such that $e(x,y)(0)=x$ and $%
e(x,y)(1)=y$. Then $\bigcup_{\substack{ x,y\in V_{n}  \\ x\sim y}}%
e(x,y)([0,1])=\overline{V}_{n}$. With an abuse of notation, we
sometimes regard $e(x,y)$ as a subset of $\mathcal{S}$.

\item A subset $A \subset K^{\bm{l}}$ is called \textit{convex} if for
any two points $x,y \in A\cap\mathcal{S}$, the geodesic path connecting $x$
to $y$ is included in $A\cap\mathcal{S}$.

\item For two adjacent $x$ and $y$ in $V_{n}$, we say that $x\prec y$ when
the geodesic distance from $0$ to $x$ in $\overline{V}_{n}$ is less than the
geodesic distance from $0$ to $y$.
\end{enumerate}
\end{definition}

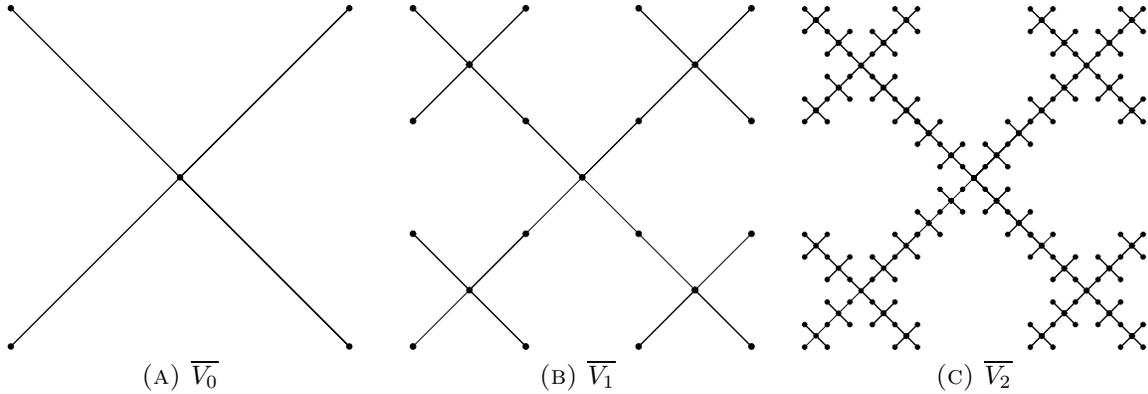
\begin{figure}[htbp]
\centering
\subfloat[][$\overline{V_0}$]{
	\begin{tikzpicture}[line width=0.5pt,scale=2.25]
	\draw (0,0)--(1,1)--(2,2)--(1,1)--(2,0)--(1,1)--(0,2);
	\filldraw (0,0) circle (.015)
	(1,1)circle (.015)
	(2,0)circle (.015)
	(0,2)circle (.015)
	(2,2)circle (.015);
	\end{tikzpicture}
   } \hspace{10pt}
\subfloat[$\overline{V_1}$]{
	\begin{tikzpicture}[line width=0.3pt,scale=0.375]
	\draw (0,0)--(1,1)--(2,2)--(3,3)--(4,4)--(2,2)--(3,1)--(4,0)--(2,2)--(1,3)--(0,4)--(2,2)--(4,4)--
	(5,5)--(6,6)--(7,7)--(8,8)--(9,9)--(10,10)--(11,11)--(12,12)--(10,10)--(9,11)--(8,12)--(10,10)--(11,9)--(12,8)--(10,10)
	--(6,6)--(5,7)--(4,8)--(3,9)--(2,10)--(1,11)--(0,12)--(2,10)--(1,9)--(0,8)--(2,10)--(3,11)--(4,12)--(2,10)--
	(6,6)--(7,5)--(8,4)--(9,3)--(10,2)--(11,1)--(12,0)--(10,2)--(9,1)--(8,0)--(10,2)--(11,3)--(12,4);
	\filldraw
	(0,0) circle (.1)
	(2,2)circle (.1)
	(0,4)circle (.1)
	(4,0)circle (.1)
	(4,4)circle (.1)
	(6,6)circle (.1)
	(8,8)circle (.1)
	(10,10)circle (.1)
	(12,12)circle (.1)
	(8,12)circle(.1)
	(12,8)circle (.1)
	(4,8)circle(.1)
	(2,10)circle(.1)
	(0,12)circle(.1)
	(2,10)circle(.1)
	(0,8)circle (.1)
	(4,12)circle(.1)
	(8,4)circle(.1)
	(10,2)circle(.1)
	(12,0)circle(.1)
	(10,2)circle(.1)
	(8,0)circle(.1)
	(12,4)circle(.1);
	\end{tikzpicture}} \hspace{10pt}
\subfloat[][$\overline{V_2}$]{
	\begin{tikzpicture}[line width=0.3pt,scale=0.15]
	\draw (0,0)--(1,1)--(2,0)--(1,1)--(0,2)--(1,1)--(2,2)--(3,3)--(4,2)--(3,3)--(2,4)--(3,3)--(4,4)--(5,5)--(6,6)--(7,7)--(8,6)--(7,7)--(6,8)--(7,7)--
	(8,8)--(9,9)--(10,10)--(9,9)--(10,8)--(9,9)--(8,10)--(9,9)--(5,5)--(6,4)--(7,3)--(6,2)--(7,3)--(8,4)--(7,3)--(8,2)--(9,1)--(10,0)--(9,1)--(8,0)--
	(9,1)--(10,2)--(9,1)--(4,6)--(3,7)--(2,6)--(3,7)--(4,8)--(3,7)--(2,8)--(1,9)--(0,10)--(1,9)--(0,8)--(1,9)--(2,10)--(1,9)--(5,5)
	(10,10)--(11,11)--(12,10)--(11,11)--(10,12)--(11,11)--(12,12)--(13,13)--(14,12)--(13,13)--(12,14)--(13,13)--(14,14)--(15,15)--(16,16)
	--(17,17)--(18,16)--(17,17)--(16,18)--(17,17)--
	(18,18)--(19,19)--(20,20)--(19,19)--(20,18)--(19,19)--(18,20)--(19,19)--(15,15)--(16,14)--(17,13)--(16,12)--(17,13)--(18,14)--(17,13)--(18,12)--(19,11)--(20,10)
	--(19,11)--(18,10)--
	(19,11)--(20,12)--(19,11)--(14,16)--(13,17)--(12,16)--(13,17)--(14,18)--(13,17)--(12,18)--(11,19)--(10,20)--(11,19)--(10,18)--(11,19)--(12,20)
	--(11,19)--(10,20)--(5,25)--
	(0,20)--(1,21)--(2,20)--(1,21)--(0,22)--(1,21)--(2,22)--(3,23)--(4,22)--(3,23)--(2,24)--(3,23)--(4,24)--(5,25)--(6,26)--(7,27)--(8,26)--(7,27)--(6,28)--(7,27)--
	(8,28)--(9,29)--(10,30)--(9,29)--(10,28)--(9,29)--(8,30)--(9,29)--(5,25)--(6,24)--(7,23)--(6,22)--(7,23)--(8,24)--(7,23)--(8,22)--(9,21)--(10,20)--(9,21)--(8,20)--
	(9,21)--(10,22)--(9,21)--(4,26)--(3,27)--(2,26)--(3,27)--(4,28)--(3,27)--(2,28)--(1,29)--(0,30)--(1,29)--(0,28)--(1,29)--(2,30)--(1,29)--(15,15)--
	(20,20)--(21,21)--(22,20)--(21,21)--(20,22)--(21,21)--(22,22)--(23,23)--(24,22)--(23,23)--(22,24)--(23,23)--(24,24)--(25,25)--(26,26)
	--(27,27)--(28,26)--(27,27)--(26,28)--(27,27)--
	(28,28)--(29,29)--(30,30)--(29,29)--(30,28)--(29,29)--(28,30)--(29,29)--(25,25)--(26,24)--(27,23)--(26,22)--(27,23)--(28,24)
	--(27,23)--(28,22)--(29,21)--(30,20)--(29,21)--(28,20)--
	(29,21)--(30,22)--(29,21)--(24,26)--(23,27)--(22,26)--(23,27)--(24,28)--(23,27)--(22,28)--(21,29)--(20,30)--(21,29)--(20,28)--(21,29)--(22,30)--
	(21,29)--(25,25)--(15,15)--(20,10)--(25,5)--
	(20,0)--(21,1)--(22,0)--(21,1)--(20,2)--(21,1)--(22,2)--(23,3)--(24,2)--(23,3)--(22,4)--(23,3)--(24,4)--(25,5)--(26,6)--(27,7)--(28,6)--(27,7)--(26,8)--(27,7)--
	(28,8)--(29,9)--(30,10)--(29,9)--(30,8)--(29,9)--(28,10)--(29,9)--(25,5)--(26,4)--(27,3)--(26,2)--(27,3)--(28,4)--(27,3)--(28,2)--(29,1)--(30,0)--(29,1)--(28,0)--
	(29,1)--(30,2)--(29,1)--(24,6)--(23,7)--(22,6)--(23,7)--(24,8)--(23,7)--(22,8)--(21,9)--(20,10)--(21,9)--(20,8)--(21,9)--(22,10);
	\filldraw (0,0) circle (.2)
	(1,1)circle (.2)
	(2,0)circle (.2)
	(0,2)circle (.2)
	(2,2)circle (.2)
	(3,3)circle (.2)(4,2)circle (.2)(2,4)circle (.2)(4,4)circle (.2)(5,5)circle (.2)(6,6)circle (.2)(7,7)circle (.2)(8,6)circle (.2)(6,8)circle (.2)
	(8,8)circle (.2)(9,9)circle (.2)(10,10)circle (.2)(10,8)circle (.2)(8,10)circle (.2)(6,4)circle (.2)(7,3)circle (.2)(6,2)circle (.2)(8,4)circle (.2)(7,3)circle (.2)(8,2)circle (.2)(9,1)circle (.2)(10,0)circle (.2)(9,1)circle (.2)(8,0)circle (.2)
	(9,1)circle (.2)(10,2)circle (.2)(9,1)circle (.2)(4,6)circle (.2)(3,7)circle (.2)(2,6)circle (.2)(4,8)circle (.2)(2,8)circle (.2)(1,9)circle (.2)(0,10)circle (.2)(0,8)circle (.2)(2,10)circle (.2)
	(10,10) circle (.2)(11,11) circle (.2)(12,10) circle (.2)(11,11) circle (.2)(10,12) circle (.2)(11,11) circle (.2)(12,12) circle (.2)(13,13) circle (.2)(14,12) circle (.2)(13,13) circle (.2)(12,14) circle (.2)(13,13) circle (.2)(14,14) circle (.2)(15,15) circle (.2)(16,16) circle (.2)
	(17,17) circle (.2)(18,16) circle (.2)(17,17) circle (.2)(16,18) circle (.2)(17,17) circle (.2)
	(18,18) circle (.2)(19,19) circle (.2)(20,20) circle (.2)(19,19) circle (.2)(20,18) circle (.2)(19,19) circle (.2)(18,20) circle (.2)(19,19) circle (.2)(15,15) circle (.2)(16,14) circle (.2)(17,13) circle (.2)(16,12) circle (.2)(17,13) circle (.2)(18,14) circle (.2)(17,13) circle (.2)(18,12) circle (.2)(19,11) circle (.2)(20,10)
	circle (.2)(19,11) circle (.2)(18,10) circle (.2)
	(19,11) circle (.2)(20,12) circle (.2)(19,11) circle (.2)(14,16) circle (.2)(13,17) circle (.2)(12,16) circle (.2)(13,17) circle (.2)(14,18) circle (.2)(13,17) circle (.2)(12,18) circle (.2)(11,19) circle (.2)(10,20) circle (.2)(11,19) circle (.2)(10,18) circle (.2)(11,19) circle (.2)(12,20) circle (.2)
	(11,19) circle (.2)(10,20) circle (.2)(5,25) circle (.2)
	(0,20) circle (.2)(1,21) circle (.2)(2,20) circle (.2)(1,21) circle (.2)(0,22) circle (.2)(1,21) circle (.2)(2,22) circle (.2)(3,23) circle (.2)(4,22) circle (.2)(3,23) circle (.2)(2,24)--(3,23)circle(.2)(4,24)circle(.2)(5,25)circle(.2)(6,26)circle(.2)(7,27)circle(.2)(8,26)circle(.2)(7,27)circle(.2)(6,28)circle(.2)(7,27)circle(.2)
	(8,28) circle (.2)(9,29)circle(.2)(10,30)circle(.2)(9,29)circle(.2)(10,28)circle(.2)(9,29)circle(.2)(8,30)circle(.2)(9,29)circle(.2)(5,25)circle(.2)(6,24)circle(.2)(7,23)circle(.2)(6,22)circle(.2)(7,23)circle(.2)(8,24)circle(.2)(7,23)circle(.2)(8,22)circle(.2)(9,21)circle(.2)(10,20)circle(.2)(9,21)circle(.2)(8,20)circle(.2)
	(9,21)circle(.2)(10,22)circle(.2)(9,21)circle(.2)(4,26)circle(.2)(3,27)circle(.2)(2,26)circle(.2)(3,27)circle(.2)(4,28)circle(.2)(3,27)circle(.2)(2,28)circle(.2)(1,29)circle(.2)(0,30)circle(.2)(1,29)circle(.2)(0,28)circle(.2)(1,29)circle(.2)(2,30)circle(.2)(1,29)circle(.2)(15,15)circle(.2)
	(20,20)circle(.2)(21,21)circle(.2)(22,20)circle(.2)(21,21)circle(.2)(20,22)circle(.2)(21,21)circle(.2)(22,22)circle(.2)(23,23)circle(.2)(24,22)circle(.2)(23,23)circle(.2)(22,24)circle(.2)(23,23)circle(.2)(24,24)circle(.2)(25,25)circle(.2)(26,26)
	circle(.2)(27,27)circle(.2)(28,26)circle(.2)(27,27)circle(.2)(26,28)circle(.2)(27,27)circle(.2)
	(28,28)circle(.2)(29,29)circle(.2)(30,30)circle(.2)(29,29)circle(.2)(30,28)circle(.2)(29,29)circle(.2)(28,30)circle(.2)(29,29)circle(.2)(25,25)circle(.2)(26,24)circle(.2)(27,23)circle(.2)(26,22)circle(.2)(27,23)circle(.2)(28,24)
	circle(.2)(27,23)circle(.2)(28,22)circle(.2)(29,21)circle(.2)(30,20)circle(.2)(29,21)circle(.2)(28,20)circle(.2)
	(29,21)circle(.2)(30,22)circle(.2)(29,21)circle(.2)(24,26)circle(.2)(23,27)circle(.2)(22,26)circle(.2)(23,27)circle(.2)(24,28)circle(.2)(23,27)circle(.2)(22,28)circle(.2)(21,29)circle(.2)(20,30)circle(.2)(21,29)circle(.2)(20,28)circle(.2)(21,29)circle(.2)(22,30)circle(.2)
	(21,29)circle(.2)(25,25)circle(.2)(15,15)circle(.2)(20,10)circle(.2)(25,5)circle(.2)
	(20,0)circle(.2)(21,1)circle(.2)(22,0)circle(.2)(21,1)circle(.2)(20,2)circle(.2)(21,1)circle(.2)(22,2)circle(.2)(23,3)circle(.2)(24,2)circle(.2)(23,3)circle(.2)(22,4)circle(.2)(23,3)circle(.2)(24,4)circle(.2)(25,5)circle(.2)(26,6)circle(.2)(27,7)circle(.2)(28,6)circle(.2)(27,7)circle(.2)(26,8)circle(.2)(27,7)circle(.2)
	(28,8)circle(.2)(29,9)circle(.2)(30,10)circle(.2)(29,9)circle(.2)(30,8)circle(.2)(29,9)circle(.2)(28,10)circle(.2)(29,9)circle(.2)(25,5)circle(.2)(26,4)circle(.2)(27,3)circle(.2)(26,2)circle(.2)(27,3)circle(.2)(28,4)circle(.2)(27,3)circle(.2)(28,2)circle(.2)(29,1)circle(.2)(30,0)circle(.2)(29,1)circle(.2)(28,0)circle(.2)
	(29,1)circle(.2)(30,2)circle(.2)(29,1)circle(.2)(24,6)circle(.2)(23,7)circle(.2)(22,6)circle(.2)(23,7)circle(.2)(24,8)circle(.2)(23,7)circle(.2)(22,8)circle(.2)(21,9)circle(.2)(20,10)circle(.2)(21,9)circle(.2)(20,8)circle(.2)(21,9)circle(.2)(22,10)circle(.2)(2,24)circle(.2);
	\end{tikzpicture}	
}
\caption{Cable systems $\overline{V_0}$, $\overline{V_1}$ and $\overline{V_2}$ for $K^{\bm{l}}$ with $\bm{l}=35\cdots$.}
\label{fig2}
\end{figure}

\begin{remark}
\begin{enumerate}[label=\textup{(\arabic*)}]
\item The measure $\nu$ is not a Radon measure, since the measure of
any ball with positive radius is infinite.

\item Although the image of $e(x,y)$ and $e(y,x)$ coincide for adjacent vertices $%
x,y$, we distinguish them when we take integration (along this
edge) by assigning $e(x,y)$ the positive orientation when $%
x\prec y$. So for an integrable non-negative function $g$, we have $0\leq
\int_{e(x,y)}g\dif \nu=-\int_{e(y,x)}g\dif \nu$ when $e(x,y)$ is positively
oriented.
\end{enumerate}
\end{remark}

\subsection{Measure on scale-irregular Vicsek set: doubling property}\label{s.meas}

Given an odd integer $l\geq3$, let $K^{l}$ be the self-similar Vicsek set with
common contraction ration $l^{-1}$. Then by solving Moran's equation \cite[%
Theorem 2.7]{Fal97} directly, we see that $K^{l}$ has Hausdorff
dimension
\begin{equation}  \label{e.df}
\alpha_{l}:=\mathrm{dim}_{H}(K^{l})=\frac{\log(2l-1)}{\log{l}}.
\end{equation}

Let
\begin{equation}
\psi(r):=\begin{dcases} \prod_{k=0}^{n}(2l_{k}-1)^{-1} &\text{ \ for \ }
\rho_{n+1}< r\leq \rho_n\ (n\geq0),\\ 1 &\text{ \ for \ } r \geq2.
\end{dcases}\label{ksai}
\end{equation}
For each word $w\in W_{m}^{\bm{l}}$, we define
\begin{align}\label{e.defmu}
\mu(K_{w}^{\bm{l}}):=(\#W_{m}^{\bm{l}})^{-1}=\psi (\rho _{m}).
\end{align}
We extend $\mu$ to be a Borel measure on $K^{\bm{l}}$ by
Kolmogorov's extension theorem.

{If the sequence $\bm{l}$ consists of finitely many contraction ratios, that is, $\sup\limits_{n\geq 1}l_n<\infty$, the \emph{volume doubling property} and \emph{reverse volume doubling property} hold, as the following states.}
\begin{proposition}
\label{l.meas} Let $K^{\bm{l}}$ be a scale-irregular Vicsek set with $\bm{l}$
satisfying {$\sup\limits_{n\geq 1}l_n<\infty$}.

\begin{enumerate}[label=\textup{(\arabic*)}]
\item There exist constants $c_{1},c_{2}>0$ such that for all $0<r<R\leq 2$,
\begin{equation}
c_1\left(\frac{R}{r}\right)^{{\inf\limits_{n\geq 1}\alpha_{l_n}}}\leq\frac{\psi(R)}{%
\psi(r)}\leq c_2\left(\frac{R}{r}\right)^{{\sup\limits_{n\geq 1}\alpha_{l_n}}}.
\label{scaling1}
\end{equation}

\item There exist constants $c_{3},c_{4}>0$ such that for all $x\in K^{\bm{l}%
}$ and $0<r\leq2$,
\begin{equation}
c_{3}\psi (r)\leq \mu(B(x,r)) \leq c_{4}\psi (r).
\end{equation}
\end{enumerate}
\end{proposition}

\begin{proof}
{(1) Note that $1<\inf\limits_{n\geq 1}\alpha_{l_n} \leq \sup\limits_{n\geq 1}\alpha_{l_n}\leq \frac{\log 5}{\log 3}$.
For $0<r<R\leq 2$, assume that $R\in(\rho_{m+1},\rho_{m}]$ and $r\in(\rho_{n+1},\rho_n]$ for some $0\leq m\leq n$, then
\begin{align*}
\frac{\psi(R)}{\psi(r)}&=\frac{\prod_{k=0}^{m}(2l_{k}-1)^{-1} }{\prod_{k=0}^{n}(2l_{k}-1)^{-1}}=\prod_{k=m+1}^{n}(2l_{k}-1)= \prod_{k=m+1}^{n} l_k^{\alpha_{l_k}}\leq \left(\prod_{k=m+1}^{n} l_k\right)^{\sup\limits_{k\geq 1}\alpha_{l_k}} \\
& =\left(\frac{\rho_m}{\rho_{n}}
\right)^{\sup\limits_{k\geq 1}\alpha_{l_k}}\leq (\sup_{k\geq 1}l_k)^{\sup_{k\geq 1}\alpha_{l_k}}\left(\frac{R}{r}\right)^{\sup\limits_{k\geq 1}\alpha_{l_k}}.
\end{align*}
Similarly, we also have
\begin{align*}
\frac{\psi(R)}{\psi(r)}\geq \left(\frac{\rho_m}{\rho_{n}}\right)^{\inf\limits_{k\geq 1}\alpha_{l_k}}\geq C (\sup\limits_{k\geq 1}l_k)^{-\inf\limits_{k\geq 1}\alpha_{l_k}}\left(\frac{R}{r}\right)^{\inf\limits_{k\geq 1}\alpha_{l_k}}.
\end{align*}
Since $\sup\limits_{n\geq 1}l_n<\infty$, we see that \eqref{scaling1} holds.}

{(2) For this, we only need
to observe that when $r\in(\rho_{m+1},\rho_{m}]$ for some $m$, any metric
ball $B(x,r)$ in $K^{\bm{l}}$ (where $x\in K^{\bm{l}}$) contains a level-$%
(m+2)$ cell and can be covered by 5 level-$m$ cells. Indeed, a level-$%
(m+2)$ cell has diameter smaller than $\rho_{m+1}$, and each level-$%
n$ cell has at most 4 neighboring level-$%
n$ cells, while level-$%
n$ cells are separated by distance $\rho_n$ when they are not neighbors. Thus $$\psi (\rho_{m+2})\leq \mu(B(x,r))\leq 5\psi (\rho_{m}),$$ our assertion then follows by using (1) and the fact that $\frac{\rho_{m+2}}{\rho_{m}}$ is bounded (as $\sup\limits_{n\geq 1}l_n<\infty$).}
\end{proof}

\subsection{Lack of Ahlfors-regularity and self-similarity}\label{s.nss}
{In this subsection, we show that scale-irregular Vicsek sets may lack Ahlfors-regularity and self-similarity. For simplicity, we make the following assumption in this subsection:\footnote{We adopt the convention $1/0 = \infty$ and the limit is considered in the extended real line $\overline{\mathbb{R}}=[-\infty,\infty]$.}
\begin{equation}
\text{$\bm{l}$ consists of two distinct odd numbers $a, b\geq3$, the limit } \theta:=\lim _{n \rightarrow \infty} \theta_n:=\lim_{n \rightarrow \infty}\frac{[n]_a}{[n]_b} \in[0,\infty) \text{ exists}, \label{e.con.a}
\end{equation}
where $[n]_{a}:=\sum_{j=1}^{n}\mathds{1}_{\{l_{j}=a\}}$ and $[n]_{b}:=\sum_{j=1}^{n}\mathds{1}_{\{l_{j}=b\}}$ are the numbers of $a$ and $b$ appeared in the first $n$-digits of $\bm{l}$}. We first analyse the behaviour of the
Hausdorff measure by using \cite[Theorem 3.1]{HRWW00}, since
scale-irregular Vicsek sets are \textit{Moran sets} (see \cite[Section 1.2]{HRWW00} for definitions).

\begin{proposition}
\label{HHH} {Assume \eqref{e.con.a} and define $\eta_n:=n(\theta_n-\theta)$. Then the following holds.}
\begin{enumerate}[label=\textup{(\arabic*)}]
\item The Hausdorff dimension of $K^{\bm{l}}$, denoted by $\alpha$, is given
by
\begin{equation}  \label{e.hdim}
\alpha=\frac{\theta \log (2a-1)+\log (2b-1)}{\theta \log a+\log b}.
\end{equation}

\item When $3\leq a< b$, we have the following equivalences:
\begin{align*}
0<\mathcal{H}^{\alpha}(K^{\bm{l}})<\infty\ & \text{if and only if}\
\liminf_{k\rightarrow\infty} \eta_k \in \mathbb{R}, \\
\mathcal{H}^{\alpha}(K^{\bm{l}})=0 \ & \text{if and only if}\
\liminf_{k\rightarrow\infty} \eta_k=-\infty, \\
\mathcal{H}^{\alpha}(K^{\bm{l}})=\infty \ & \text{if and only if}\
\liminf_{k\rightarrow\infty} \eta_k=+\infty.
\end{align*}

\item When $3\leq b<a$, we have the following equivalences:
\begin{align*}
0<\mathcal{H}^{\alpha }(K^{\bm{l}})<\infty \ & \text{if and only if}\
\limsup_{k\rightarrow \infty }\eta _{k}\in \mathbb{R}, \\
\mathcal{H}^{\alpha }(K^{\bm{l}})=0\ & \text{if and only if}\
\limsup_{k\rightarrow \infty }\eta _{k}=+\infty , \\
\mathcal{H}^{\alpha }(K^{\bm{l}})=\infty \ & \text{if and only if}\
\limsup_{k\rightarrow \infty }\eta _{k}=-\infty .
\end{align*}
\end{enumerate}
\end{proposition}

\begin{proof}
\begin{enumerate}[label=\textup{(\arabic*)}]
\item This is a direct application of \cite[Theorem 3.1]{HRWW00}.

\item Define
\begin{align}
\xi_n:=\sum_{w \in W^{\bm{l}}_n} (\mathrm{diam}(K_{w}^{\bm{l}}))^{\alpha}=\rho_n^\alpha(\psi(\rho_n))^{-1}=(a^{-[n]_a}b^{-[n]_b})^\alpha
(2a-1)^{[n]_a}(2b-1)^{[n]_b}.
\end{align}
By \cite[Theorem 3.1]{HRWW00}, $\mathcal{H}^{\alpha}(K^{\bm{l}})$
and $\liminf_{k\rightarrow\infty} \xi_k$ are simultaneously zero, finite and
positive, or infinite. So by taking the logarithm, we only need to show that $%
\liminf_{k\rightarrow\infty} \log\xi_k$ and $\liminf_{k\rightarrow\infty}
\eta_k$ are simultaneously $-\infty$, finite, or $\infty$, respectively.
Let
\begin{equation*}
f(x):=\frac{x \log (2a-1)+\log (2b-1)}{x \log a+\log b},\quad x\in \mathbb{R}
\end{equation*}
Then for sufficiently large $n$,
\begin{align}
\log \xi_n & =-\alpha\left([n]_a \log a+[n]_b \log b\right)+\left([n]_a \log
(2a-1)+[n]_b \log (2b-1)\right)  \notag \\
& =\left([n]_a \log a+[n]_b \log b\right)\left(\frac{[n]_a \log (2a-1)+[n]_b
\log (2b-1)}{[n]_a \log a+[n]_b \log b}-\alpha\right)  \notag \\
& \asymp n\left(f(\theta_n)-f(\theta)\right)  \notag \\
& \asymp f^{\prime }(\theta)\eta_n,  \label{alfr}
\end{align}
where we used the conclusion $\alpha=f(\theta)$ of (1) in the third line,
and in the fourth line we use $3\leq a< b$ to obtain
\begin{align*}
f^{\prime}(x)&=\frac{\log (2a-1) \log b-\log a \log (2b-1)}{(x \log a+\log
b)^2} \\
&=\frac{\log a\log b}{(x \log a+\log
b)^2}\left(\frac{\log (2a-1)}{\log a}-\frac{\log (2b-1)}{\log b}%
\right)> 0
\end{align*}
by noting that $\log(2l-1)/\log l$ strictly decreases in $l$. Therefore $%
\liminf_{k\rightarrow\infty} \log\xi_k$ and $\liminf_{k\rightarrow\infty}
\eta_k$ are simultaneously $-\infty$, finite, or $+\infty$.

\item Note that in the case of $3\leq b<a$, $\liminf_{n\rightarrow \infty
}\log \xi _{n}\asymp f^{\prime }(\theta )\limsup_{n\rightarrow \infty }\eta
_{n}$ as $f^{\prime }(\theta )<0$. The rest of the proof is the same as (2).
\end{enumerate}
\end{proof}
{
\begin{remark}\label{eqv}
Assume \eqref{e.con.a}. If the Hausdorff measure $\mathcal{H}^{\alpha}$ of $K^{\bm{l}}$ exists, that is, $%
0<\mathcal{H}^{\alpha}(K^{\bm{l}})<\infty$, then $\mathcal{H}^{\alpha}$ is equivalent
to $\mu$. The reason is that $\mathcal{H}^{\alpha}(K_{w}^{\bm{l}})=\mathcal{H}%
^{\alpha}(K^{\bm{l}})\mu (K_{w}^{\bm{l}})$ as level-$k$ cells are only translations of each
other for all $w\in W_{k}^{\bm{l}}\ (k\geq 1)$, and any measurable set of $K^{\bm{l}}$ can be
approximated by some unions of cells of level $k$ (as $k\rightarrow \infty $),
which means that $\mathcal{H}^{\alpha}=\mathcal{H}^{\alpha}(K^{\bm{l}})\mu$.
\end{remark}
}

In our setting $\mu$ is not necessarily Ahlfors regular, as we will see in
the following. {A compact set \( F \subseteq \mathbb{R}^d \) is called \textit{Ahlfors regular} if there exists a constant \( c > 0 \) such that
\[\frac{1}{c} r^{\dim_H F} \leq \mathcal{H}^{\dim_H F} \left(B(x, r) \cap F\right) \leq c \, r^{\dim_H F}\]
for all \( x \in F \) and all \( 0 < r < \diam(F) \) (see for example \cite[Section 6.4]{Fra21}).}
{
\begin{proposition}
\label{af} Assume \eqref{e.con.a}. Let $\alpha $ be given by \eqref{e.hdim}%
. Then the followings are equivalent.
\begin{enumerate}[label=\textup{({\arabic*})},align=left,leftmargin=*,topsep=5pt,parsep=0pt,itemsep=2pt]
\item $\{\eta _{n}\}_{n}$ is bounded.
\item $K^{\bm{l}}$ is Ahlfors regular.
\item The measure $\mu $ on $K^{\bm{l}}$ is $\alpha $-Ahlfors regular.
\end{enumerate}
\end{proposition}
\begin{proof}
\begin{enumerate}[align=left,leftmargin=*,topsep=5pt,parsep=0pt,itemsep=2pt]
\item[$(1)\Leftrightarrow(3).$]  By
Proposition \ref{l.meas}, $\mu $ is $\alpha $-Ahlfors regular if and only if $\psi (\rho _{n})\asymp \rho
_{n}^{\alpha }$, i.e., the sequence $\{(\psi (\rho _{n}))^{-1}\rho
_{n}^{\alpha }\}_{n}$ is bounded below from $0$ and away from $\infty $,
which, by taking the logarithm and using the definition of $\eta _{n}$ and $\theta
_{n}$, is equivalent to that the sequence
\begin{equation*}
\frac{\eta _{n}}{1+\theta _{n}}\log \left( {\frac{2a-1}{a^{\alpha }}}\right)
+\frac{n}{1+\theta _{n}}\left( {\theta \log \left( {\frac{2a-1}{a^{\alpha }}}%
\right) +\log \left( {\frac{2b-1}{b^{\alpha }}}\right) }\right)
\end{equation*}%
is bounded. By \eqref{e.hdim}, we have
\begin{equation*}
{\theta \log \left( {\frac{2a-1}{a^{\alpha }}}\right) +\log \left( {\frac{%
2b-1}{b^{\alpha }}}\right) }=0
\end{equation*}%
and $\log \left( {\frac{2a-1}{a^{\alpha }}}\right) \neq 0$ by $\theta\in [0,\infty)$. The conclusion immediately follows by \eqref{e.con.a}.
\item[$(1)\Rightarrow(2).$] The existence of the Hausdorff measure $\mathcal{H}^{\alpha}$ immediately follows from Proposition \ref{HHH}. Then by Remark \ref{eqv}, we just need to check the $\alpha $-Ahlfors regularity of $\mu $, which is guaranteed by $(1)\Leftrightarrow(3)$.
\item[$(2)\Rightarrow(3).$] This immediately follows from Remark \ref{eqv}.
\end{enumerate}
 \end{proof}
}
Now we can present our sufficient conditions for $K^{%
\bm{l}}$ to be non-self-similar. {For Moran sets, we can use a different and simpler approach from that in \cite{FHZ24} for graph-directed attractors.}
\begin{theorem}
\label{t.non-sml} {Assume \eqref{e.con.a} with $a<b$.}
If $\{\eta_n\}_n$ satisfies one of the following two conditions:
\begin{enumerate}[label=\textup{(\arabic*)}]
\item $\liminf_{n\rightarrow\infty}\eta_n=\infty$,
\item $\liminf_{n\rightarrow\infty} \eta_n \in \mathbb{R}$ but $%
\limsup_{n\rightarrow\infty} \eta_n=\infty$,
\end{enumerate}
then $K^{\bm{l}}$ is not self-similar, that is, it is not the attractor of any standard iterated function system.
\end{theorem}

\begin{proof}
\begin{enumerate}[label=\textup{(\arabic*)}]
\item By Proposition \ref{HHH}, $\liminf_{n\rightarrow\infty} \eta_n=\infty$
implies $\mathcal{H}^{\alpha}(K^{\bm{l}})=\infty $. But it is known that $\mathcal{H}^{\dim_H(K)}(K)<\infty$ for any self-similar set $K$
by \cite[Corollary 3.3]{Fal97}, thus $K^{\bm{l}}$ is not
self-similar.

\item By Proposition \ref{HHH}, $\liminf_{n\rightarrow\infty} \eta_n \in
\mathbb{R} $ implies $0<\mathcal{H}^{\alpha}(K^{\bm{l}})<\infty$. Then by  Remark \ref{eqv}, $\mathcal{H}^{\alpha}$ is equivalent to $\mu$ on $%
K^{\bm{l}}$. By Proposition \ref{af}, $\limsup_{n\rightarrow\infty}
\eta_n=\infty$ implies that $\mu$ is not Ahlfors regular, so that $\mathcal{H}
^{\alpha}$ is not Ahlfors regular. But \cite[Theorem 2.1]%
{AKT20} states that for any self-similar set $K$,
the Hausdorff measure $\mathcal{H}^{\dim_H(K)}$ is Ahlfors regular, should
it exist. This shows that $K^{\bm{l}}$ cannot be self-similar.
\end{enumerate}
\end{proof}

\begin{example}
We give some examples of scale-irregular Vicsek set with $a=3$ and $b=5$ satisfying the conditions in Theorem \ref{t.non-sml}. The sequence $\bm{l}=335333553333555333335555\cdots$ satisfies Theorem \ref{t.non-sml} (1).  That is, ``$3$'' appears consecutively $k+1$ times and then ``$5$'' appears consecutively $k$ times ($k\geq 1$), as the following shows.
	\[\bm{l}=\underbrace{\overbrace{33}^{2}\overbrace{5}^{1}\overbrace{333}^{3}\overbrace{55}^{2}\overbrace{3333}^{4}\overbrace{555}^{3}\cdots\overbrace{3\cdots3}^{k+1}\overbrace{5\cdots 5}^{k}}_{\substack{\text{``$3$'' appears $(k^{2}+3k)/2$ times}\\\text{``$5$'' appears $(k^{2}+k)/2$ times} }}\overbrace{33\cdots3}^{k+2}\overbrace{55\cdots 5}^{k+1}\cdots \]
We claim that $\{\theta_n\}_n$ has limit $\theta=1$ and $\liminf_{n\rightarrow\infty}\eta_n=\infty$. Indeed, when $n=(k^{2}+2k)+t $ for some $k\geq1$ and $1\leq t\leq k+2$,``$3$'' occurs $t+(k^{2}+3k)/2$ times and ``$5$'' occurs $(k^{2}+k)/2$ times in the first $n$-digits of $\bm{l}$, thus for $k\geq1$ and $1\leq t\leq k+2$,
\begin{align}
\theta_{k^2+2k+t}=\frac{k^{2}+3k+2t}{k^{2}+k} \text{ and } \eta_{k^2+2k+t}=\frac{k^2+2k+t}{k(k+1)}(2k+2t).
\end{align}
When $n=(k^{2}+3k+2)+t $ for some $k\geq1$ and $1\leq t\leq k+1$, ``$3$'' occurs $(k^{2}+5k+4)/2$ times and ``$5$'' occurs $t+(k^{2}+k)/2$ times {in the first $n$-digits of $\bm{l}$}, thus for $k\geq1$ and $1\leq t\leq k+1$,
\[ \theta_{k^2+3k+2+t}=\frac{k^{2}+5k+4}{k^{2}+k+2t} \text{ and }\eta_{k^2+3k+2+t}=\frac{k^2+3k+2+t}{k^2+k+2t}(4k+4-2t).\]	
For both cases, the claim follows by noting that
\[\theta_n=1+O(n^{-1/2}) \text{ and } \eta_{n}\geq 2\cdot 3^{-1} n^{1/2}.\]

In the same way, one can verify that the sequence $\bm{l}=35335533355533335555\cdots$ satisfies Theorem \ref{t.non-sml} (2). For both sequences, their corresponding Vicsek sets
cannot be self-similar.

\end{example}

\section{$p$-energy and its associated energy measure}

\label{s.energy} In this section, we prove Theorem \ref{t.energy} and Theorem \ref{t.emc}%
. Throughout this section, we fix $\bm{l}=(l_{k})_{k=1}^{\infty}$ where each $l_{k}\geq 3$ is an odd integer ($l_{0}:=1$). We omit
the superscript $\bm{l}$ and write $K^{\bm{l}}$ as $K$.

\subsection{Construction of $p$-energy norm}
We construct $p$-energy norms on scale-irregular Vicsek sets by extending the methods in \cite{BC23}.
\begin{definition}
\label{penergy_1} For each convex $A\subset K$, $1<p<\infty$, $n\geq0$ and $%
u\in C(K)$, define
\begin{equation*}
\mathcal{E}_{p,n;A}(u):=\frac{1}{2}\left(\prod_{j=0}^{n}
l_{j}^{p-1}\right)\sum_{\substack{ x,y\in A\cap V_{n}  \\ x\sim y}}%
\left|u(x)-u(y)\right|^{p}
\end{equation*}
and we denote $\mathcal{E}_{p,n;K}$ as $\mathcal{E}_{p,n}$. Define the
function space $\mathcal{F}_{p}$ by
\begin{equation}  \label{e.domain}
\mathcal{F}_{p}:=\left\{u\in C(K):\sup_{n\geq0}\mathcal{E}%
_{p,n}(u)<\infty\right\};
\end{equation}
and for each $u\in\mathcal{F}_{p}$, define
\begin{equation}  \label{e.energy}
\mathcal{E}_{p;A}(u):=\sup_{n\geq0}\mathcal{E}_{p,n;A}(u),\ \mathcal{E}%
_{p}(u):=\mathcal{E}_{p;K}(u)\ \text{and } \lVert u\rVert_{\mathcal{F}%
_{p}}:=\left(\lVert u\rVert_{L^{p}(K,\mu)}^{p}+\mathcal{E}_{p}(u)\right)^{1/p}.
\end{equation}
\end{definition}

It is then easy to verify that $\left(\mathcal{F}_{p}, \lVert \cdot\rVert_{%
\mathcal{F}_{p}}\right)$ is a normed real vector space. The following
proposition shows that these discrete energies $\{\mathcal{E}_{p,n;A}(u)\}$
increase in $n$.

\begin{proposition}
\label{VE-mono} For each $0\leq m\leq n$ and each $u\in C(K)$, we have
\begin{align}
\mathcal{E}_{p,m;A}(u)\leq \mathcal{E}_{p,n;A}(u) .  \label{penergy_dec}
\end{align}
In particular,
\begin{equation}
\mathcal{E}_{p}(u)=\sup_{n\geq0}\mathcal{E}_{p,n}(u)=\lim_{n\rightarrow
\infty}\mathcal{E}_{p,n}(u)\ \text{for all } u\in C(K).  \label{ve1}
\end{equation}
\end{proposition}

\begin{proof}
For two adjacent vertices $x$ and $y$ in $V_{n}$, by the tree structure of $%
K $, there are unique $(l_{n+1}+1)$ vertices $\{x_{j}\}_{j=0}^{l_{n+1}}\subset
V_{n+1}$ such that $x_{j}\sim x_{j+1}$, $1\leq j\leq l_{n+1}$ and $x_{0}=x$,
$x_{l_{n+1}}=y$. Thus
\begin{equation*}
\left|u(x_{0})-u(x_{l_{n+1}})
\right|^{p}=\left|\sum_{j=1}^{l_{n+1}}\left(u(x_{j-1})-u(x_{j})\right)%
\right|^{p}\leq
l_{n+1}^{p-1}\sum_{j=1}^{l_{n+1}}\left|u(x_{j-1})-u(x_{j})\right|^{p}.
\end{equation*}%
By adding all adjacent vertices in $V_{n}$ on both sides, we conclude that $%
\mathcal{E}_{p,n;A}(u)\leq \mathcal{E}_{p,n+1;A}(u)$ for each convex $%
A\subset K$, which completes the proof.
\end{proof}

Next, we analyse the ``core'' functions in the domain $\mathcal{F}_{p}$.
\begin{definition}[Piecewise affine functions]
A continuous function $\Psi: K \rightarrow \mathbb{R}$ is called \textit{$n$%
-piecewise affine}, if $\Psi$ is linear between the vertices of $\overline{V}%
_{n}$ and constant on any connected component of $\overline{V}_{m}\setminus
\overline{V}_{n}$ for every $m > n$. A continuous function $\Psi: K
\rightarrow \mathbb{R}$ is called \textit{piecewise affine}, if there exists
$n\geq0$ such that $\Psi$ is $n$-piecewise affine.
\end{definition}

\begin{proposition}
\label{l.affine} For each $u\in C(K)$ and $n\geq0$, define $H_{n} u$ to be
the unique $n$-piecewise affine function on $K$ that coincides with $u$ on $%
V_{n}$. Then $H_{n}u\rightarrow u$ uniformly on $K$ as $n\rightarrow\infty$. Moreover, $H_{n}u\in \mathcal{F}_{p}$ for each $u\in C(K)$ and $%
n\geq0$.
\end{proposition}

\begin{proof}
For each word $w\in W_{n}$, we have
\begin{equation*}
\min_{K_{w}}u\leq H_{n}u(x)\leq\max_{K_{w}}u,\quad\forall x\in K_{w}.
\end{equation*}
It follows that
\begin{equation*}
|H_{n}u(x)-u(x)|\leq \osc _{K_{w}}u, \quad\forall x\in K_{w},
\end{equation*}
where $\osc _{K_{w}}u:=\max_{x\in K_{w}}u(x)-\min_{x\in K_{w}}u(x)$. The first assertion then follows from
\begin{equation*}
\sup_{x\in K}|H_{n}u(x)-u(x)|\leq \sup_{w\in W_{n}}\osc _{K_{w}}u\rightarrow
0\quad(n\rightarrow\infty)
\end{equation*}
as $u$ is uniformly continuous. The second assertion is obvious.
\end{proof}

As in the work of Baudoin and Chen \cite[Theorem 3.1]{BC23}, a notable characteristic of Vicsek sets lies in their distinctive geometric structure, which permits the existence of gradients. The following proposition says that {\cite[Theorem 3.1]{BC23}} also holds on scale-irregular Vicsek sets.

\begin{proposition}[{\cite[Theorem 3.1]{BC23}}]
\label{p.grad} Let $1<p<\infty$ and $u \in C(K)$. The followings are
equivalent:

\begin{enumerate}[label=\textup{(\arabic*)}]
\item $u \in \mathcal{F}_{p}$;

\item There exists $g \in L^{p}(\mathcal{S}, \nu)$ such that, for every $n
\geq 0$ and every adjacent $x, y \in V_{n}$,%
\begin{equation}  \label{e.grad}
u(y)-u(x)=\int_{e(x, y)} g \dif \nu.
\end{equation}
\end{enumerate}

In this case, $g$ is unique in $L^{p}(\mathcal{S}, \nu)$. Moreover, for
every $u \in \mathcal{F}_{p}$ and every convex $A\subset K$,
\begin{equation}
 \mathcal{E}_{p;A}(u)=\int_{A\cap\mathcal{S}} |g|^{p}\dif\nu. \label{e.grad-1}
\end{equation}
\end{proposition}
We denote $g$ in this proposition by $\partial u$ and refer it as the gradient of $u$.
\begin{proof}
We first show that (1) implies (2). If $\Psi$ is a piecewise affine function, it
is obvious that there exists a function, denoted by $\partial \Psi$, such
that for every adjacent $x,y \in V_{n}$,%
\begin{equation*}
\Psi(y)-\Psi(x)=\int_{{e}(x,y)} \partial \Psi \dif\nu.
\end{equation*}
In fact, for each adjacent $x,y\in V_{n}$ with $x\prec y$ (so that $e(x,y)$
is positively oriented), we can choose $\Psi$ such that $\partial \Psi$ takes the value $%
(\Psi(y)-\Psi(x))\cdot d(x,y)^{-1}$ on $e(x,y)((0,1))$. Therefore,
\begin{equation*}
\int_{A\cap\mathcal{S}}\left|\partial\Psi\right|^{p} \dif\nu=\frac{1}{2}
\sum _{\substack{ x,y\in A\cap V_{n}  \\ x\sim y}}d(x,y)^{-p+1}|\Psi(x)-%
\Psi(y)|^{p}=\frac{1}{2}\left(\prod_{j=0}^{n} l_{j}^{p-1}\right) \sum
_{\substack{ x,y\in A\cap V_{n}  \\ x\sim y}}|\Psi(x)-\Psi(y)|^{p}.
\end{equation*}
For any $u \in \mathcal{F}_{p}$ and for each $n \geq 0$, we define $%
H_{n}u$ by Proposition \ref{l.affine}. Then%
\begin{equation*}
\sup _{n} \int_{A\cap\mathcal{S}}\left|\partial\left(H_{n}
u\right)\right|^{p} \dif\nu=\sup _{n}\frac{1}{2} \left(\prod_{j=0}^{n}
l_{j}^{p-1}\right) \sum_{\substack{ x,y\in A\cap V_{n}  \\ x\sim y}}%
|u(x)-u(y)|^{p}=\mathcal{E}_{p;A}(u)<\infty.
\end{equation*}%
The reflexivity of $L^{p}(\mathcal{S}, \nu)$ and Mazur's Lemma imply that,
there exists a convex combination of a subsequence of $\partial\left(H_{n}
u\right)$ that converges in $L^{p}(\mathcal{S}, \nu)$ to some $g \in$ $L^{p}(%
\mathcal{S}, \nu)$. Since $H_{n} u$ converges uniformly to $u$ by Proposition \ref{l.affine}, we have then
for every adjacent $x,y\in V_{n}$,%
\begin{equation*}
u(y)-u(x)=\int_{{e}(x,y)} g \dif\nu.
\end{equation*}
This proves (2). Furthermore, since a convex combination of a subsequence of
$\partial\left(H_{n} u\right)$ converges in $L^{p}(\mathcal{S}, \nu)$ to
some $g$, we have
\begin{equation*}
\int_{A\cap \mathcal{S}}|g|^{p} \dif\nu \leq \sup _{n} \int_{A\cap\mathcal{S}%
}\left|\partial\left(H_{n} u\right)\right|^{p} \dif\nu \leq \mathcal{E}%
_{p;A}(u).
\end{equation*}

We then show (2) implies (1). It follows from (2) and H\"older's inequality
that for $n \geq 0$,
\begin{equation*}
\begin{aligned} \left(\prod_{j=0}^{n} l_{j}^{p-1}\right)
\sum_{\substack{x,y\in A\cap V_{n}\\ x\sim y}}|u(x)-u(y)|^{p} & \leq
\left(\prod_{j=0}^{n} l_{j}^{p-1}\right)\sum_{\substack{x,y\in A\cap V_{n}\\
x\sim y}}\left(\int_{{e}(x, y)}|g| \dif\nu\right)^{p} \\ & \leq
\sum_{\substack{x,y\in A\cap V_{n}\\ x\sim y}} \int_{{e}(x, y)}|g|^{p}
\dif\nu \leq 2 \int_{A\cap\mathcal{S}}|g|^{p} \dif\nu. \end{aligned}
\end{equation*}
Hence%
\begin{equation*}
\mathcal{E}_{p;A}(u)=\sup_{n\geq0}\mathcal{E}_{p,n}(u)\leq \int_{A\cap \mathcal{S}
}|g|^{p} \dif\nu
\end{equation*}
and we deduce that $u\in \mathcal{F}_{p}$ with $\lVert u\rVert_{\mathcal{F}%
_{p}}^{p}\leq\lVert u\rVert_{ L^{p}(K, \mu)}^{p}+\lVert g\rVert_{ L^{p}(%
\mathcal{S}, \nu)}^{p}$.

If $g_{1},g_{2}$ both satisfy (2), then $\int_{e(x,y)}\left(
g_{1}-g_{2}\right) \dif\nu =0$ for all $n\geq 0$ and every adjacent $x,y\in
V_{n}$. Since for each $x,y\in V_{n}$, $e(x,y)$ is the union of $l_{n+1}$
edges in $V_{n+1}$, we may apply the Lebesgue differentiation Theorem to $%
(e(x,y),\nu |_{e(x,y)})$ (note that $\nu |_{e(x,y)}$ is the Lebesgue measure)
and conclude that $g_{1}-g_{2}=0$ $\nu $-a.e. on $e(x,y)$, thus on all $%
\mathcal{S}$. This proves the uniqueness.
\end{proof}

\begin{remark}
The uniqueness in Proposition \ref{p.grad} tells us more: for any $u\in%
\mathcal{F}_{p}$, if a convex combination of a subsequence of $%
\partial\left(H_{n} u\right)$ converges, then it must converge to $\partial
u$.
\end{remark}

We recall two inequalities that will be used later. These are easy extensions of \cite[Theorem 3.13 and Corollary 3.14]{BC23} for the scale-irregular Vicsek sets. {Since the geodesic distance and the Euclidean distance are bi-Lipschitz equivalent on a scale-irregular Vicsek set}, their proofs are virtually identical with that in \cite{BC23} and we omit them.
\begin{lemma}[Morrey's inequality, {\cite[Theorem 3.13]{BC23}}]
\label{l.morrey} Let $1< p<\infty$. {There exists a constant $C\geq1$ such that}, for every $u \in \mathcal{F}_{p}$ and $%
x, y \in A$,
\begin{equation*}
|u(x)-u(y)|^{p} \leq Cd(x, y)^{p-1} \mathcal{E}_{p;A}(u).
\end{equation*}
\end{lemma}

\begin{lemma}[Poincar\'e inequality, {\cite[Corollary 3.14]{BC23}}]
\label{l.pi}
Let $1<p<\infty$. {There exists a constant $C\geq1$ such that, for every closed convex set $A \subset K$ and every $u \in \mathcal{F}_{p}$,}
\begin{equation*}
\fint_{A}\left|u(x)-\fint_{A} u \dif \mu \right|^{p} \dif \mu(x) \leq C \diam%
(A)^{p-1} \mathcal{E}_{p;A}(u).
\end{equation*}
\end{lemma}

Combining the above facts, we prove Theorem \ref%
{t.energy}.

\begin{proof}[Proof of Theorem \protect\ref{t.energy}]
Let $(\mathcal{F}_{p},\lVert \cdot\rVert_{\mathcal{F}_{p}})$ be defined by %
\eqref{e.domain} and \eqref{e.energy}.

\begin{enumerate}[label=\textup{(\arabic*)}]
\item Let $\{u_{n}\}$ be a Cauchy sequence in $(\mathcal{F}%
_{p},\lVert\cdot\rVert_{\mathcal{F}_{p}})$. Then $\{u_{n}\}$ is also a
Cauchy sequence in $L^{p}(K,\mu)$. Assume that $u_{n}\rightarrow u$ in $%
L^{p}(K,\mu)$ and a subsequence $u_{n_{k}}\rightarrow u$ $\mu$-a.e.. Since $%
\mathcal{E}_{p}(u_{n}-u_{m})=\lVert \partial u_{n}-\partial
u_{m}\rVert_{L^{p}(\mathcal{S},\nu)}$, we know that $\{\partial u_{n}\}$ is a
Cauchy sequence in $L^{p}(\mathcal{S},\nu)$. Assume that $\partial
u_{n}\rightarrow g$ in $L^{p}(\mathcal{S},\nu)$. Fix a point $x_{0}\in K$ and let $%
h_{n}=u_{n}-u_{n}(x_{0})$. By Lemma \ref{l.morrey}, for all $m,n\geq1$ and
all $x\in K$,
\begin{equation*}
|h_{n}(x)-h_{m}(x)|^{p}\leq Cd(x,x_{0})^{p-1}\mathcal{E}_{p}(u_{n}-u_{m})\leq %
C\diam(K)^{p-1}\mathcal{E}_{p}(u_{n}-u_{m}),
\end{equation*}
which implies that $\{h_{n}\}_n$ is a Cauchy sequence in $C(K)$. Therefore, there exists $h\in
C(K)$ such that $h_{n}\rightarrow h$ in $C(K)$, and also in $L^{p}(K,\mu)$ by
H\"older's inequality. It is immediate that $u_{n}-h_{n}$ also converges to $%
u-h$ in $L^{p}(K,\mu)$. Thus $u_{n_{k}}(x_{0})=u_{n_{k}}-h_{n_{k}}$ converges
to $u-h$ $\mu$-a.e.. Hence $u-h$ must be a constant, say $u-h\equiv c $, and $%
u $ admits a continuous $\mu$-version on $K$, which will also be denoted by $%
u$. Note that
\begin{equation*}
\lVert u_{n_{k}}-u\rVert_{C(K)}\leq \lVert
h_{n_{k}}-h\rVert_{C(K)}+|u_{n_{k}}(x_{0})-c|\rightarrow
0\quad(k\rightarrow\infty).
\end{equation*}%
For every $n\geq0$ and every adjacent $x,y\in V_{n}$, we have by Proposition \ref%
{p.grad} that
\begin{equation*}
u_{n_{k}}(y)-u_{n_{k}}(x)=\int_{e(x, y)}\partial u_{n_{k}} \dif\nu.
\end{equation*}%
Letting $k\rightarrow\infty$, we have $u(y)-u(x)=\int_{e(x,
y)}g \dif\nu$ for $\mu$-a.e. $x,y$. By Proposition \ref{p.grad} again, we conclude that $u\in
\mathcal{F}_{p}$ and $\partial u=g$ $\nu$-a.e., thus
\begin{equation*}
\mathcal{E}_{p}(u_{n}-u)=\lVert \partial u_{n}-\partial u\rVert_{L^{p}(%
\mathcal{S},\nu)}\rightarrow0,
\end{equation*}
and then $u_{n}\rightarrow u$ in $\mathcal{F}_p$.

To prove that $(\mathcal{F}_{p},\lVert\cdot\rVert_{\mathcal{F}_{p}})$ is
uniformly convex, we first define the norm $\lVert\cdot\rVert_{p}$ on the product space $L^{p}(K, \mu)\times L^{p}(\mathcal{S},\nu)$ by $\lVert(u, v)\rVert_{p}:=\left(\lVert u\rVert_{L^{p}(K, \mu)}^{p}+\lVert v\rVert_{L^{p}(\mathcal{S}, \nu)}^{p}\right)^{1/p}$ for $(u,v)\in L^{p}(K, \mu)\times L^{p}(\mathcal{S},\nu)$, and a map $T:(\mathcal{F}_{p},\lVert\cdot\rVert_{\mathcal{F}_{p}})\rightarrow \left(L^{p}(K, \mu)\times L^{p}(\mathcal{S},\nu) ,\lVert\cdot\rVert_{p}\right)$ by $Tu:=(u,\partial u)$ for all $u\in \mathcal{F}_{p}$. Then by Proposition \ref{p.grad} and the above paragraph, we know that $T(\mathcal{F}_{p})$ is isometric to $\mathcal{F}_{p}$ and is a closed subspace of $\left(L^{p}(K, \mu)\times L^{p}(\mathcal{S},\nu) ,\lVert\cdot\rVert_{p}\right)$.  Since $L^{p}$ spaces are uniformly convex for $1<p<\infty$, we conclude by \cite[Theorem 1]{Cla36} that the product space $\left(L^{p}(K, \mu)\times L^{p}(\mathcal{S},\nu) ,\lVert\cdot\rVert_{p}\right)$ is uniformly convex, and so does its closed subspace $T(\mathcal{F}_{p})$. Since $\mathcal{F}_{p}$
and $T(\mathcal{F}_{p})$ are isometric, we conclude that $\mathcal{F}_{p}$
is uniformly convex. The reflexivity then follows from the uniform convexity and
Milman-Pettis theorem (see for example \cite[Theorem 3.31]{Bre11}).

The separability is obvious since the space of all piecewise affine
functions is dense in $\mathcal{F}_{p}$, and clearly there is a countable
dense subset of all piecewise affine functions.

\item For any $u,v\in \mathcal{F}_{p}$,
\begin{align*}
\mathcal{E}_{p,n}(uv)&=\frac{1}{2}\left(\prod_{j=0}^{n}
l_{j}^{p-1}\right)\sum_{\substack{ x,y\in V_{n}  \\ x\sim y}}%
\left|u(x)v(x)-u(y)v(y)\right|^{p} \\
&\leq \frac{1}{2}\left(\prod_{j=0}^{n} l_{j}^{p-1}\right)\sum_{\substack{ %
x,y\in V_{n}  \\ x\sim y}}2^{p-1}\left(\lVert
u\rVert_{C(K)}^{p}\left|v(x)-v(y)\right|^{p}+\lVert
v\rVert_{C(K)}^{p}\left|u(x)-u(y)\right|^{p}\right) \\
&\leq 2^{p-1}\left(\lVert u\rVert_{C(K)}^{p}\mathcal{E}_{p,n}(v)+\lVert
v\rVert_{C(K)}^{p}\mathcal{E}_{p,n}(u)\right).
\end{align*}%
Taking the supremum of $n$ on both sides, we have
\begin{equation*}
\mathcal{E}_{p}(uv)\leq2^{p-1}\left(\lVert u\rVert_{C(K)}^{p}\mathcal{E}%
_{p}(v)+\lVert v\rVert_{C(K)}^{p}\mathcal{E}_{p}(u)\right),
\end{equation*}
which means that the subset $\mathcal{F}_{p}\subset C(K)$ is an algebra under the
product operation.

\item The regularity follows directly from Proposition \ref{l.affine}.

\item This follows by noting that $|\varphi(u(x))-\varphi(u(y))|^{p}%
\leq|u(x)-u(y)|^{p}$.

\item This follows from Lemma \ref{l.pi} with $A=K$.

\item We may assume $a=0$ as $\mathcal{E}(v)=\mathcal{E}(v-a\mathds{1}_{K})$
and $\mathcal{E}(u+v)=\mathcal{E}(u+v-a\mathds{1}_{K})$ by definition. Write
$A=\supp(u)$ and $B=\supp(v)$. Then $A$ and $B$ are two disjoint
compact subsets of $K$ and thus $d(A,B)>0$. We can find $n_{0}$
sufficiently large so that for all $n\geq n_{0}$, the closed subsets $%
A_{n}:=\bigcup_{\substack{ w\in W_{n}  \\ K_{w}\cap A\neq\emptyset}}K_{w}$
and $B_{n}:=\bigcup_{\substack{ w\in W_{n}  \\ K_{w}\cap B\neq\emptyset}}%
K_{w}$ are also disjoint, and each $x\in A_{n}\cap V_{n}$ and $y\in
B_{n}\cap V_{n}$ are not adjacent. Then
\begin{align*}
\mathcal{E}_{p,n}(u+v)&=\frac{1}{2}\left(\prod_{j=0}^{n}
l_{j}^{p-1}\right)\sum_{\substack{ x,y\in V_{n}  \\ x\sim y}}%
\left|(u(x)-u(y))+(v(x)-v(y))\right|^{p} \\
&=\frac{1}{2}\left(\prod_{j=0}^{n} l_{j}^{p-1}\right)\left(\sum_{\substack{ %
x,y\in A_{n}\cap V_{n}  \\ x\sim y}}\left|u(x)-u(y)\right|^{p}+\sum
_{\substack{ x,y\in B_{n}\cap V_{n}  \\ x\sim y}}\left|v(x)-v(y)\right|^{p}%
\right) \\
&= \mathcal{E}_{p,n}(u)+\mathcal{E}_{p,n}(v).
\end{align*}
Letting $n\rightarrow\infty$, we derive $\mathcal{E}_{p}(u+v)=\mathcal{E}_{p}(u)+%
\mathcal{E}_{p}(v)$.
\end{enumerate}
\end{proof}
{
\begin{remark}\label{r.Clarkson}
In fact, we can show the uniform convexity of the semi-normed space $(\mathcal{F}_{p},\mathcal{E}_{p}^{1/p})$ by applying \cite[Proposition 3.5]{KS25} and proving the following \emph{$p$-Clarkson’s inequality}:
\begin{equation}\label{e.Clarkson}
\begin{dcases}
\mathcal{E}_{p}(f+g)+\mathcal{E}_{p}(f-g)\geq 2\left(\mathcal{E}_{p}(f)^{\frac{1}{p-1}}+\mathcal{E}_{p}(g)^{\frac{1}{p-1}}\right)^{p-1},\quad&\text{if }p\in(1,2],\\
\mathcal{E}_{p}(f+g)+\mathcal{E}_{p}(f-g)\leq 2\left(\mathcal{E}_{p}(f)^{\frac{1}{p-1}}+\mathcal{E}_{p}(g)^{\frac{1}{p-1}}\right)^{p-1},\quad&\text{if }p\in[2,\infty).
\end{dcases}
\end{equation}
Roughly speaking, since $\mathcal{E}_{p}(f)=\lVert \partial f\rVert^p_{L^{p}(\mathcal{S},\nu)}$ by Proposition \ref{p.grad}, and $(\mathcal{S},\nu)$ is $\sigma$-finite, we can apply \cite[(4) and (6) in p.462]{Bre11} to the gradients of $f$ and $g$ to obtain \eqref{e.Clarkson}.
\end{remark}
}
\begin{remark}
\label{r.inter} As shown in Proposition \ref{l.affine} and Theorem \ref{t.energy},
for two exponents $p,q\in(1,\infty)$ and any non-constant $u\in C(K)$, $%
H_{1}u\in \mathcal{F}_{p}\cap\mathcal{F}_{q}$. As we can easily choose $u$ so
that $H_{1}u$ is non-constant, we see that $\mathcal{F}_{p}\cap\mathcal{F}%
_{q}$ contains non-constant functions.
\end{remark}

{We conclude this subsection by stating a consequence of Proposition \ref{p.grad} and Lemma \ref{l.morrey}. We use the definition of the {\em $p$-resistance} $R_{p}(\cdot,\cdot)$ in \cite[p.5]{Yan25b}:
\begin{equation}
\label{e.Resdef}
R_{p}(x,y):=\sup\left\{\frac{\lvert u(x)-u(y)\rvert^{p}}{\mathcal{E}_{p}(u)}: u\in\mathcal{F}_{p} \text{ and }\mathcal{E}_{p}(u)>0 \right\},\ \forall x,y\in K.
\end{equation}
\begin{proposition}\label{p.Res}
There exists $C>1$ such that \begin{equation}
C^{-1}d(x,y)^{p-1}\leq R_{p}(x,y)\leq Cd(x,y)^{p-1},\ \forall x,y\in K.
\end{equation}
\end{proposition}
\begin{proof}
Lemma \ref{l.morrey} immediately implies that $R_{p}(x,y)\leq C d(x,y)^{p-1}$. To see the lower bound, for $n\geq1$ and $x,y \in V_{n}$, we choose a continuous function $u\in\mathcal{F}_{p}$ satisfying that $u(x)=1$, $u(y)=0$, $|\partial u|$ is constant over the geodesic cables connecting $x$ and $y$, and $u$ is piecewise constant over other cables. By a direct computation, $\mathcal{E}_{p}(u)^{-1/(p-1)}$ equals to the geodesic distance between $x$ and $y$. By the bi-Lipschitz equivalence of the Euclidean distance and the geodesic distance restricted to the skeleton, we see that \begin{equation}
R_{p}(x,y)\geq c d(x,y)^{p-1},\ \forall x,y\in\bigcup_{n}V_{n}.\label{e.Res0}
\end{equation} The lower bound is then established by the fact that $R_p:K\times K \to[0,\infty)$ is upper semi-continuous, that is $\limsup\limits_{n\to\infty}R_p(x_{n},y_{n})\leq R_p(x,y)$ if $x_{n}\to x$ and $y_{n}\to y$. Indeed, given $\epsilon>0$, we may choose $N$ large enough such that $d(x_{n},x)\vee d(y_{n},y)<\epsilon$ for all $n\geq N$. Then for any $u\in \mathcal{F}_{p}$ with $\mathcal{E}_{p}(u)>0$, we have\begin{align}
\left|\frac{u(x_{n})-u(y_{n})}{\mathcal{E}_{p}(u)^{1/p}}\right|&\leq \left|\frac{u(x_{n})-u(x)}{\mathcal{E}_{p}(u)^{1/p}}\right|+\left|\frac{u(x)-u(y)}{\mathcal{E}_{p}(u)^{1/p}}\right|+\left|\frac{u(y)-u(y_{n})}{\mathcal{E}_{p}(u)^{1/p}}\right|\\
&\leq Cd(x_{n},x)^{(p-1)/p}+ Cd(y_{n},y)^{(p-1)/p}+\left|\frac{u(x)-u(y)}{\mathcal{E}_{p}(u)^{1/p}}\right|\ \text{(by Lemma \ref{l.morrey})}\\
&\leq 2C\epsilon^{(p-1)/p}+\left|\frac{u(x)-u(y)}{\mathcal{E}_{p}(u)^{1/p}}\right|.\label{e.Resist}
\end{align}
Taking the supremum over all $u\in \mathcal{F}_{p}$ with $\mathcal{E}_{p}(u)>0$ in \eqref{e.Resist}, we see by definition that for all $ n\geq N$,
\begin{equation}
R_{p}(x_{n},y_{n})^{1/p}\leq 2C\epsilon^{(p-1)/p}+R_{p}(x,y)^{1/p},
\end{equation}
which shows the upper semi-continuity. Since $\bigcup_{n}V_{n}$ is dense in $K$ and $R_{p}(\cdot,\cdot)$ is upper semi-continuous, we can extend \eqref{e.Res0} to all $x,y\in K$, completing the proof.
\end{proof}
Some consequences of Proposition \ref{p.Res} are discussed in Section \ref{s.Discuss}.
}

\subsection{Associated $p$-energy measure}

After constructing the $p$-energy norm, it is nature to consider the corresponding $p$%
-energy measure. It is shown by Murugan and
Shimizu in \cite[Section 9]{MS25} that the $p$-energy measure with good
properties on the standard Sierpi\'nski carpet can be constructed, which
heavily relies on self-similarity. Nevertheless, we can use the
special gradient structure of scale-irregular Vicsek sets to achieve our aim.
Before discussing the $p$-energy measure, we record some properties
of the operator $\partial$.

\begin{lemma}
\label{l.weakd} For any $u\in\mathcal{F}_{p}$ and any adjacent $x,y$, the
function $u_{e(x,y)}:(0,1)\rightarrow \mathbb{R}$ defined by $%
u_{e(x,y)}(t)=u(e(x,y)(t))$ belongs to $W^{1,p}((0,1))$, and its weak derivative $%
Du_{e(x,y)}$ can be chosen as $t\mapsto \partial u(e(x,y)(t))$.
\end{lemma}

\begin{proof}
Fix $u\in\mathcal{F}_{p}$. Let $g_{e(x,y)}(t)= \partial u(e(x,y)(t))$, then $%
g_{e(x,y)}\in L^{p}((0,1))$. By the continuity of $u$ and the density of $%
\bigcup_{n}V_{n}$ in $K$, we can extend \eqref{e.grad} to
\begin{equation}  \label{e.weakd}
u_{e(x,y)}(b)-u_{e(x,y)}(a)=\int_{(a,b)}g_{e(x,y)}\dif \mathscr{L}%
^{1},\quad0\leq a<b\leq1,
\end{equation}%
where $\mathscr{L}^{1}$ is the Lebesgue measure on $\mathbb{R}$. For any
$\phi\in C_{c}^{\infty}((0,1))$, by Fubini's theorem and the fundamental
theorem of calculus, we have
\begin{align*}
&\int_{0}^{1}u_{e(x,y)}(t)\phi^{\prime}(t)\dif\mathscr{L}^{1}(t) \\
=\ & \int_{0}^{1}\left(\int_{(0,t)}g_{e(x,y)}(s)\dif \mathscr{L}%
^{1}(s)\right)\phi^{\prime}(t)\dif\mathscr{L}^{1}(t)+u_{e(x,y)}(0)%
\int_{0}^{1}\phi^{\prime}(t)\dif\mathscr{L}^{1}(t)\ \text{(by \eqref{e.weakd}%
)} \\
=\ & \int_{0}^{1}\left(\int_{(s,1)}\phi^{\prime}(t)\dif\mathscr{L}%
^{1}(t)\right)g_{e(x,y)}(s)\dif \mathscr{L}^{1}(s)\ \text{(by Fubini's
theorem and $\phi(1)=\phi(0)=0$ )} \\
=\ & \int_{0}^{1}(\phi(1)-\phi(s))g_{e(x,y)}(s)\dif \mathscr{L}^{1}(s)
=-\int_{0}^{1}\phi(s)g_{e(x,y)}(s)\dif \mathscr{L}^{1}(s).
\end{align*}
Thus $u_{e(x,y)}\in W^{1,p}((0,1))$ and $Du_{e(x,y)}=g_{e(x,y)}$.
\end{proof}
The following properties on $\partial$ are generalizations of {\cite[Proposition 2.2]{BC24}}.
\begin{proposition}[{\cite[Proposition 2.2]{BC24}}]
\label{p.grad.cal} Let $\partial: u\mapsto \partial u,\ (u\in\mathcal{%
F}_{p})$ be defined as in Proposition \ref{p.grad}. The following properties hold.

\begin{enumerate}[label=\textup{(\arabic*)}]
\item {\normalfont{(Linearity)}} For any two $u_{1},u_{2}\in\mathcal{F}_{p}$%
, $\partial(u_{1}+u_{2})=\partial u_{1}+\partial u_{2}$.

\item {{\normalfont{(Leibniz rule)}} For any two $u_{1},u_{2}\in\mathcal{F}%
_{p}$, $u_{1}u_{2}\in\mathcal{F}_{p}$ and $\partial
(u_{1}u_{2})=u_{1}\partial u_{2}+u_{2}\partial u_{1}$.}

\item {\normalfont{(Chain rule)}} For any $f \in C^{1}(\mathbb{R})$ and
any $u\in \mathcal{F}_{p}$, $\partial (f \circ u)=f ^{\prime
}(u)\partial u$.

\item {\normalfont{(Closedness)}} The operator $\partial :\mathcal{F}%
_{p}\rightarrow L^{p}(\mathcal{S},\nu)$ is closed, if we view $\partial $ as
an unbounded operator on $C(K)$.
\end{enumerate}
\end{proposition}

\begin{proof}
The assertions (1), (2) and (3) are obtained by Lemma \ref{l.weakd}, the linearity,
Leibniz rule, and chain rule of the weak derivatives, respectively. The assertion (4) is an
immediate consequence of Proposition \ref{p.grad}.
\end{proof}

After these preparations, we can prove Theorem \ref{t.emc}.

\begin{proof}[Proof of Theorem \protect\ref{t.emc}]
Define
\begin{equation}  \label{e.emea}
\Gamma_{p}\langle u\rangle(A):=\int_{A\cap\mathcal{S}}\left|\partial
u\right|^{p}\dif\nu,\text{ i.e. } \dif \Gamma_{p}\langle
u\rangle:=\left|\partial u\right|^{p}\dif\nu.
\end{equation}
We will prove that $\{\Gamma_{p}\langle u\rangle\}_{u\in\mathcal{F}_{p}}$ on
$K $ is a family of Borel finite measures having the properties stated in Theorem \ref%
{t.emc}.

\begin{enumerate}[label=\textup{(\arabic*)}]
\item That $\Gamma_{p}\langle u\rangle(K)=\mathcal{E}_{p}(u)$ follows from
Proposition \ref{p.grad}. The ``if" part in the second assertion is
trivial by definition. For the converse, if $\Gamma_{p}\langle u\rangle=0$,
then $\mathcal{E}_{p}\left(u\right)=0$. By Morrey's inequality in Lemma \ref%
{l.morrey}, $u$ must be constant.

\item Note that for any non-negative Borel measurable function $g$ on $K$,
\begin{equation*}
\left(\int_{K}g\dif\Gamma_{p}\langle u\rangle\right)^{1/p}=\left(\int_{%
\mathcal{S}}\left(|g|^{1/p}\left|\partial u\right|\right)^{p}\dif%
\nu\right)^{1/p}=\lVert |g|^{1/p}\left|\partial u\right|\rVert_{L^{p}(%
\mathcal{S},\nu)}.
\end{equation*}%
So \eqref{e.em.tri0} holds by the Minkowski inequality of $L^{p}(\mathcal{S%
},\nu)$.

\item The identity \eqref{e.lip1} holds by the definition of $\Gamma_{p}\langle
u\rangle$ and the Leibniz rule in Proposition \ref{p.grad.cal}.

\item The proof is essentially the same as in \cite[Theorem 4.3.8]%
{CF12} and \cite[Proposition 7.6]{Shi24}. Since all
compact subsets generate the Borel $\sigma$-algebra of $\mathbb{R}$, we only
need to prove that $u_{*} \left(\Gamma_{p}\langle u\rangle\right)(F)=0$
whenever $u \in \mathcal{F}_p$ and $F$ is a compact subset of $\mathbb{R}$
with $\mathscr{L}^1(F)=0$. We can choose a sequence $\left\{\phi_n\right%
\}_{n \geq 1}\subset C^{\infty}_{c}(\mathbb{R})$ such that $%
\left|\phi_n\right| \leq 1, \lim _{n \rightarrow \infty} \phi_n(x)=\mathds{1}%
_{F}(x)$ for each $x \in \mathbb{R}$, and
\begin{equation*}
\int_0^{\infty} \phi_n(t) \dif t=\int_{-\infty}^0 \phi_n(t) \dif t=0
\end{equation*}
for each $n \in \mathbb{N}$. Let $\Phi_n(x):=\int_0^x \phi_n(t) \dif t$ for
each $x \in \mathbb{R}$ and $n \in \mathbb{N}$. Then we see that $\Phi_n \in
C^1(\mathbb{R})$ with compact support, $\Phi_n(0)=0$, and $%
\left|\Phi_n^{\prime}(x)\right| \leq 1$ for each $x \in \mathbb{R}$ and $n \in \mathbb{N}$. By the
dominated convergence theorem, we know that $\lim _{n \rightarrow
\infty} \Phi_n(x)=0$ for each $x \in \mathbb{R}$ and $\Phi_n \circ u$
converges to 0 in $L^p(K, \mu)$. Since $\mathcal{E}_p\left(\Phi_n \circ
u\right) \leq \mathcal{E}_p(u)$ by the Lipschitz contractivity of $\mathcal{E%
}_p$ in Theorem \ref{t.energy}, we know that $\left\{\Phi_n\circ u\right\}_{n
\geq 1}$ is bounded in $\mathcal{F}_{p}$. By the uniform convexity of $\mathcal{F%
}_{p}$ and \cite{Kak39}, there exists a subsequence $\{n_{k}\}_{k\geq 1}$ of $\mathbb{N}$ such that
the arithmetic mean $\Psi_{j} \circ u:=\frac{1}{j}\sum_{k=1}^{j}%
\Phi_{n_k}\circ u\rightarrow 0$ in $\mathcal{F}_p$ as $j\rightarrow \infty$%
. For each $x\in\mathbb{R}$, since
\begin{equation*}
\left|\mathds{1}_{F}(x)-\frac{1}{j}\sum_{k=1}^{j}\phi_{n_{k}}(x)%
\right|=\left|\frac{1}{j}\sum_{k=1}^{j}\left(\mathds{1}_{F}(x)-{%
\phi_{n_{k}}(x)}\right)\right|\rightarrow 0 \text{ as $j\rightarrow \infty$},
\end{equation*}
we conclude by Fatou's lemma that
\begin{equation*}
\begin{aligned} \left(u_{*}\Gamma_{p}\langle u\rangle\right)(F)&
=\int_{\mathbb{R}} \lim_{j \rightarrow
\infty}\left|\frac{1}{j}\sum_{k=1}^{j}
\Phi_{n_k}^{\prime}(t)\right|^p  \dif  \left(u_{*}\Gamma_{p}\langle
u\rangle\right)( t) \\ & \leq \liminf_{j \rightarrow \infty}
\int_K\left|\Psi_{j}^{\prime}(u(x))\right|^p \dif \Gamma_{p}\langle
u\rangle(x) \\ & =\liminf_{j \rightarrow \infty}
\Gamma_{p}\left\langle\Psi_{j} \circ u\right\rangle(K)=\liminf_{l
\rightarrow \infty} \mathcal{E}_p\left(\Psi_{j} \circ u\right)=0 .
\end{aligned}
\end{equation*}
\end{enumerate}
\end{proof}

\begin{remark}
\label{r.medm}

\begin{enumerate}[label=\textup{(\arabic*)}]
\item Theorem \ref{t.emc} shows that it is possible to define $p$-energy measures
on some fractals without self-similarity. However, it is still an open problem to give a general procedure to define $p$-energy and associated energy measure on general Moran fractals, such as scale-irregular Sierpi\'nski gaskets.

\item By definition, $\Gamma_{p}\langle u\rangle\ll \nu$ for all $u\in%
\mathcal{F}_{p}$. So $\nu$ is a \textit{minimal energy-dominant measure} in
the sense of \cite{Hin10}. Since $\nu$ is independent of $p$, this
gives an example of $p$-energy on fractals whose minimal
energy-dominant measure for different exponents can be absolutely continuous
(with or without self-similarity).
\end{enumerate}
\end{remark}

The approach in \cite[Section 9]{MS25}, using the word space of a fractal, provides another way to construct energy measures, which also works in the scale-irregular Vicsek fractal setting with minor modification. Proposition \ref{p.coinc} shows that the
construction using the word space agree with that using the
gradient structure in Theorem \ref{t.emc}.

\begin{proposition}
\label{p.em} For any $u\in\mathcal{F}_{p}$ and $n\geq1 $, the measure $m_{p}^{(n)}\langle u\rangle$ defined by
\begin{equation*}
E\mapsto \sum_{w\in E}\mathcal{E}%
_{p;K_{w}}\left(u\right)=:m_{p}^{(n)}\langle u\rangle(E), \quad\forall
E\subset W_{n}
\end{equation*}
satisfies
\begin{equation}  \label{e.em1}
\sum_{v\in S(w)}m_{p}^{(n+1)}\langle u\rangle(\{v\}) =m_{p}^{(n)}\langle
u\rangle(\{w\}).
\end{equation}
\end{proposition}

\begin{proof}
This follows directly from \eqref{e.grad-1} that
\begin{align*}
\sum_{v\in S(w)}\mathcal{E}_{p;K_{v}}\left(u\right)&=\sum_{v\in S(w)}\int_{%
\mathcal{S}\cap K_{v}}|\partial u|^p\dif\nu \\
&=\int_{\mathcal{S}\cap K_{w}}|\partial u|^p\dif\nu-\frac{1}{2}%
\sum_{v,v^{\prime}\in S(w)}\int_{\mathcal{S}\cap K_{v}\cap
K_{v^{\prime}}}|\partial u|^p\dif\nu \\
&=\int_{\mathcal{S}\cap K_{w}}|\partial u|^p\dif\nu=\mathcal{E}%
_{p;K_{w}}\left(u\right)
\end{align*}
since $\bigcup_{v\in S(w)}K_{v}=K_{w}$ and $\nu$ has no atom.
\end{proof}

By the Kolmogorov's extension theorem and Proposition \ref{p.em}, we obtain
a finite Borel measure $m_{p}\langle u\rangle $ on $W_{\infty }$ such that
\begin{equation*}
m_{p}\langle u\rangle \left( \left\{ \tau \in W_{\infty }:[\tau
]_{n}=w\right\} \right) =\mathcal{E}_{p;K_{w}}\left( u\right) ,\quad \forall
n\geq 1,\ w\in W_{n}.
\end{equation*}%
Clearly, $m_{p}\langle u\rangle (W_{\infty })=\mathcal{E}_{p}\left( u\right)
$. Also $m_{p}\langle u\rangle $ is non-atomic, since $m_{p}\langle u\rangle \left( w\right) \leq \mathcal{E}%
_{p;K_{[w]_{n}}}\left( u\right) \rightarrow 0$ as $n\rightarrow \infty $ for any
$w\in W_{\infty }$. Now we show that, these two different ways give the same $p$-energy measure.

Recall that $\chi$ is the coding map in Proposition \ref{p.coding}.
\begin{proposition}
\label{p.coinc} For any $u\in\mathcal{F}_{p}$, the push-forward of $m_{p}\langle u\rangle$ under $\chi$ coincides with $\Gamma_{p}\langle u\rangle$, namely, $\Gamma_{p}\langle u\rangle=\chi_{*}m_{p}\langle u\rangle$, where $\chi_{*}m_{p}\langle u\rangle(\cdot):=m_{p}\langle u\rangle(\chi^{-1}(\cdot))$.
\end{proposition}

\begin{proof}
We first prove that for any $w\in W_{*}$, $\Gamma_{p}\langle
u\rangle(K_{w})=\chi_{*}m_{p}\langle u\rangle(K_{w})$, namely,
\begin{equation*}
m_{p}\langle u\rangle(\chi^{-1}(K_{w}))=\mathcal{E}_{p;K_{w}}\left(u\right).
\end{equation*}
Assume that $w\in W_{n}$. Since $\left\{\tau\in W_{\infty}:[\tau]_{n}=w
\right\}\subset\chi^{-1}(K_{w}) $, we first have
\begin{equation*}
m_{p}\langle u\rangle(\chi^{-1}(K_{w}))\geq\mathcal{E}_{p;K_{w}}\left(u%
\right).
\end{equation*}
If $\tau\in \chi^{-1}(K_{w})$ but $[\tau]_{n}\neq w$, then by the last
assertion in Proposition \ref{p.coding}, $\chi(\tau)$ must belong to $%
K_w \cap V_n$. So the set $\chi^{-1}(K_{w})\setminus \left\{\tau\in
W_{\infty}:[\tau]_{n}=w \right\}$ has countably many elements. Since $%
m_{p}\langle u\rangle$ is non-atomic, we must have $m_{p}\langle
u\rangle(\chi^{-1}(K_{w}))=\mathcal{E}_{p;K_{w}}\left(u\right)$. We write
$\mathcal{P}$ as the collection of all $K_{w}(w\in W_{*})$ and all singletons in $K$, and $\mathcal{L}$ as the collection of all Borel
subsets $B$ of $K$ such that $\Gamma_{p}\langle u\rangle
(B)=\chi_{*}m_{p}\langle u\rangle(B)$. It is easy to verify that $\mathcal{P}
$ forms a $\pi$-system and $\mathcal{L}$ forms a $\lambda$-system. A
standard application of $\pi-\lambda$ theorem shows that $\mathcal{L}$
contains $\sigma(\mathcal{P})$, the $\sigma$-algebra generated by $\mathcal{P%
}$. It suffices to show that $\sigma(\mathcal{P})$ is the Borel $\sigma$%
-algebra of $K$. To show this, we only need to show that every closed subset
$F$ of $K$ is in $\sigma(\mathcal{P})$. Let $F_{m}=\bigcup_{\substack{ w\in
W_{m}  \\ K_{w}\cap F\neq\emptyset  \\ }}K_{w}$ so that $F_{m}\in \sigma(%
\mathcal{P})$. Since $\max_{w\in W_{m}}\diam( K_{w})\rightarrow 0$ as $%
m\rightarrow \infty$, we have $\bigcap_{m\geq1}F_{m}=F$. Therefore $F\in
\sigma(\mathcal{P})$ and we complete the proof.
\end{proof}

We conclude this section by comparing the energy measure constructed here
with that in \cite{Kuw24}. In the case $p=2$, Theorem \ref{t.energy} gives a
regular Dirichlet form $(\mathcal{E}_{2},\mathcal{F}_{2})$ on a Vicsek set in
the definition of \cite[Chapter 1]{FOT11}, and Theorem \ref{t.emc} gives the
energy measure with respect to $(\mathcal{E}_{2},\mathcal{F}_{2})$ in the
definition of \cite[Chapter 3]{FOT11}. As we said in Remark \ref{r.medm}, $%
\nu$ is a minimal energy-dominant measure for the Dirichlet form $(\mathcal{E%
}_{2},\mathcal{F}_{2})$. For every $u\in\mathcal{F}_{2}$, we see from %
\eqref{e.emea} that the Radon-Nikodym derivative is
\begin{equation*}
\frac{\dif\Gamma_{2}\langle u\rangle}{\dif\nu}=|\partial u|^{2}
\end{equation*}
and therefore, at least formally,
\begin{equation*}
\mathcal{E}_{p}(u)=\int_{K}\left|\frac{\dif\Gamma_{2}\langle u\rangle}{\dif%
\nu}\right|^{p/2}\dif\nu\text{ and }\dif\Gamma_{p}\langle u\rangle=\left|%
\frac{\dif\Gamma_{2}\langle u\rangle}{\dif\nu}\right|^{p/2}\dif\nu.
\end{equation*}%
It turns out that the $p$-energy norm and $p$-energy measure in this paper
are also equivalent to those in \cite{Kuw24}.

\section{Besov-Lipschitz norms and their properties}

\label{s.norm}
Throughout this section, we fix a contraction ratio sequence $\bm{l}$ satisfying {$\sup\limits_{n\geq 1}l_n<\infty$, and fix a $\beta^{*}\in(0,\infty)$}.

We write $K^{\bm{l}}=K,\ K_{w}^{\bm{l}}=K_{w}$ and omit the
index $p$ when no confusion is caused. We first investigate some basic properties of Besov-Lipschitz spaces in Section \ref{subs.bls}. The norm equivalences and the critical exponent in Theorem \ref{thm3} will be proved in Section \ref{subs.noeq}. The weak monotonicity property and BBM convergence in Theorem \ref{thm3} will be proved in Section \ref{subs.wm}.
\subsection{Besov-Lipschitz spaces related to the $p$-energy}\label{subs.bls}

Recall \eqref{ksai} and define a function
\begin{equation}
\phi (r)=\begin{dcases} \rho_{n}^{p-1}\psi(\rho_{n}) &\text{ for
}\rho_{n+1}< r\leq \rho_n\ (n\geq0),\\ 2^{p-1} &\text{ for } r \geq2. \end{dcases}
\label{phi}
\end{equation}

{
\begin{remark}\label{r.coinBau}
If we assume that $\bm{l}$ consists of only one odd number $l\geq3$ and choose $\beta^{*}:=1+(\alpha_{l}-1)/p$ in Definition \ref{def[}, then $B_{p,\infty}^{\beta}$ (defined with respect to this particular $\beta^{*}$) is the same as $\mathcal{B}^{\beta,p}$ in \cite{Bau24}.
\end{remark}
}

Recall Definition \ref{def[}. Using \eqref{scaling1}, the definition of $\rho_{n}$ and $\sup\limits_{n\geq 1}l_n<\infty$, there exists a constant $C>0$ such that, for all $n\geq0$, $\rho_{n+1}< r\leq \rho_n$ and all $u\in L^{p}(K,\mu)$,
\begin{equation}\label{e.comphi}
•C^{-1}\Phi_{u}^{\beta}(\rho_{n+1})\leq \Phi_{u}^{\beta}(r)\leq C\Phi_{u}^{\beta}(\rho_n).
\end{equation}
With these notions, we generalize the $p$-energy norm given in Definition \ref{penergy_1}.
\begin{definition}
\label{defE} For every $1<p<\infty $, $0\leq\beta <\infty $, $%
n\in \mathbb{N}$  and every $u\in C(K)$, define
\begin{equation}
\mathcal{E}_{n}^{\beta }(u):=\frac{1}{2}\phi (\rho _{n})^{-\frac{\beta }{
\beta^{\ast }}}\psi (\rho _{n})\sum_{\substack{ x,y\in V_{n}  \\ x\sim
y }}\left\vert u(x)-u(y)\right\vert ^{p},  \label{En}
\end{equation}
and
\begin{equation}
\mathcal{E}_{p,\infty }^{\beta }(u):=\sup_{n\geq 0}\mathcal{E}_{n}^{\beta
}(u),\quad \mathcal{E}_{p,p}^{\beta }(u):=\sum_{n=0}^{\infty }\mathcal{E}
_{n}^{\beta }(u).  \label{penergy1}
\end{equation}
\end{definition}

\begin{remark}
\begin{enumerate}[label=\textup{(\arabic*)}]
\item In view of \eqref{e.energy} and \eqref{ve1}, we have for each $n\in\mathbb{N}$ and all $u\in C(K)$ that $\mathcal{E}_{p,n;K}(u)=\mathcal{E}_{n}^{\beta^{\ast }}(u)$. Therefore,
\begin{equation} \label{E5}
\mathcal{E}_{p}(u)=\mathcal{E}_{p,\infty}^{\beta^{\ast }}(u)=\sup_{n\geq 0}\mathcal{E}%
_{n}^{\beta^{\ast }}(u)=\limsup_{n\rightarrow \infty }\mathcal{E}%
_{n}^{\beta^{\ast }}(u).
\end{equation}
\item We may write $\mathcal{E}_{p,p }^{\beta }(u)$ in a ``non-local $p$-energy'' manner:\begin{align}
\mathcal{E}_{p,p}^\beta(u)=\frac{1}{2}\sum_{n = 0}^{\infty} \phi (\rho _{n})^{-\frac{\beta }{
\beta^{\ast }}}\psi (\rho _{n}) \sum_{\substack{ x,y\in V_{n}  \\ x\sim
y }} |u (x) - u (y) |^p  =\iint_{(K\times K)\setminus \mathrm{diag}}|u(x)-u(y)|^p\dif J_{\beta}(x,y)
\label{219}
\end{align}
with the symmetric positive measure
\begin{equation}
\dif J_{\beta}(x,y)=\frac{1}{2}\sum_{n = 0}^{\infty} \sum_{\substack{ v,w\in V_{n}  \\ v\sim
w }}\phi (\rho _{n})^{-\frac{\beta }{
\beta^{\ast }}}\psi (\rho _{n}) \dif\delta_v(x)\dif\delta_w(y)  \label{218}
\end{equation}
defined on $(K\times K)\setminus \mathrm{diag}$, where $\mathrm{diag}:=\{(x,x):x\in K\}$ and $\delta _{v}$ is the Dirac measure at point $v$.
\end{enumerate}
\end{remark}

For any $1<p<\infty$, we define a real number $\epsilon_{p}$ by
\begin{equation}
 \epsilon_{p}:=\left( 1+\frac{p-1}{{\sup\limits_{n\geq 1}\alpha_{l_n}}}\right) ^{-1}{\in (0,1)},\label{e.epsilonp}
\end{equation}
{where $\alpha_{l}$ is given in \eqref{e.df}. For each odd integer $l\geq 3$, let
\begin{equation}
 \beta_{l}^{(p)}:=p-1+\alpha_{l}.\label{bln}
\end{equation}
}
{For the non-self-similar case, we need the following estimates.}
\begin{proposition}\label{condB_p1}
Let $1<p<\infty$.
\begin{enumerate}[label=\textup{(\arabic*)}]
\item\label{lb.est1} For any $\beta\in (\epsilon_{p}\beta^{\ast},\infty)$, we have {$\inf\limits_{n\geq 1}\left( \frac{\beta }{\beta^{\ast }}\beta_{l_n}^{(p)}-\alpha _{l_n}\right) >0$}.     For every $\delta $ satisfying
\begin{equation}
0\leq \delta <{\inf\limits_{n\geq 1}\left( \frac{\beta }{\beta^{\ast }}\beta_{l_n}^{(p)}-\alpha _{l_n}\right)},  \label{delta}
\end{equation}
there exists $C=C(\beta ,\delta)$ such that for any integer $n\geq 0$,
\begin{align}
\sum_{k=n}^{\infty }\phi(\rho _{k})^{\frac{\beta }{\beta^{\ast }}}\psi(\rho _{k})^{-1}\rho _{k}^{-\delta }&\leq C\phi(\rho _{n})^{\frac{\beta }{\beta^{\ast}}}\psi (\rho _{n})^{-1}\rho _{n}^{-\delta },\label{e.est1}\\
\sum_{k=0}^{n} \phi(\rho_k)^{-\frac{\beta}{\beta^{\ast}}}\psi(\rho_k)\rho_k^{\delta}&\leq C \phi(\rho_n)^{-\frac{\beta}{\beta^{\ast}}}\psi(\rho_n)\rho_n^{\delta}.\label{e.est2}
\end{align}
Moreover, the sequence
\begin{equation}
\left\{\phi(\rho _{k})^{\frac{\beta }{\beta^{\ast }}}\psi (\rho
	_{k})^{-1}\rho _{k}^{-\delta }\right\}_{k\geq0}\  \text{decreases to }\ 0 \ \text{as}\ k\rightarrow\infty.\label{d1}
\end{equation}

\item\label{lb.est2} For any integer $n\geq 0$ and any $\delta >0$, we have
\begin{equation}\label{condB_p3}
\sum_{k=n}^{\infty }\phi(\rho _{k})^{\delta }\in \left[ \frac{\phi(\rho _{n})^{\delta
}}{1-\left({\sup\limits_{n\geq 1}t_{l_n}}\right)^{-\delta }},\frac{\phi(\rho
_{n})^{\delta }}{1-\left({\inf\limits_{n\geq 1}t_{l_n}}\right)^{-\delta }}\right] ,
\end{equation}
where $t_{l}:=(2l-1)l^{p-1}$.
\end{enumerate}
\end{proposition}

\begin{proof}
\begin{enumerate}[label=\textup{(\arabic*)}]
\item  That {$\inf\limits_{n\geq 1}\left( \frac{\beta }{\beta^{\ast }}\beta_{l_n}^{(p)}-\alpha _{l_n}\right) >0$} follows by a simple calculation {and the assumption that $\{l_n\}_{n\geq 1}$ is a finite set of odd numbers}. To show the rest, let $a_{k}:=\phi(\rho _{k})^{\frac{\beta }{\beta^{\ast}}}\psi(\rho _{k})^{-1}\rho _{k}^{-\delta}=\rho_{k}^{(p-1)\frac{\beta}{\beta^{\ast}}-\delta}\psi(\rho_{k})^{\frac{\beta }{\beta^{\ast }}-1}$, we have for each $j\geq n$,
\begin{align}\label{e.quotient}
\frac{a_{j+1}}{a_{j}}\leq \max\left\{l^{\delta-(p-1)\frac{\beta}{\beta^{\ast}}}(2l-1)^{1-\frac{\beta}{\beta^{\ast}}}:\ {l\in\{l_n\}_{n\geq 1}}\right\}:=c_{0}<1.
\end{align}
Therefore,
\begin{align}\label{e.quotient+}
\sum_{k=n}^{\infty }\phi(\rho _{k})^{\frac{\beta }{\beta^{\ast }}}\psi(\rho _{k})^{-1}\rho _{k}^{-\delta }&=a_{n}+\sum_{k=n+1}^{\infty}a_{n}\prod_{j=n}^{k-1}\frac{a_{j+1}}{a_{j}}\leq a_{n}\sum_{k=n}^{\infty}c_{0}^{k-n}\notag\\ &\leq Ca_{n} =C\phi(\rho _{n})^{\frac{\beta }{\beta^{\ast }}}\psi (\rho _{n})^{-1}\rho _{n}^{-\delta },
\end{align}
which implies \eqref{e.est1}. For $0\leq k \leq n$, let  $b_{k}:=a_{n-k}^{-1}=\phi(\rho_{n-k})^{-\frac{\beta}{\beta^*}}\psi(\rho_{n-k})\rho_{n-k}^{\delta}$. Then  \[\frac{b_{n-i}}{b_{n-i-1}}=\frac{a_{i+1}}{a_{i}}\leq c_{0}.
\]
Therefore \begin{align}
\sum_{k=0}^{n} \phi(\rho_k)^{-\frac{\beta}{\beta^{\ast }}}\psi(\rho_k)\rho_k^{\delta}&=\sum_{k=0}^{n}b_{n-k}=b_{0}+\sum_{k=0}^{n-1}b_{0}\prod_{j=0}^{n-k-1}\frac{b_{n-k-j}}{b_{n-k-j-1}}\\
&\leq b_{0}\sum_{k=0}^{n}c_{0}^{n-k}\leq Cb_{0}=C\phi(\rho_n)^{-\frac{\beta}{\beta^{\ast }}}\psi(\rho_n)\rho_n^{\delta},
\end{align}
which implies \eqref{e.est2}.
By \eqref{e.quotient} and the convergence of the series in \eqref{e.quotient+}, we see that the sequence $\left\{\phi(\rho _{k})^{\frac{\beta }{\beta^{\ast }}}\psi (\rho_{k})^{-1}\rho _{k}^{-\delta }\right\}_{k\geq1}$ decreases to $0$ as $k\rightarrow\infty$.
\item Since { $\phi(\rho _{k+1})^{\delta }\phi(\rho _{k})^{-\delta }$ belong to a finite set $\left\{ t_{l_n}^{-\delta }\right\}_{n\geq 1}$} for all $k\geq0$, the proof is virtually identical with that of (1) and we omit it.\end{enumerate}
\end{proof}
In Lemma \ref{l.morrey}, we study the Morrey-Sobolev inequality for the space $\mathcal{F}_{p}$. Such inequality still
holds for the Besov spaces $B_{p,\infty }^{\beta }$ and $B_{p,p}^{\beta }$ when $\beta\in (\epsilon_{p}\beta^{\ast},\infty)$, {where $\epsilon_{p}$ is given in \eqref{e.epsilonp}.} The proof is adapted from \cite[Theorem 4.11(iii)]%
{GHL03} ($p=2$ therein) and does not use the gradient structure as in Lemma \ref{l.morrey}.

\begin{lemma}[Morrey-Sobolev inequality]
\label{lemmaMR}
For any $\beta\in (\epsilon_{p}\beta^{\ast},\infty)$ and any $u\in B_{p,\infty }^{\beta }$, there exists a continuous
version $\widetilde{u}\in C(K)$ satisfying $\widetilde{u}=u$ $\mu $-almost
everywhere on $K$ and
\begin{equation}
|\widetilde{u}(x)-\widetilde{u}(y)|^{p}\leq C\phi (d(x,y))^{\frac{\beta }{%
\beta^{\ast }}}\psi (d(x,y))^{-1}[u]_{B_{p,\infty }^{\beta }}^{p},
\label{MS_eq}
\end{equation}%
for all $x,$ $y\in K$, where $C$ is a positive constant.
\end{lemma}

\begin{proof}
For any $x \in K$ and $0<r\leq 2/3$, define
\begin{equation*}
u_{r}(x):=\frac{1}{\mu(B(x, r))} \int_{B(x, r)} u(\xi) \dif \mu(\xi).
\end{equation*}

We claim that for any $u \in B_{p,\infty }^{\beta }$, and all $x, y
\in K$ with $r=d(x, y) <2/3$, the following inequality holds:
\begin{equation}
\left|u_{r}(x)-u_{r}(y)\right| \leq C \psi (r)^{-1/p}\phi (r)^{\frac{\beta }{p\beta^{\ast }}%
}\sup_{r\in (0,3d(x,y)]}\left(\Phi _{u}^{\beta }(r)\right)^{1/p}.  \label{20-3}
\end{equation}

Indeed, letting $B_{1}=B(x,r)$, $B_{2}=B(y,r)$, we have
\begin{equation*}
u_{r}(x)=\frac{1}{\mu \left( B_{1}\right) }\int_{B_{1}}u(\xi )\dif\mu (\xi )=%
\frac{1}{\mu \left( B_{1}\right) \mu \left( B_{2}\right) }%
\int_{B_{1}}\int_{B_{2}}u(\xi )\dif\mu (\eta )\dif\mu (\xi ),
\end{equation*}%
and
\begin{equation*}
u_{r}(y)=\frac{1}{\mu \left( B_{1}\right) \mu \left( B_{2}\right) }%
\int_{B_{1}}\int_{B_{2}}u(\eta )\dif\mu (\eta )\dif\mu (\xi ).
\end{equation*}%
Assume that $\rho _{k+1}\leq 3r< \rho _{k}$, by the H{\"{o}}lder's inequality,
\begin{align*}
\left\vert u_{r}(x)-u_{r}(y)\right\vert ^{p}& =\left( \frac{1}{\mu \left(
B_{1}\right) \mu \left( B_{2}\right) }\int_{B_{1}}\int_{B_{2}}(u(\xi
)-u(\eta ))\dif\mu (\eta )\dif\mu (\xi )\right) ^{p}  \notag \\
& \leq \frac{1}{\mu \left( B_{1}\right) \mu \left( B_{2}\right) }%
\int_{B_{1}}\int_{B_{2}}|u(\xi )-u(\eta )|^{p}\dif\mu (\eta )\dif\mu (\xi )
\notag \\
& \leq C_{1}\psi (r)^{-2}\int_{K}\left[ \int_{B(\xi ,3r)}|u(\xi )-u(\eta
)|^{p}\dif\mu (\eta )\right] \dif\mu (\xi )  \notag \\
& \leq C_{2}\psi (r)^{-1}\phi (r)^{\frac{\beta }{\beta^{\ast }}%
}\sup_{r\in (0,3d(x,y)]}\Phi _{u}^{\beta }(r),
\end{align*}
thus showing \eqref{20-3}.

Next, let $L$ be the set of $\mu$-Lebesgue points of $u$, and fix $x\in L$.
Define $r_{0}=r$, $r_{k}={\rho_k} r$, we have $%
r_{k}+r_{k+1}= r_{k}+{l_{k+1}^{-1}}r_{k}<2r_{k}$ for $k\geq 0$. Similar to the above arguments for \eqref{20-3}, we have for any $u\in B_{p,\infty }^{\beta
}$,
\begin{equation*}
\left\vert u_{r_{k}}(x)-u_{r_{k+1}}(x)\right\vert ^{p}\leq C_{3}\psi
(r_{k})^{-1}\phi (r_{k})^{\frac{\beta }{\beta^{\ast }}}\sup_{r\in
(0,3d(x,y)]}\Phi _{u}^{\beta }(r).
\end{equation*}
Therefore, \begin{align}
 \left\vert u(x)-u_{r}(x)\right\vert & \leq \sum_{k=0}^{\infty }\left\vert
u_{r_{k}}(x)-u_{r_{k+1}}(x)\right\vert  \notag \\
& \leq C_{3}\sum_{k=0}^{\infty }\phi (r_{k})^{\frac{\beta }{p\beta
_{p}^{\ast }}}\psi (r_{k})^{-1/p}\sup_{r\in (0,3d(x,y)]}\left( \Phi
_{u}^{\beta }(r)\right) ^{1/p}  \notag   \\
& \leq C_{4}\phi (r)^{\frac{\beta }{p\beta^{\ast }}}\psi
(r)^{-1/p}\sup_{r\in (0,3d(x,y)]}\left( \Phi _{u}^{\beta }(r)\right) ^{1/p}\ \text{(similar to \eqref{e.est1})}.\label{206}
\end{align}%
Similar inequality holds for $\left\vert u(y)-u_{r}(y)\right\vert $. Combining \eqref{20-3} and \eqref{206}, we have
\begin{align}
|u(x)-u(y)|& \leq |u(x)-u_{r}(x)|+|u_{r}(x)-u_{r}(y)|+|u_{r}(y)-u(y)|  \notag
\\
& \leq C_{5}\phi (r)^{\frac{\beta }{p\beta^{\ast }}}\psi
(r)^{-1/p}\sup_{r\in (0,3d(x,y)]}\left( \Phi _{u}^{\beta }(r)\right) ^{1/p}
\notag \\
& \leq C_{5}\phi (r)^{\frac{\beta }{p\beta^{\ast }}}\psi
(r)^{-1/p}[u]_{B_{p,\infty }^{\beta }}  \label{205}
\end{align}%
for all $x,$ $y\in L$.

Finally, as $\phi (r)^{\frac{\beta }{\beta^{\ast }}}\psi
(r)^{-1}\rightarrow 0$ as $r\rightarrow0$ by \eqref{d1} with $\delta=0$, we can
use the standard procedure as in \cite[Lemma 2.1]{GYZ23} to obtain a
continuous version $\widetilde{u}\in C(K)$ for any $u\in B_{p,\infty
}^{\beta }$ and the desired inequality \eqref{MS_eq}.
\end{proof}

\begin{proposition}
There exists $C>1$ such that for any $\beta\in[0,\infty)$ and any $u\in L^{p}(K,\mu)$,
\begin{equation}
C^{-1}\sum_{n=0}^{\infty }\Phi_{u}^{\beta}(\rho_{n})\leq \lbrack u]_{B_{p,p}^{\beta }}^{p}\leq C\sum_{n=0}^{\infty }\Phi_{u}^{\beta}(\rho_{n})
. \label{40}
\end{equation}
In particular, $B_{p,p }^{\beta }$ is continuously embedded into $B_{p,\infty }^{\beta }$, namely, there exists $C>0$ such that $\lbrack u]_{B_{p,\infty}^{\beta }}^{p}\leq C\lbrack u]_{B_{p,p}^{\beta }}^{p}$ for all $u\in L^{p}(K,\mu)$.
\end{proposition}
\begin{proof}
Splitting the integral domain $(0,2]$ into $(\rho _{n+1},\rho
_{n}]\ (n\geq 0$), we have by \eqref{e.comphi} that
\begin{align}
\lbrack u]_{B_{p,p}^{\beta }}^{p} &=\int_{0}^{2}\Phi _{u}^{\beta }(r)\frac{%
\dif r}{r}=\sum_{n=0}^{\infty }\int_{\rho _{n+1}}^{\rho _{n}}\Phi
_{u}^{\beta }(r)\frac{\dif r}{r}  \notag \\
&\leq C \sum_{n=0}^{\infty }\int_{\rho _{n+1}}^{\rho _{n}}\Phi_{u}^{\beta}(\rho_{n})  \frac{\dif r}{r}  \leq C \log\left({\sup\limits_{n\geq 1}l_n}\right)\sum_{n=0}^{\infty }\Phi_{u}^{\beta}(\rho_{n})  \label{41}
\end{align}
and
\begin{align}
\lbrack u]_{B_{p,p}^{\beta }}^{p} =\sum_{n=0}^{\infty }\int_{\rho
_{n+1}}^{\rho _{n}}\Phi _{u}^{\beta }(r)\frac{\dif r}{r} \geq C^{-1}\sum_{n=0}^{\infty }\int_{\rho _{n+1}}^{\rho _{n}}\Phi_{u}^{\beta}(\rho_{n+1}) \frac{\dif r}{r} \geq C^{-1}\log\left({\inf\limits_{n\geq 1}l_n}\right)\sum_{n=1}^{\infty }\Phi_{u}^{\beta}(\rho_{n}).  \label{42}
\end{align}

Since $K$ satisfies the chain condition (see \cite[Definition 3.4]{GHL03}), we can use the argument in \cite[%
Corollary 2.2]{Yan18} to obtain that, there
exists some constant $C(n)>0$ depending on $n\geq 1$ such that%
\begin{align*}
\Phi_{u}^{\beta}(\rho_{n}) \geq C(n)^{-1}\int_{K}\int_{K}|u(x)-u(y)|^{p}\dif\mu (y)\dif\mu (x)
= C(n)^{-1}\Phi_{u}^{\beta}(\rho_{0}) .
\end{align*}
It follows by \eqref{42} that
\begin{equation}
\lbrack u]_{B_{p,p}^{\beta }}^{p}\geq c\sum_{n=0}^{\infty }\Phi_{u}^{\beta}(\rho_{n})
\label{43}
\end{equation}%
for some constant $c>0$. We obtain \eqref{40} by \eqref{41} and \eqref{43}. By \eqref{40} and \eqref{e.comphi}, we see that there exists $C>0$ such that $\lbrack u]_{B_{p,\infty}^{\beta }}^{p}\leq C\lbrack u]_{B_{p,p}^{\beta }}^{p}$ and consequently $B_{p,p}^{\beta }\hookrightarrow B_{p,\infty}^{\beta}$.
\end{proof}

Combining Lemma \ref{lemmaMR} and the above continuous embedding $B_{p,p }^{\beta }\hookrightarrow B_{p,\infty}^{\beta}$, we immediately derive the following Morrey-Sobolev inequality for $B_{p,p }^{\beta }$.
\begin{corollary}
\label{prop2}For $\beta\in (\epsilon_{p}\beta^{\ast},\infty)$ and $u\in
B_{p,p }^{\beta }$, there exists a continuous version $\widetilde{u}\in
C(K)$ such that
\begin{equation*}
|\widetilde{u}(x)-\widetilde{u}(y)|^p \leq C\phi (d(x,y))^{\frac{\beta }{\beta^{\ast }}}\psi (d(x,y))^{-1}[u]_{B_{p,p}^{\beta }}^p
\end{equation*}%
for all $x,$ $y\in K$, where $C$ is a positive constant.
\end{corollary}

\begin{remark}
In view of Lemma \ref{lemmaMR} and Corollary \ref{prop2}, we always regard $B_{p,\infty}^{\beta}$ and $B_{p,p}^{\beta}$ as subsets of $C(K)$ whenever $\beta\in (\epsilon_{p}\beta^{\ast},\infty)$. That is, we represent every function $u$ in $B_{p,\infty}^{\beta},
B_{p,p}^{\beta}$ by its continuous version. In particular, for such $u$, the energies in Definitions \ref{penergy_1} and \ref{defE} are well-defined.
\end{remark}

\subsection{Norm equivalences and critical Besov exponent}\label{subs.noeq}
For any positive integer $m$, let $\mu _{m}$ be the Borel measure on $V_{m}$ given by
\begin{equation}\label{e.mu_n}
\mu _{m}:=\frac{1}{\# V_{m}}\sum_{a\in V_{m}}\delta _{a}.
\end{equation}
Technically, we will use the discrete measures $\mu_{m}$ to approximate $\mu$ and \textit{convert the estimates of the integrations in \eqref{e.belip} to the estimates of discrete sums}. To be precise, denote the \emph{ball-energy}
\begin{equation}
I_{m,n}(u):=\int_{K}\int_{B(x,\rho_{n})}|u(x)-u(y)|^{p}\dif \mu _{m}(y)\dif %
\mu _{m}(x).  \label{201}
\end{equation}
Since $\mu _{m}$ weak $\ast$-converges to $\mu $ when $m\rightarrow\infty$, we have the weak $\ast$-convergence
of $\mu _{m}\times \mu _{m}$ to $\mu\times \mu$. For any $u\in C(K)$, the
set of discontinuity points of $(x,y)\mapsto\mathds{1}_{B(x,\rho_{n})}|u(x)-u(y)|^{p} $ is $\mu\times \mu$-null. {By an argument similar to \cite[Remark 1]%
{GL20}, we have that for any $\beta\in (\epsilon_{p}\beta^{\ast},\infty)$ and
any $u\in C(K)$, the following limit exists
\begin{equation} \label{mu_m}
I_{\infty ,n}(u):=\lim_{m\rightarrow \infty }I_{m,n}(u)=\int_{K}\int_{B(x,\rho
_{n})}|u(x)-u(y)|^{p}\dif \mu(y)\dif \mu(x).
\end{equation}}
{
Indeed, the weak $\ast$-convergence
of $\mu _{m}\times \mu _{m}$ to $\mu\times \mu$ implies that $\int_{K\times K}f\dif (\mu_{m}\times \mu_{m})$ converges to $\int_{K\times K}f\dif (\mu \times \mu )$ for any measurable function $f$ as long as the set of discontinuous points of $f$ have $\mu$-measure $0$. For each $n\in\mathbb{N}$, we define \begin{equation}
f_{n}(x,y):=\mathds{1}_{B(x,\rho_{n})}|u(x)-u(y)|^{p},\ \forall x,y\in K.
\end{equation}
By Proposition \ref{l.meas} we know that $\liminf_{r\to 0}\frac{\log \mu(B(x,r))}{\log r}=:\alpha^{\prime}>1$ and therefore the Hausdorff dimension of any measurable set with positive $\mu$-measure is at least $\alpha^{\prime}>1$ (see \cite[Proposition 2.3]{Fal97}). Since every $\partial B(x,r)\subset \mathbb{R}^{2}$ has Hausdorff dimension no greater than $1$, we know that $\mu(\partial B(x,r))=0$ for all $x\in K$ and all $r>0$. Since $u$ is continuous, we know that the set of discontinuous points of $f_{n}$ has $\mu\times \mu$-measure $0$, showing \eqref{mu_m}.
}

With this notion, we  see that
\begin{equation}
\Phi_{u}^{\beta}(r)\asymp \phi (\rho _{n})^{-\frac{\beta }{\beta
_{p}^{\ast }}}\psi (\rho _{n})^{-1}I_{\infty ,n}(u)\ \  \text{for}\ \rho_{n+1}<r\leq\rho_{n}.  \label{In2}
\end{equation}

The following two lemmas compares the energies in Definitions \ref{def[} and \ref{defE}.
\begin{lemma}\label{lem6.2}
For any $\beta\in (\epsilon_{p}\beta^{\ast},\infty)$, there exists $C>0$ such that
\begin{equation}
I_{\infty ,n}(u)\leq C\phi (\rho _{n})^{\frac{\beta }{\beta
_{p}^{\ast }}}\psi (\rho _{n})\sup_{k\geq n}\mathcal{E}_{k}^{\beta }(u)\ \text{ for all $u\in C(K)$}.
\label{eq6.3}
\end{equation}
In other words,
\begin{equation}
\Phi_{u}^{\beta}(\rho_{n})\leq C\sup_{k\geq n}\mathcal{E}_{k}^{\beta }(u)\ \text{ for all $u\in C(K)$}.
\label{eq6.3+}
\end{equation}
\end{lemma}

\begin{proof}
We first estimate $I_{m,n}(u)$ for all integers $m>n\geq 0$. For any $x,y\in K$
with $d(x,y)\leq \rho _{n}$ and $x\in K_{w}$ for some $w\in W_{n}$, due to
the geometry of $K$, there exists a word $\widetilde{w}\in W_{n}$ (not necessarily
distinct from $w$) such that $y\in K_{\widetilde{w}}$ and $K_{\widetilde{w}}\cap
K_{w}\neq \emptyset $. Therefore,
\begin{align*}
I_{m,n}(u)&\leq  \sum_{\substack{w,\widetilde{w}\in W_{n}\\ K_{\widetilde{w}}\cap
K_{w}\neq \emptyset}}\int_{K_{w}}\int_{K_{\widetilde{w}%
}}|u(x)-u(y)|^{p}\dif\mu _{m}(y)\dif\mu _{m}(x) \\
&= \sum_{\substack{w,\widetilde{w}\in W_{n}\\ K_{\widetilde{w}}\cap
K_{w}\neq \emptyset}}\sum_{x\in K_{w}\cap V_{m}}\sum_{y\in K_{\widetilde{%
w}}\cap V_{m}}\frac{1}{\# V_{m}^{2}}|u(x)-u(y)|^{p}.
\end{align*}%
For every pair $(w,\widetilde{w})\in W_{n}\times W_{n}$ with $K_{\widetilde{w}}\cap
K_{w}\neq \emptyset$, there exists a
common vertex $z\in K_{w}\cap K_{\widetilde{w}}\cap V_n$ (as in the proof of Proposition \ref{p.coding}). Moreover, since $%
|u(x)-u(y)|^{p}\leq 2^{p-1}\left(|u(x)-u(z)|^{p}+|u(z)-u(y)|^{p}\right)$, we have
\begin{align*}
I_{m,n}(u)&\leq  2^{p-1}\sum_{w,\widetilde{w}\in W_{n}}\sum_{x\in K_{w}\cap
V_{m}}\sum_{y\in K_{\widetilde{w}}\cap V_{m}}\frac{1}{\# V_{m}^{2}}\left(
|u(x)-u(z)|^{p}+|u(z)-u(y)|^{p}\right) \\
&\leq  C_{1}\sum_{w\in W_n}\sum_{x\in K_{w}\cap V_{m}}\sum_{z\in K_{w}\cap V_n}\frac{%
\# K_{\widetilde{w}}\cap V_{m}}{\# V_{m}^{2}}|u(x)-u(z)|^{p} \\
&\leq  C_{1}\psi (\rho _{n})\psi (\rho _{m})\sum_{w\in W_n}\sum_{x\in K_{w}\cap
V_{m}}\sum_{z\in K_{w}\cap V_n}|u(x)-u(z)|^{p},
\end{align*}%
where we have used $\# V_{m}\asymp \psi (\rho _{m})^{-1}$ and $\# K_{\widetilde{w}}\cap V_{m}\asymp \psi (\rho _{m})^{-1}\psi (\rho _{n})$ in the third line.

Then we estimate $|u(x)-u(z)|^{p}$. For every $w\in W_{n}$, $x\in K_{w}\cap V_{m}$ and $z\in
K_{w}\cap V_n$, we pick (and fix) a decreasing sequence of cells $%
\{K_{w_{k}}\}_{k=n}^{m}$ such that $w_{k}\in W_k$ with $z\in K_{w_{n}}\cap
V_{n} $, $x\in K_{w_{m}}\cap V_{m}$. Then we obtain a sequence of vertices $%
\{z=x_{n},x_{n+1},\cdots ,x_{m}=x\}$ such that
\begin{equation}
 x_{k}\in K_{w_{k}}\cap V_{k} \ \text{for}\ \ k=n,\cdots ,m. \label{ds}
\end{equation}
By H\"{o}lder's inequality,
\begin{align*}
|u(z)-u(x)|^{p}&\leq \left( \sum_{k=n}^{m-1}\left( \psi (\rho
_{n})^{-1}\psi (\rho _{k})\right) ^{q/p}\right) ^{p/q}\left(
\sum_{k=n}^{m-1}\psi (\rho _{n})\psi (\rho
_{k})^{-1}|u(x_{k})-u(x_{k+1})|^{p}\right) \\
&\leq  C_{2}\sum_{k=n}^{m-1}\psi (\rho _{n})\psi (\rho
_{k})^{-1}|u(x_{k})-u(x_{k+1})|^{p}.
\end{align*}%
Note that the cardinality of $(x,z)\in (K_{w}\cap V_{m})\times (K_{w}\cap V_n)$
with $(s,t)=(x_{k},x_{k+1})$ is no greater than $C^{\prime }\psi (\rho
_{k})\psi (\rho _{m})^{-1}$ for some $C^{\prime }>0$, since the cardinality of $%
K_{w}\cap V_n$ is uniformly bounded and the cardinality of $K_{w}\cap V_{m}$ is
equivalent to the total number of level-$m$ cells located in the level-$k$ cells containing $x_{k}$. Hence,
\begin{align}
I_{m,n}(u)& \leq C_{3}\psi (\rho _{n})\psi (\rho
_{m})\sum_{w\in W_n}\sum_{\substack{x\in K_{w}\cap V_{m}\\
z\in K_{w}\cap V_n}}\sum_{k=n}^{m-1}\psi (\rho _{n})\psi (\rho
_{k})^{-1}|u(x_{k})-u(x_{k+1})|^{p}  \notag \\
& \leq C_{3}\psi (\rho _{n})\psi (\rho
_{m})\sum_{w\in W_n}\sum_{k=n}^{m-1}\sum_{\QATOP{w^{\prime }\in W_k}{K_{w^{\prime
}}\subset K_{w}}}\sum_{\substack{ (x,z)\in (K_{w}\cap V_{m})\times (K_{w}\cap V_n) \\
(s,t)\in (K_{w^{\prime }}\cap V_{k})\times (K_{w^{\prime
}}\cap V_{k+1})}}\psi (\rho _{n})\psi (\rho _{k})^{-1}|u(s)-u(t)|^{p}  \notag
\\
& \leq C_{4}\psi (\rho _{n})\psi (\rho
_{m})\sum_{w\in W_n}\sum_{k=n}^{m-1}\sum_{\QATOP{w^{\prime }\in W_k}{K_{w^{\prime
}}\subset K_{w}}}\sum_{s,t\in K_{w^{\prime }}\cap V_{k+1}}\psi (\rho
_{m})^{-1}\psi (\rho _{n})|u(s)-u(t)|^{p}  \notag \\
& =C_{4}\psi (\rho _{n})^{2}\sum_{k=n}^{m-1}\sum_{w\in W_n}\sum_{\QATOP{%
w^{\prime }\in W_k}{K_{w^{\prime }}\subset K_{w}}}\sum_{s,t\in K_{w^{\prime
}}\cap V_{k+1}}|u(s)-u(t)|^{p}  \notag \\
& \leq C_{5}\psi (\rho _{n})^{2}\sum_{k=n}^{m-1}\sum_{\substack{s,t\in V_{k+1}\\s\sim
t}}|u(s)-u(t)|^{p}  \notag \\
& \leq 2C_{5}\psi (\rho _{n})^{2}\sum_{k=n}^{m}\phi (\rho _{k})^{\frac{\beta
}{\beta^{\ast }}}\psi (\rho _{k})^{-1}\sup_{k\geq n}\mathcal{E}%
_{k}^{\beta }(u)\text{ \ (by \eqref{En})},  \label{lem1.1}
\end{align}%
where in the second inequality we have used the fact that $(s,t)=(x_{k},x_{k+1})\in (K_{w^{\prime }}\cap V_{k})\times (K_{w^{\prime
}}\cap V_{k+1})$ with $w^{\prime }=w_k$ by \eqref{ds}, and in the third inequality we use $(K_{w^{\prime }}\cap V_{k})\subset (K_{w^{\prime
}}\cap V_{k+1})$. Therefore, we obtain by \eqref{lem1.1} and \eqref{e.est1} with $\delta =0$ that
\begin{equation*}
I_{m,n}(u)\leq C_{6}\psi (\rho _{n})^{2}\phi (\rho _{n})^{\frac{\beta }{%
\beta^{\ast }}}\psi ^{-1}(\rho _{n})\sup_{k\geq n}\mathcal{E}%
_{k}^{\beta }(u)=C_{6}\phi (\rho _{n})^{\frac{\beta }{\beta^{\ast }}%
}\psi (\rho _{n})\sup_{k\geq n}\mathcal{E}_{k}^{\beta }(u).
\end{equation*}%
Letting $m\rightarrow \infty $, we complete the proof by \eqref{mu_m}.
\end{proof}

Denote the \emph{ring-energy} by
\begin{equation*}
I_{n}(u):=\int_{K}\int_{\{\rho_{n+1}\leq d(x,y)<\rho _{n}\}}|u(x)-u(y)|^{p}%
\dif \mu(y)\dif \mu(x).
\end{equation*}%

\begin{lemma}
\label{lem6.3} For $\beta\in (\epsilon_{p}\beta^{\ast},\infty)$ and $u\in C(K)$, we have
\begin{equation}
\mathcal{E}_{n}^{\beta }(u)\leq C\sup_{k\geq n}\Phi _{u}^{\beta }(\rho
_{k}).  \label{eq6.4}
\end{equation}
\end{lemma}

\begin{proof}
For $s,t\in V_{n}$ with $s\sim t$, there exists some $w\in W_n$ such that $s, t \in K_w\cap V_n$. Note that for any $x \in K_w$,
\begin{equation*}
  |u(s) - u(t) |^p \leq 2^{p - 1} (| u(s) - u(x)
|^p + | u(x) - u(t) |^p).
\end{equation*}
Integrating with respect to $x$
and dividing by $\mu (K_w)$, we have
\begin{align}
 \sum_{\substack{s,t\in V_{n}\\ s\sim t}}|u(s)-u(t)|^{p}&=\sum_{w\in W_n}\sum_{ s, t \in K_w\cap V_n} |u (s) - u (t) |^p  \notag \\
&\leq 2^{p-1} \sum_{w\in W_n}\sum_{ s, t \in K_w\cap V_n} \fint_{K_{w}}|u(s)-u(x)|^{p}+|u(x)-u(t)|^{p}\dif\mu (x) \notag \\
&\leq C_{1}\sum_{w\in W_n}\sum_{ s \in K_w\cap V_n}%
\fint_{K_{w}}|u(s)-u(x)|^{p}\dif\mu (x).  \label{lem3.0}
\end{align}%
For every $m\geq n+1$ and $s\in K_w\cap V_n$ with $w\in W_n$, we can find (and fix) a decreasing sequence of
cells $\{K_{w_{k}}\}_{k=n}^{m}$
such that $s\in \cap_{k=n}^{m}K_{w_{k}}$ with $w_{n}=w$, $w_{k}\in W_k$. Let
\begin{equation}
\delta \in \left( 0,{\inf\limits_{n\geq 1}\left( \frac{\beta }{\beta^{\ast }}\beta_{l_n}^{(p)}-\alpha _{l_n}\right)} \right),  \label{delta1}
\end{equation}%
{where $\beta_{l_n}^{(p)}$ is given in \eqref{bln}.} By H\"{o}lder's inequality, we have for all $x_{k}\in K_{w_{k}}$ that
\begin{align}
& |u(s)-u(x_{n})|^{p}  \notag \\
\leq\ & 2^{p-1}|u(s)-u(x_{m})|^{p}+2^{p-1}\left( \sum_{k=n}^{m-1}\left( \rho
_{n}^{-\delta }\rho _{k}^{\delta }\right) ^{q/p}\right) ^{p/q}\left(
\sum_{k=n}^{m-1}\rho _{n}^{\delta }\rho _{k}^{-\delta
}|u(x_{k})-u(x_{k+1})|^{p}\right)  \notag \\
\leq\ & 2^{p-1}|u(s)-u(x_{m})|^{p}+C_{2}\left( \sum_{k=n}^{m-1}\rho
_{n}^{\delta }\rho _{k}^{-\delta }|u(x_{k})-u(x_{k+1})|^{p}\right) .
\label{lem2.5}
\end{align}%
Integrating \eqref{lem2.5} with respect to $x_{k}\in K_{w_{k}}$ and dividing
by $\mu (K_{w_{k}})$ for all $n\leq k\leq m$ successively, then combining with \eqref{lem3.0}, we have for $m\geq n+1$ that
\begin{align}\label{lem3.2}
 \sum_{\substack{s,t\in V_{n}\\ s\sim t}}|u(s)-u(t)|^{p}&\leq C_{3}\bigg(\sum_{w\in W_n}\sum_{ s \in K_w\cap V_n}\fint_{K_{w_{m}}}|u(s)-u(x_{m})|^{p}\dif\mu (x_{m})   \notag \\
&\quad +\sum_{w\in W_n}\sum_{ s \in K_w\cap V_n}\sum_{k=n}^{m-1}\rho _{n}^{\delta }\rho _{k}^{-\delta }\fint_{K_{w_{k}}}%
\fint_{K_{w_{k+1}}}|u(x_{k})-u(x_{k+1})|^{p}\dif\mu (x_{k+1})\dif\mu (x_{k})\bigg)\notag\\
&:=C_{3}\left(J_{1}(n,m)+J_{2}(n,m)\right).
\end{align}

For $s, x_{m}\in K_{w_{m}}$,
\begin{align*}
|u(s)-u(x_{m})| \leq \osc_{K_{w_{m}}}u,
\end{align*}%
which implies
\begin{align}
J_{1}(n,m)&=\sum_{w\in W_n}\sum_{ s \in K_w\cap V_n}\frac{1}{\mu (K_{w_{m}})}%
\int_{K_{w_{m}}}|u(s)-u(x_{m})|^{p}\dif\mu (x_{m})  \notag \\
 &\leq  |V_n| \left(\osc_{K_{w_{m}}}u\right)^{p}
\rightarrow   0 \text{ as $m\rightarrow \infty $},  \label{lem3.3}
\end{align}
since $u$ is uniformly continuous on the compact set $K$.

On the other hand, for all $x_{k}\in K_{w_{k}},x_{k+1}\in K_{w_{k+1}}$, we
have $|x_{k}-x_{k+1}|\leq \rho _{k}$, thus
\begin{align}
J_{2}(n,m)&=\sum_{w\in W_n}\sum_{ s \in K_w\cap V_n}\sum_{k=n}^{m-1}\frac{\rho _{n}^{\delta }\rho
_{k}^{-\delta }}{\mu (K_{w_{k}})\mu (K_{w_{k+1}})}\int_{K_{w_{k}}}%
\int_{K_{w_{k+1}}}|u(x_{k})-u(x_{k+1})|^{p}\dif\mu (x_{k+1})\dif\mu (x_{k})
\notag \\
&\leq C_{4}\sum_{k=n}^{\infty}\rho _{n}^{\delta }\rho _{k}^{-\delta }\psi (\rho
_{k})^{-2}\int_{K}\int_{B(x,\rho _{k})}|u(x)-u(y)|^{p}\dif\mu (y)\dif\mu (x)
\label{E4} \\
&\leq  C_{4}\sum_{k=0}^{\infty }\rho _{n}^{\delta }\rho _{n+k}^{-\delta
}\psi (\rho _{n+k})^{-2}\cdot \sum_{l=k}^{\infty }I_{n+l}(u)
=C_{4}\sum_{l=0}^{\infty }\left( \sum_{k=0}^{l}\rho _{n}^{\delta }\rho
_{n+k}^{-\delta }\psi (\rho _{n+k})^{-2}\right) I_{n+l}(u)  \notag \\
&\leq C_{5}\sum_{k=0}^{\infty }\rho
_{n}^{\delta }\rho _{n+k}^{-\delta }\psi (\rho _{n+k})^{-2}I_{n+k}(u),
\label{eq4.34}
\end{align}
where in the last inequality we have used the fact that%
\begin{align*}
\sum_{k=0}^{l}\frac{\rho _{n+k}^{-\delta }\psi (\rho _{n+k})^{-2}}{\rho
_{n+l}^{-\delta }\psi (\rho _{n+l})^{-2}} &=\sum_{k=0}^{l}\left( \frac{\rho
_{n+l}}{\rho _{n+k}}\right) ^{\delta }\left( \frac{\psi (\rho _{n+l})}{\psi
(\rho _{n+k})}\right) ^{2} \\
&\leq \sum_{k=0}^{l}\left({\inf\limits_{n\geq 1}l_n}\right)^{-\delta (l-k)}\left(2{\inf\limits_{n\geq 1}l_n}-1\right)^{-2(l-k)}<%
\left(1-\left({\inf\limits_{n\geq 1}l_n}\right)^{-(\delta +2)}\right)^{-1}.
\end{align*}

By the definition of $\mathcal{E}_{n}^{\beta }(u)$ and \eqref{lem3.2}, we have for $m\geq n+1$ that
\begin{align}
\mathcal{E}_{n}^{\beta }(u)&=\frac{1}{2}\phi (\rho _{n})^{-\frac{\beta }{\beta
_{p}^{\ast }}}\psi (\rho _{n})\sum_{x,y\in V_{n},x\sim y}|u(x)-u(y)|^{p}
\notag \\
&\leq C_{3}\phi (\rho _{n})^{-\frac{\beta }{\beta^{\ast }}}\psi (\rho
_{n})\left(J_{1}(n,m)+J_{2}(n,m)\right).
\label{E3}
\end{align}%
Letting $m\rightarrow\infty$, we see by \eqref{lem3.3} that
\begin{align}
\mathcal{E}_{n}^{\beta }(u)&\leq C_{3}\phi (\rho _{n})^{-\frac{\beta }{\beta^{\ast }}}\psi (\rho
_{n})\sup_{m\geq n+1}J_{2}(n,m)\label{E2}
\\ &\leq C_{6}\phi (\rho _{n})^{-\frac{\beta }{\beta^{\ast }}}\psi  (\rho_{n})\sum_{k=0}^{\infty }\rho _{n}^{\delta }\rho _{n+k}^{-\delta }\psi (\rho
_{n+k})^{-2}I_{n+k}(u) \text{ (by \eqref{eq4.34}) } \label{e.E9} \\
&\leq C_{7}\sum_{k=0}^{\infty }\left( \frac{\phi (\rho _{n})}{\phi (\rho
_{n+k})}\right) ^{-\frac{\beta }{\beta^{\ast }}}\left( \frac{\psi (\rho
_{n})}{\psi (\rho _{n+k})}\right) \left( \frac{\rho _{n}}{\rho _{n+k}}%
\right) ^{\delta }\sup_{k\geq n}\Phi _{u}^{\beta }(\rho _{k})  \notag \\
&\leq C_{8}\sup_{k\geq n}\Phi _{u}^{\beta }(\rho _{k})\text{ \ (by
\eqref{e.est1})},
\end{align}
where the third inequality used the fact that $I_{k+n}(u)\leq I_{\infty,k+n}(u)$ and \eqref{In2}. The proof is complete.
\end{proof}

The following lemma shows the equivalence between $[\cdot]_{B_{p,p}^{\beta }}^{p}$ and $\mathcal{E}_{p,p}^{\beta}$.
\begin{lemma}
\label{lem6-2} If $\beta\in (\epsilon_{p}\beta^{\ast},\infty)$, then there
exists $C>0$ such that for all $u\in C(K)$,
\begin{equation}
C^{-1}[u]_{B_{p,p}^{\beta }}^{p}\leq \mathcal{E}_{p,p}^{\beta }(u)\leq
C[u]_{B_{p,p}^{\beta }}^{p}.  \label{eq6.23}
\end{equation}
\end{lemma}

\begin{proof}
Using the fifth line in \eqref{lem1.1}, for $l\in \mathbb{Z}^{+}$, we
have
\begin{align}
& \sum_{n=0}^{\infty }\phi (\rho _{n})^{-\frac{\beta }{\beta^{\ast }}%
}\psi (\rho _{n})^{-1}I_{n+l,n}(u)  \notag \\
\leq\ & C_{1}\sum_{n=0}^{\infty }\phi (\rho _{n})^{-\frac{\beta }{\beta
^{\ast }}}\psi (\rho _{n})\sum_{k=n}^{n+l}\sum_{\substack{x,y\in V_{k}\\ x\sim y}}|u(x)-u(y)|^{p}  \notag \\
=\ & C_{1}\sum_{k=0}^{\infty }\sum_{n=0}^{k}\phi (\rho _{n})^{-\frac{\beta }{%
\beta^{\ast }}}\psi (\rho _{n})\sum_{\substack{x,y\in V_{k}\\ x\sim y}}|u(x)-u(y)|^{p}  \notag \\
\leq\ & C_{2}\sum_{k=0}^{\infty }\phi (\rho _{k})^{-\frac{\beta }{\beta^{\ast }}}\psi (\rho _{k})\sum_{\substack{x,y\in V_{k}\\x\sim y}}|u(x)-u(y)|^{p}\text{
(by using \eqref{e.est2} with }\delta =0 ) \\
=\ &2C_{2}\mathcal{E}_{p,p}^{\beta}(u).\label{43-1}
\end{align}%
Letting $l\rightarrow \infty $ and applying Fatou's lemma in \eqref{43-1},
the left-hand inequality of \eqref{eq6.23} then follows using \eqref{43}.

On the other hand, fix $\delta $ as required in \eqref{delta1}, then
\begin{align*}
 \mathcal{E}_{p,p}^{\beta }(u)&=\sum_{n=0}^{\infty }\mathcal{E}_{n}^{\beta
}(u) \leq C_{3}\sum_{n=0}^{\infty }\phi (\rho _{n})^{-\frac{\beta }{\beta^{\ast }}}\psi (\rho_{n})\sup_{m\geq n+1}J_{2}(n,m) \text{ (by \eqref{E2})}\\
& \leq C_{3}\sum_{n=0}^{\infty }\phi (\rho _{n})^{-\frac{\beta }{\beta
_{p}^{\ast }}}\psi (\rho _{n})\sum_{k=n}^{\infty}\rho _{n}^{\delta }\rho
_{k}^{-\delta }\psi (\rho _{k})^{-2}\int_{K}\int_{B(x,\rho
_{k})}|u(x)-u(y)|^{p}\dif\mu (y)\dif\mu (x) \text{ (by \eqref{E4})}\\
& \leq C_{4}\sum_{k=0}^{\infty }\left(\sum_{n=0}^{k}\frac{\phi(\rho
_{n})^{-\beta/\beta^{\ast }} \psi(\rho_{n})\rho_{n}^{\delta}}{\phi(\rho
_{k})^{-\beta/\beta^{\ast }} \psi(\rho_{k})\rho_{k}^{\delta}}\right)
\Phi_{u}^{\beta}(\rho_{k}) \text{ (by \eqref{In2})} \\
& \leq C_{5}\sum_{k=0}^{\infty }\Phi_{u}^{\beta}(\rho_{k})  \leq C_{5}[u]_{B_{p,p}^{\beta }}^{p}\text{ (by \eqref{e.est2} and \eqref{41}),}
\end{align*}%
showing the right-hand inequality of \eqref{eq6.23}.
\end{proof}

\begin{proof}[Proof of Theorem \ref{thm3} (1) and (2)]

\begin{enumerate}[label=\textup{(\arabic*)}]
\item[]\mbox{}\\
\item[(1)]Taking limsup and sup (of $n$) on both sides of \eqref{eq6.3+} and
\eqref{eq6.4} respectively, we have
	\begin{equation}
\limsup_{n\rightarrow \infty }\mathcal{E}_{n}^{\beta}(u)\asymp
\limsup_{n\rightarrow \infty }\Phi _{u}^{\beta}(\rho_{n})\text{ and } \mathcal{E}%
_{p,\infty }^{\beta}(u)=\sup_{n\rightarrow \infty }\mathcal{E}_{n}^{\beta}(u)\asymp \sup_{n\rightarrow \infty }\Phi _{u}^{\beta}(\rho_{n})=\lbrack u]_{B_{p,\infty}^{\beta}}^{p}.
\label{eq6.2}
\end{equation}
The assertion then follows by noting Lemma \ref{lem6-2} and \eqref{E5}.
\item[(2)]Given any $\beta >\beta^{\ast }$ and any $u\in B_{p,\infty}^{\beta }$, we first have
\[\sup_{r\in(0,2]}\Phi _{u}^{\beta^{\ast }}(r)=\sup_{r\in(0,2]} \phi(r)^{\frac{\beta}{\beta^{\ast}}-1}\Phi _{u}^{\beta}(r)\leq C\sup_{r\in(0,2]} \Phi _{u}^{\beta}(r)=C[u]^p_{B_{p,\infty}^{\beta }},\]
which implies $u\in B_{p,\infty}^{\beta^{\ast}}$. Since $\beta^{\ast }\in (\epsilon_{p}\beta^{\ast},\infty)$, it follows from \eqref{E5}, \eqref{eq6.2} and \eqref{e.comphi} that
\begin{equation}
\sup_{r\in(0,2]}\Phi _{u}^{\beta^{\ast }}(r)\asymp \sup_{n\geq 0}\Phi _{u}^{\beta^{\ast }}(\rho _{n})\asymp \sup_{n\rightarrow \infty }\mathcal{E}_{n}^{\beta^{\ast }}(u)=\limsup_{n\rightarrow \infty }\mathcal{E}_{n}^{\beta^{\ast }}(u)\asymp \limsup_{n\rightarrow \infty }\Phi _{u}^{\beta^{\ast }}(\rho _{n})\asymp \limsup_{r\rightarrow0}\Phi _{u}^{\beta^{\ast }}(r).
\label{e.coppp}
\end{equation}
Consequently,
\begin{align}
\sup_{r\in(0,2]}\Phi _{u}^{\beta^{\ast }}(r)\leq C\limsup_{r\rightarrow 0}\Phi _{u}^{\beta^{\ast }}(r)&=C\limsup_{r\rightarrow0}\phi(r)^{\frac{\beta}{\beta^{\ast}}-1}\Phi _{u}^{\beta}(r)\\ &\leq C\sup_{r\in(0,2]}\Phi _{u}^{\beta}(r)\cdot\limsup_{r\rightarrow0}\phi(r)^{\frac{\beta}{\beta^{\ast}}-1}\\ &=C [u]^p_{B_{p,\infty}^{\beta }}\cdot\limsup_{r\rightarrow0}\phi(r)^{\frac{\beta}{\beta^{\ast}}-1}=0.
\end{align}
Therefore $\Phi _{u}^{\beta^{\ast }}(r)\equiv0$ for all $r>0$ which implies that $u$ is constant by its continuity. On the other hand, Lemma \ref{l.affine} shows that $\mathcal{F}_{p}=B_{p,\infty }^{\beta^{\ast }}$ contains all piecewise affine functions, which can be chosen to be non-constant, proving \eqref{e.critical}.

\end{enumerate}
\end{proof}

\subsection{Weak monotonicity property and BBM convergence}\label{subs.wm}
We need the following key lemma that gives the upper bound of semi-norms by the lower limit of $\Phi _{u}^{\beta^{\ast }}(r)$.

\begin{lemma}
\label{thm_3}There exists $C>0$ such that for all $u\in \mathcal{F}_{p}= B_{p,\infty}^{\beta^{\ast}}$,
\begin{align}
\mathcal{E}_{p}(u)=\mathcal{E}_{p,\infty}^{\beta^{\ast}}(u)=\lim_{n\rightarrow \infty }\mathcal{E}_{n}^{\beta^{\ast } }(u)\leq
C\liminf_{n\rightarrow \infty }\Phi _{u}^{\beta^{\ast } }(\rho_{n}).
\label{limphi2}
\end{align}
\end{lemma}

\begin{proof}
The proof is based on the monotonicity of discrete $p$-energy norms.
We first have
\begin{align}
I_{n}(u)&\leq I_{\infty,n}(u)\\
& \leq C_{1}\phi (\rho _{n})\psi (\rho _{n})\sup_{k\geq n}\mathcal{E}_{k}^{\beta^{\ast }}(u)  \text{ (by \eqref{eq6.3})}\notag \\
& =C_{1}\phi (\rho _{n})\psi (\rho _{n})\lim_{n\rightarrow \infty }%
\mathcal{E}_{n}^{\beta^{\ast }}(u) \text{ (by \eqref{ve1})}.  \label{I_n}
\end{align}%
Since $\beta^{\ast }\in (\epsilon_{p}\beta^{\ast},\infty)$ and $\beta_{l}^{(p)}-\alpha_{l}=p-1$, we fix $%
\delta \in (0,p-1)$ as in \eqref{delta} with $\beta=\beta^{\ast }$. Using the fact that $\phi(\rho_{n})\psi(\rho_{n})^{-1}=\rho_{n}^{p-1}$
and \eqref{I_n}, we see that for any $L\geq1$,
\begin{align*}
\sum_{k=L+1}^{\infty }\psi (\rho _{n+k})^{-2}\rho _{n+k}^{-\delta
}I_{n+k}(u)
&\leq C_{1}\sum_{k=L+1}^{\infty }\phi (\rho _{n+k})\psi (\rho _{n+k})^{-1}\rho
_{n+k}^{-\delta }\lim_{n\rightarrow \infty }\mathcal{E}_{n}^{\beta
^{\ast }}(u) \\
&= C_{1}\sum_{k=L+1}^{\infty }\rho
_{n+k}^{p-1-\delta }\lim_{n\rightarrow \infty }\mathcal{E}_{n}^{\beta
^{\ast }}(u)\leq C_{2}\rho
_{n+L+1}^{p-1-\delta }\lim_{n\rightarrow \infty }\mathcal{E}_{n}^{\beta
^{\ast }}(u)\text{,}
\end{align*}
where in the last inequality we have used \eqref{e.est1} with $%
\beta =\beta^{\ast }$. Therefore,%
\begin{align}
\rho _{n}^{1-p+\delta}\sum_{k=L+1}^{\infty }\psi (\rho
_{n+k})^{-2}\rho _{n+k}^{-\delta }I_{n+k}(u)
&\leq C_{2}\left( \frac{\rho _{n+L+1}}{\rho _{n}}\right) ^{p-1-\delta
}\liminf_{n\rightarrow \infty }\mathcal{E}_{n}^{\beta^{\ast }}(u)
\notag \\
&\leq C_{2}({\inf\limits_{n\geq 1}l_n})^{(-p+1+\delta )(L+1)}\liminf_{n\rightarrow \infty }%
\mathcal{E}_{n}^{\beta^{\ast }}(u).  \label{In1}
\end{align}

We know by \eqref{E2} that,
\begin{align*}
\mathcal{E}_{n}^{\beta^{\ast }}(u)&\leq  C_{3}\phi
(\rho _{n})^{-1}\psi (\rho _{n})\rho _{n}^{\delta}\sum_{k=0}^{\infty }\psi (\rho _{n+k})^{-2}\rho _{n+k}^{-\delta }I_{k+n}(u) \\
 &= C_{3}\rho _{n}^{1-p+\delta}\left(\sum_{k=0}^{L}+\sum_{k=L+1}^{\infty}\right)\psi (\rho _{n+k})^{-2}\rho _{n+k}^{-\delta }I_{k+n}(u)\\
& \leq C_{3}\rho _{n}^{1-p+\delta}\sum_{k=0}^{L}\psi (\rho
_{n+k})^{-2}\rho _{n+k}^{-\delta
}I_{n+k}(u)+C_{3}C_{2}({\inf\limits_{n\geq 1}l_n})^{(-p+1+\delta )(L+1)}\liminf_{n\rightarrow
\infty }\mathcal{E}_{n}^{\beta^{\ast }}(u),
\end{align*}%
where \eqref{In1} is used in the last inequality. Taking $%
\liminf_{n\rightarrow \infty }$ on the right-hand side above, we have
\begin{equation*}
C_{4}\liminf_{n\rightarrow \infty }\mathcal{E}_{n}^{\beta^{\ast
}}(u)\leq \liminf_{n\rightarrow \infty }\rho _{n}^{1-p+\delta}\sum_{k=0}^{L}\psi (\rho
_{n+k})^{-2}\rho _{n+k}^{-\delta
}I_{n+k}(u),
\end{equation*}%
where
\begin{equation*}
C_{4}=\frac{1}{C_{3}}-C_{2}\left({\inf\limits_{n\geq 1}l_n}\right)^{-(p-1-\delta )(L+1)}.
\end{equation*}%
Fix $L$ large enough such that $C_{4}>0$, then for every $0\leq k\leq L$, \[ \left(\frac{\psi (\rho _{n})}{\psi (\rho_{n+k})}\right) ^{2}\left( \frac{\rho _{n}}{\rho _{n+k}}\right)^{\delta}\leq \left(\frac{\psi (\rho _{n})}{\psi (\rho_{n+L})}\right) ^{2}\left( \frac{\rho _{n}}{\rho _{n+L}}\right)^{\delta}\leq 4^{L}\left({\sup\limits_{n\geq 1}l_n}\right)^{(2+\delta)L}\]
and therefore
\begin{align*}
C_{4}\liminf_{n\rightarrow \infty }\mathcal{E}_{n}^{\beta^{\ast
}}(u)& \leq \liminf_{n\rightarrow \infty }\left(\frac{\rho _{n}^{1-p}}{\psi (\rho _{n})^{2}}\sum_{k=0}^{L}\left( \frac{\psi (\rho _{n})}{\psi (\rho
_{n+k})}\right) ^{2}\left( \frac{\rho _{n}}{\rho _{n+k}}\right) ^{\delta}I_{n+k}(u)\right) \\
& \leq C_{5}\liminf_{n\rightarrow \infty }\phi (\rho _{n})^{-1}\psi (\rho
_{n})^{-1}\sum_{k=0}^{L}I_{n+k}(u)\\ & \leq C_{5}\liminf_{n\rightarrow \infty }\phi (\rho _{n})^{-1}\psi (\rho_{n})^{-1} I_{\infty,n}(u) \\
& \leq C_{6}\liminf_{n\rightarrow \infty
}\Phi _{u}^{\beta^{\ast }}(\rho _{n}) \ \ (\text{by}\ \eqref{In2}).
\end{align*}%
Thus \eqref{limphi2} holds with $C=C_{6}/C_{4}$.
\end{proof}

We are now ready to prove the weak monotonicity property and BBM convergence.
\begin{proof}[Proof of Theorem \ref{thm3} (3) and (4)]
\begin{enumerate}[label=\textup{(\arabic*)}]
\item[]\mbox{}\\
\item[(3)] For any $u\in B_{p,\infty}^{\beta^{\ast}}=\mathcal{F}_{p}$, we have
\begin{align}
\sup_{r\in (0,2]}\Phi _{u}^{\beta^{\ast }}(r)&\leq C_{1}\sup_{n\geq 0 }\mathcal{E%
}_{n}^{\beta^{\ast }}(u)~(\text{by}~\eqref{e.coppp})\\
&= C_{1}\liminf_{n\rightarrow \infty }\mathcal{E}_{n}^{\beta^{\ast }}(u)~(\text{%
	by}~\eqref{ve1}) \\
&\leq C_{2}\liminf_{r\rightarrow 0}\Phi _{u}^{\beta^{\ast }}(r)~(\text{by}~%
\eqref{limphi2}).
\end{align}
\item[(4)] We first show \eqref{e.bbmdis}. Fix a function $u \in B_{p,\infty}^{\beta^{\ast}}=\mathcal{F}_p$. By \eqref{condB_p3}, for any $\beta<\beta^{\ast }$,
\begin{align*}
(\beta^{\ast }-\beta )\mathcal{E}_{p,p}^{\beta }(u)&=(\beta^{\ast }-\beta ) \sum_{n=0}^{\infty }\phi
(\rho _{n})^{1-\frac{\beta }{\beta^{\ast }}}\mathcal{E}_{n}^{\beta^{\ast}}(u)\\  &\leq(\beta^{\ast }-\beta ) \sum_{n=0}^{\infty }\phi (\rho
_{n})^{1-\frac{\beta }{\beta^{\ast }}}\mathcal{E}_{p,\infty}^{\beta^{\ast}}(u) \\
&\leq  \beta^{\ast }\left( 1-\frac{\beta }{\beta^{\ast }}\right) \left(1-\left({\inf\limits_{n\geq 1}t_{l_n}}\right)^{- 1+\frac{\beta }{\beta^{\ast }}}\right)^{-1}\mathcal{E}_{p,\infty}^{\beta^{\ast}}(u),
\end{align*}%
where {$t_{l}=l^{p-1}(2l-1)$}. Therefore,
\begin{equation}
\limsup_{\beta \uparrow \beta^{\ast }}(\beta^{\ast }-\beta )%
\mathcal{E}_{p,p}^{\beta }(u)\leq \frac{\beta^{\ast }}{\log \left({\inf\limits_{n\geq 1}t_{l_n}}\right)}\mathcal{E}_{p,\infty }^{\beta^{\ast }}(u).  \label{eq5.4}
\end{equation}

On the other hand, for any $A<\mathcal{E}_{p,\infty}^{\beta^{\ast}}(u)$, there exists $N\geq 1$ such
that $\mathcal{E}_{n}^{\beta^{\ast}}(u)>A$ for all $n>N$. Similarly, by \eqref{condB_p3} again,
\begin{align*}
(\beta^{\ast }-\beta )  \sum_{n=0}^{\infty }\phi
(\rho _{n})^{1-\frac{\beta }{\beta^{\ast }}}\mathcal{E}_{n}^{\beta^{\ast}}(u) & >(\beta^{\ast }-\beta ) \sum_{n=N+1}^{\infty }\phi (\rho _{n})^{1-%
\frac{\beta }{\beta^{\ast }}}A \\
& \geq(\beta^{\ast }-\beta ) \left({%
1-\left({\sup\limits_{n\geq 1}t_{l_n}}\right)^{\frac{\beta }{\beta^{\ast }}-1}}\right)^{-1}\phi (\rho
_{N+1})^{1-\frac{\beta }{\beta^{\ast }}}A \\
& \geq \beta^{\ast }\left( 1-\frac{\beta }{\beta^{\ast }}\right) \frac{\left({\sup\limits_{n\geq 1}t_{l_n}}\right)^{(N+1)\left( \frac{\beta }{\beta^{\ast }}-1\right) }}{%
1-\left({\sup\limits_{n\geq 1}t_{l_n}}\right)^{\frac{\beta }{\beta^{\ast }}-1}}A.
\end{align*}%
Therefore,
\begin{equation*}
\liminf_{\beta \uparrow \beta^{\ast }}(\beta^{\ast }-\beta )%
\mathcal{E}_{p,p}^{\beta}(u)\geq \frac{\beta^{\ast }}{\log
\left({\sup\limits_{n\geq 1}t_{l_n}}\right)}A,
\end{equation*}%
for any $A<\mathcal{E}_{p,\infty}^{\beta^{\ast}}(u)$. Letting $A\rightarrow \mathcal{E}_{p,\infty}^{\beta^{\ast}}(u)$, we obtain
\begin{equation}
\liminf_{\beta \uparrow \beta^{\ast }}(\beta^{\ast }-\beta )%
\mathcal{E}_{p,p}^{\beta }(u)\geq \frac{\beta^{\ast }}{\log \left({\sup\limits_{n\geq 1}t_{l_n}}\right)}\mathcal{E}_{p,\infty }^{\beta^{\ast }}(u).  \label{eq5.6}
\end{equation}%
Then \eqref{e.bbmdis} follows by combining \eqref{eq5.6} with \eqref{eq5.4}. The convergence in \eqref{e.bbmbes} is a consequence of \eqref{e.bbmdis} and Theorem \ref{thm3} (1).
\end{enumerate}
\end{proof}
\begin{remark}
If we define the \textit{Korevaar-Schoen} norms by
\begin{align}
\|u\|_{\mathrm{KS}_{p,\infty }^{\beta}}:=\limsup_{r\rightarrow 0 }\Phi _{u}^{\beta}(r),
\end{align}
then its corresponding BBM convergence, that is, \eqref{e.bbmbes} with $\|\cdot\|_{\mathrm{KS}_{p,\infty }^{\beta^{\ast }}}$ in place of $\|\cdot\|_{{B}_{p,\infty }^{\beta^{\ast }}}$, also holds by \eqref{eq6.2} and \eqref{E5}.
\end{remark}

\section{Further discussions}\label{s.Discuss}

{We discuss some possible extensions of our results and assume that $\bm{l}=(l_{n})_{n\geq1}$ satisfies $\sup\limits_{n\geq1}l_{n}<\infty$.}

{
In a very recent paper \cite{Yan25b}, Yang shows that the \emph{$p$-resistance estimate} is equivalent to the conjunction of \emph{$p$-Poincar\'e inequality} and \emph{$p$-cutoff Sobolev inequality} under the so-called \emph{slow volume regular condition}. Yang also verifies the $2$-cutoff Sobolev inequality in \cite{Yan25a} for standard self-similar Vicsek sets. These conditions are also satisfied for scale-irregular Vicsek sets. To see this, we first construct a strictly increasing version of $\psi$, termed $\widetilde{\psi}$, as follows:
\begin{equation}
\widetilde{\psi}(r):=
\begin{dcases}
\frac{\psi(\rho_n) - \psi(\rho_{n+1})}{\rho_n - \rho_{n+1}} \cdot \left(r - \rho_{n+1}\right) + \psi(\rho_{n+1}), &\text{for}\ \rho_{n+1} < r \leq \rho_n \ (n\geq 0),\\
\qquad\qquad\qquad\frac{1}{2} r, &\text{for}\ r \geq 2;
\end{dcases}
\end{equation}
and we define
\begin{equation}
\widetilde{\phi}(r) = r^{p-1} \widetilde{\psi}(r).
\end{equation}
It is easy to check that $\widetilde{\psi}$ and $\widetilde{\phi}$ are strictly increasing, thus their inverses $\widetilde{\psi}^{-1}$ and $\widetilde{\phi}^{-1}$ exist. By a direct computation, we see that for all $0<r\leq 2$,
\begin{align}
 (2\sup_{n}l_n-1)^{-1} \psi(r) &\leq \widetilde{\psi}(r)\leq \psi(r), \\
 (\sup_{n}l_n)^{-p+1} (2\sup_{n}l_n-1)^{-1} \phi(r) &\leq \widetilde{\phi}(r)\leq \phi(r).
\end{align}
By Proposition \ref{l.meas}-(2) and the fact that
\begin{equation}
  \frac{\widetilde{\psi}(R)}{\widetilde{\psi}(r)}=\left(\frac{r}{R}\right)^{p-1} \frac{\widetilde{\phi}(R)}{\widetilde{\phi}(r)}\ \text{for any $0<r\leq R<2$},
\end{equation}
we know that the slow volume regular condition $\mathrm{SVR}(\widetilde{\psi},\widetilde{\phi})$ in \cite{Yan25b} holds. By Proposition
\ref{p.Res}, the resistance estimate $\mathrm{R}(\widetilde{\psi},\widetilde{\phi})$ in \cite{Yan25b} holds. An application of \cite[Theorem 2.3]{Yan25b} immediately gives
\begin{proposition}\label{p.PICS}
Assume that $\bm{l}=(l_{n})_{n\geq1}$ satisfies $\sup\limits_{n\geq1}l_{n}<\infty$. Then for any $p\in(1,\infty)$, the $p$-Poincar\'e inequality $\mathrm{PI}(\widetilde{\phi})$ and the $p$-cutoff Sobolev inequality $\mathrm{CS}(\widetilde{\phi})$ hold for $(\mathcal{E}_{p},\mathcal{F}_{p})$ on the scale-irregular Vicsek set $K^{\bm{l}}$.
\end{proposition}
}

{
In the special case $p = 2$, by combining Proposition \ref{p.PICS} with \cite[Theorem 1.2]{GHL15}, or by combining Proposition \ref{p.Res} with \cite[Theorem 3.1]{Kum04} (see also \cite[Theorem 15.10]{Kig12}), we have the sub-Gaussian type heat kernel estimate for the strongly local, regular Dirichlet form $(\mathcal{E}_{2}, \mathcal{F}_{2})$ on $L^{2}(K^{\bm{l}}, \mu)$, or equivalently, for the associated \emph{Hunt process}, which is a diffusion on $K^{\bm{l}}$. More precisely, there exist constants $C, c, c', \delta > 0$ and a jointly continuous heat kernel $p_{t}(x, y) $ satisfying the two-sided estimate
\begin{equation}\label{e.HKE}
\frac{c^{\prime}}{\mu\left(B(x,\widetilde{\phi}^{-1}(t))\right)}\mathds{1}_{\left\{d(x,y)\leq\delta\widetilde{\phi}^{-1}(t)\right\}}\leq p_{t}(x,y)\leq \frac{C}{\mu\left(B(x,\widetilde{\phi}^{-1}(t))\right)}\exp\left(-\frac{1}{2}t\Phi\left(c\frac{d(x,y)}{t}\right)\right),
\end{equation}
for any $x,y\in K^{\bm{l}}$ and any $t>0$, where $ \Phi(s):=\sup_{r>0}\left\{\dfrac{s}{r}-\dfrac{1}{\widetilde{\phi}(r)}\right\}$.
}

{Let us sketch some other possible extensions of some recent related works.
\begin{enumerate}[label=\textup{({\arabic*})},align=left,leftmargin=*,topsep=5pt,parsep=0pt,itemsep=2pt]
\item (\emph{Heat-kernel based $p$-energy norms}) By \cite[Lemma 1.3.4]{FOT11}, it is known that \begin{equation}
\mathcal{E}_{2}(u)=\lim_{t\to0^{+}}\frac{1}{2t}\int_{K^{\bm{l}}}\int_{K^{\bm{l}}} |u(x)-u(y)|^{2} p_{t}(x, y)\dif\mu(y)\dif\mu(x).
\end{equation} One can define an alternative $p$-energy functional $\widetilde{\mathcal{E}}_{p}^{\beta}$ based on the heat kernel associated with $(\mathcal{E}_{2}, \mathcal{F}_{2})$, given formally by
\begin{equation}\label{e.HKpenergy}
\widetilde{\mathcal{E}}_{p}^{\beta}(u):=\sup_{t\in(0,2^{p-1})}\frac{1}{t^{\beta/\beta^{\ast}}} \int_{K^{\bm{l}}}\int_{K^{\bm{l}}} |u(x)-u(y)|^{p} p_{t}(x, y)\dif\mu(y)\dif\mu(x)
\end{equation}
with domain consisting of all continuous functions $u$ for which $\widetilde{\mathcal{E}}_{p}^{\beta}(u) < \infty$. Using the cake-layer decomposition
\begin{equation}
K^{\bm{l}}=\bigcup_{n=1}^\infty \left(B\left(x,2^n\widetilde{\phi}^{-1}(t)\right)\setminus B\left(x,2^{n-1}\widetilde{\phi}^{-1}(t)\right)\right)\cup B\left(x,\widetilde{\phi}^{-1}(t)\right)
\end{equation} and the two-sided heat kernel bounds \eqref{e.HKE}, one can show that $\widetilde{\mathcal{E}}_{p}^{\beta}(u)\asymp \lbrack u]_{B_{p,\infty}^{\beta}}^{p}$. Thus $\widetilde{\mathcal{E}}_{p}^{\beta^*}$ is equivalent to the energy $\mathcal{E}_{p}$ introduced in Theorem \ref{t.energy} by Theorem \ref{thm3}-(1), in the sense that the two functionals have the same domain and comparable energy norms.
\item (\emph{Gradient estimates and Hodge structure}) Recall the operator $\partial$ in Proposition \ref{p.grad}. Based on the two-sided bounds of heat kernel \eqref{e.HKE}, we expect that, on scale-irregular Vicsek sets, gradient estimates for the associated heat kernel and the corresponding Hodge structure can be developed by extending the approach in \cite{BC24}.
\item (\emph{$p$-Laplacian and PDEs}) Moreover, as in \cite[Lemma 3.9]{BC24}, the ($2$-)Laplacian (or generator) $\Delta_{2}$ for $(\mathcal{E}_{2},\mathcal{F}_{2})$ can be expressed by $\partial$ and its formal adjoint $\partial^{*}$. For the nonlinear case $p\neq2$, the Fréchet differentiability of $\mathcal{E}_{p}$ (see \cite[Theorem 3.7]{KS25} and Remark \ref{r.Clarkson}) ensures that
\begin{equation}
\mathcal{E}_{p}(f;g):=\frac{1}{p}\frac{\dif}{\dif t}\mathcal{E}(f+tg)\bigg|_{t=0}\in\mathbb{R} \ \text{exists for all $f, g \in \mathcal{F}_p$}.
\end{equation}
This naturally suggests a definition of the \emph{$p$-Laplacian operator} $\Delta_{p}$ on scale-irregular Vicsek sets, following the framework of Strichartz and Wong \cite{SW04}. It is an intriguing question whether $\Delta_{p}$ can also be expressed in terms of by $\partial$ and $\partial^{*}$. Furthermore, a corresponding \emph{variational principle} for
$p$-energies could be established in analogy with \cite[Theorem 3.1]{SW04}. In particular, the presence of a gradient operator $\partial$ is expected to give finer analysis of $p$-harmonic functions and regularity properties of solutions to nonlinear PDEs. These investigations could be pursued within the function spaces established in Theorem \ref{t.energy}, Definition~\ref{def[}, and Definition~\ref{defE}, using methods adapted from the theory of nonlinear elliptic PDEs under the fractal setting.
\end{enumerate}
}

{
\textbf{Acknowledgement:}
The authors are very grateful to the anonymous referee for many helpful comments.
}

\section*{Declarations}
\begin{itemize}
\item \textbf{Funding:} The authors were supported by National Natural Science Foundation of China (12271282). Jin Gao was also supported by Zhejiang Provincial Natural Science Foundation of China (LQN25A010019). Zhenyu Yu was also supported by the Natural Science Foundation of Hunan Province, China (No.2025JJ60039).
\item \textbf{Conflict of interest:} The authors declare that they have no conflict of interest.
\item \textbf{Data availability:} Data sharing is not applicable to this article as no datasets were generated or analysed during
the current study. Our manuscript has no associated data.
\end{itemize}


\end{document}